\documentclass[11pt,a4paper]{amsart}
\usepackage[utf8]{inputenc}
\usepackage[english]{babel}
\usepackage{amsmath}
\usepackage{amsfonts}
\usepackage{amssymb}
\usepackage{graphicx}
\usepackage{stackengine}
\usepackage{fullpage}
\usepackage{hyperref}

\numberwithin{equation}{section}

\newtheorem{thmm}{Theorem}
\newtheorem{thmmm}{Theorem}
\newtheorem{conj}[thmmm]{Conjecture}
\newtheorem*{conja}{Conjecture}
\newtheorem{thm}{Theorem}[section]
\newtheorem{lemma}[thm]{Lemma}

\newtheorem{prop}[thm]{Proposition}
\newtheorem{cor}[thm]{Corollary}
\newtheorem{assume}{Assumption}
\theoremstyle{remark}
\newtheorem{rem}[thm]{Remark}

\newcommand{\ind}{\mathbf 1}
\newcommand{\C}{\mathbb C}
\newcommand{\E}{\mathbb E}
\newcommand{\N}{\mathbb N}
\newcommand{\R}{\mathbb R}
\newcommand{\Z}{\mathbb Z}
\newcommand{\FM}{\mathfrak M}
\newcommand{\LG}{\mathcal G}
\newcommand{\LL}{\mathcal L}
\newcommand{\lvl}{\mathtt C}
\newcommand{\LT}{\mathcal T}
\newcommand{\TT}{\mathtt T}
\newcommand{\hol}{\mathsf{hol}}
\newcommand{\ED}{\operatorname{ED}}
\newcommand{\Int}{\operatorname{Int}}
\newcommand{\Lift}{\operatorname{Lift}}
\newcommand{\bpi}{\boldsymbol{\pi}}
\newcommand{\Var}{\operatorname{Var}}
\newcommand{\Diff}{\mathtt D}
\newcommand{\bic}{\mathfrak b}
\newcommand{\Triv}{\operatorname{Triv}}
\newcommand{\sgn}{\operatorname{sign}}
\newcommand{\ax}{\mathsf A}
\newcommand{\tub}{\operatorname{Tube}}
\newcommand{\diam}{\operatorname{diam}}
\newcommand{\enc}{\operatorname{enc}}
\newcommand{\Loop}{\operatorname{Loop}}

\title{An equivalence between gauge-twisted and topologically conditioned scalar Gaussian free fields}

\author{Titus Lupu}
\address {CNRS and LPSM, UMR 8001,
Sorbonne Université,
4 place Jussieu,
75252 Paris cedex 05,
France}
\email
{titus.lupu@sorbonne-unversite.fr}

\begin{document}

\maketitle

\begin{abstract}
We study on the metric graphs two types of scalar Gaussian free fields (GFF),
the usual one and the one twisted by a $\{-1,1\}$-valued gauge field.
We show that the latter can be obtained, up to an additional deterministic transformation, 
by conditioning the first on a topological event.
This event is that all the sign clusters of the field should be trivial for the gauge field, that is to say should not contain cycles with holonomy $-1$.
We also express the probability of this topological event as a ratio of two determinants of Laplacians to the power $1/2$, the usual Laplacian and the gauge-twisted Laplacian.
As an example, this gives on annular planar domains the probability that no sign cluster of the metric graph GFF surrounds the inner hole of the domain.

Based on our result on the metric graph, and on previous works by Werner and Cai-Ding on the clusters of the metric graph GFF in high dimension,
we formulate an intensity doubling conjecture.
According to it, if the space dimension is high enough,
the cycles in the sign clusters of the metric graph GFF converge in the scaling limit to a Brownian loop soup of intensity parameter $\alpha=1 = 2\times \dfrac{1}{2}$,
which is the double of the intensity parameter appearing in isomorphism theorems.
\end{abstract}

\section{Introduction}
\label{Sec intro}

In this article we consider two types of scalar Gaussian free fields (GFF),
the usual one and the one twisted by a $\{-1,1\}$-valued gauge field, 
and observe that the second is essentially obtained from the first by conditioning on a topological (more precisely homotopical) event.

We will work on an abstract finite electrical network
$\LG=(V,E)$ endowed with conductances
$C(x,y)=C(y,x)>0$ for $\{x,y\}\in E$.
The set of vertices $V$ will be divided into two parts,
$V_{\rm int}$ and $V_{\partial}$,
with $V_{\rm int}$ being considered as the interior vertices, 
and $V_{\partial}$ as the boundary.
The discrete GFF $\phi$ with $0$ boundary conditions on $V_{\partial}$
is given by the distribution
\begin{equation}
\label{Eq Z}
\dfrac{1}{Z}
\exp
\Big(
-\dfrac{1}{2}
\sum_{\{ x,y\}\in E} C(x,y)(\varphi(y)-\varphi(x))^{2}
\Big)
\prod_{z\in V_{\rm int}} d\varphi(z).
\end{equation}

Further, in the language of the gauge theory,
we consider $\{ -1,1\}$ as our gauge group,
and take a gauge field $\sigma\in \{ -1,1\}^{E}$.
The $\sigma$-twisted GFF $\phi_{\sigma}$ with $0$ boundary conditions on $V_{\partial}$ has for distribution
\begin{equation}
\label{Eq Z sigma}
\dfrac{1}{Z_{\sigma}}
\exp
\Big(
-\dfrac{1}{2}
\sum_{\{ x,y\}\in E} C(x,y)(\sigma(x,y)\varphi(y)-\varphi(x))^{2}
\Big)
\prod_{z\in V_{\rm int}} d\varphi(z).
\end{equation}
The gauge field $\sigma$ corresponds to disorder operators in the language of the
Ising model \cite{KadanoffCeva71DisorderIsing}.

To see the relation between $\phi$ and $\phi_{\sigma}$
one has to look at the level of metric graphs.
The metric graph $\widetilde{\LG}$ associated to $\LG$
is obtained by replacing each discrete edge $\{x,y\}$
by a continuous line of length $C(x,y)^{-1}$ joining $x$ and $y$.
The discrete GFF $\phi$ has a natural extension $\tilde{\phi}$ to the metric graph
$\widetilde{\LG}$,
which is a continuous Gaussian random field satisfying a Markov property.
This extension has been introduced in \cite{Lupu2016Iso}.
The field $\tilde{\phi}$, unlike $\phi$,
is known to satisfy a certain number exact identities, including for instance the probabilities of crossings.
These relations have been explored in the articles
\cite{Lupu2016Iso,LupuWerner2016Levy,DrewitzPrevostRodriguez22PTRF,
DrewitzPrevostRodriguez21CriticalExp};
see Section \ref{Subsec metric} for details.
The twisted discrete GFF $\phi_{\sigma}$
also has a natural extension $\tilde{\phi}_{\sigma}$ to the metric graph 
$\widetilde{\LG}$.
It is introduced in this paper in Section \ref{Subsec subdivision}.
Unlike $\tilde{\phi}$, the field $\tilde{\phi}_{\sigma}$
is not continuous in general,
and has one discontinuity point inside each $e\in E$
for which $\sigma(e) = -1$.
However, the absolute value $\vert \tilde{\phi}_{\sigma}\vert$ is a continuous field on the whole $\widetilde{\LG}$,
since the discontinuities of $\tilde{\phi}_{\sigma}$
consist in switching to the opposite sign by keeping the same absolute value.
By taking a double cover of $\widetilde{\LG}$
induced by the gauge field $\sigma$, 
one can extend $\tilde{\phi}_{\sigma}$ to a continuous field on the double cover.
This is explained in Section \ref{Subsec cov metric}.

Let $\LT_{\sigma}$ denote the subset of continuous functions on 
$\widetilde{\LG}$, 
made of functions $f$ such that for every connected component $U$ of the non-zero set $\{ f\neq 0\}$, $U$ does not contain loops of holonomy $-1$ for $\sigma$.
For the notion of holonomy in this setting
(product of the values of $\sigma$ along the edges of the loop), 
we refer to Section \ref{Subsec gauge}.
But let us give an example.
Consider a planar annular domain (one hole), such as depicted on
Figures \ref{Fig non triv ann} and \ref{Fig ex ann}.
One can consider a gauge field $\sigma$ on this annular domain that gives a holonomy $-1$ to loops that turn an odd number of times around the inner hole,
and holonomy $1$ to other loops; see Figure \ref{Fig non triv ann}.
Then $f\in \LT_{\sigma}$ if and only if
no connected component of $\{ f\neq 0\}$ surrounds the inner hole;
see Figure \ref{Fig ex ann}.

It is easy to see that $\vert \tilde{\phi}_{\sigma}\vert \in \LT_{\sigma}$
a.s., and this is proved in Lemma \ref{Lem T sigma}
by relying on the extension of $\tilde{\phi}_{\sigma}$
to the double cover of $\widetilde{\LG}$.
Our main result is the following.

\begin{thmmm}
\label{Thm main pres}
Let be a gauge field $\sigma\in\{-1,1\}^{E}$.
Then
\begin{equation}
\label{Eq ratio partition func}
\mathbb{P}(\tilde{\phi}\in\LT_{\sigma})
= \dfrac{Z_{\sigma}}{Z},
\end{equation}
where $Z$ and $Z_{\sigma}$ are the partition functions appearing in
\eqref{Eq Z} and \eqref{Eq Z sigma}.
Moreover, conditionally on the event
$\{\tilde{\phi}\in\LT_{\sigma}\}$,
the field $\vert \tilde{\phi}\vert$ has the same distribution as
$\vert\tilde{\phi}_{\sigma}\vert$.
\end{thmmm}

To obtain the field $\tilde{\phi}_{\sigma}$,
rather than just the absolute value $\vert\tilde{\phi}_{\sigma}\vert$,
from the field $\tilde{\phi}$ conditioned on $\tilde{\phi}\in\LT_{\sigma}$,
one has to additionally apply a deterministic sign flipping procedure
across the discontinuity points.
This is explained in Corollary \ref{Cor bicolor tilde phi}.

The identity \eqref{Eq ratio partition func} is thus a newcomer to the
family of exact identities known to be satisfied by $\tilde{\phi}$.
We would like to emphasize that Theorem \ref{Thm main pres}
does not require at all the graph $\LG$ to be planar.
However, for planar graphs the subset $\LT_{\sigma}$
is simpler to interpret;
see Section \ref{Sec interpret} and in particular \ref{Subsec dim geq 3}.

\begin{figure}
\includegraphics[scale=0.48]{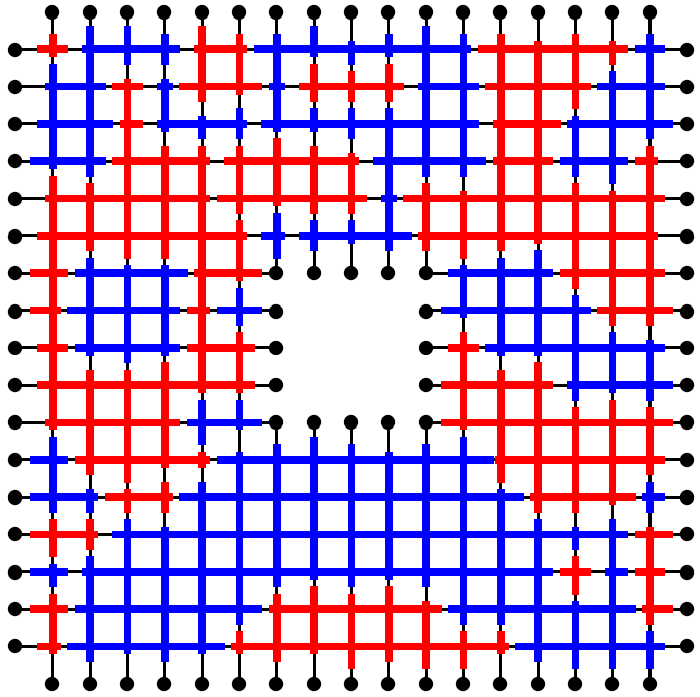}
\qquad
\includegraphics[scale=0.48]{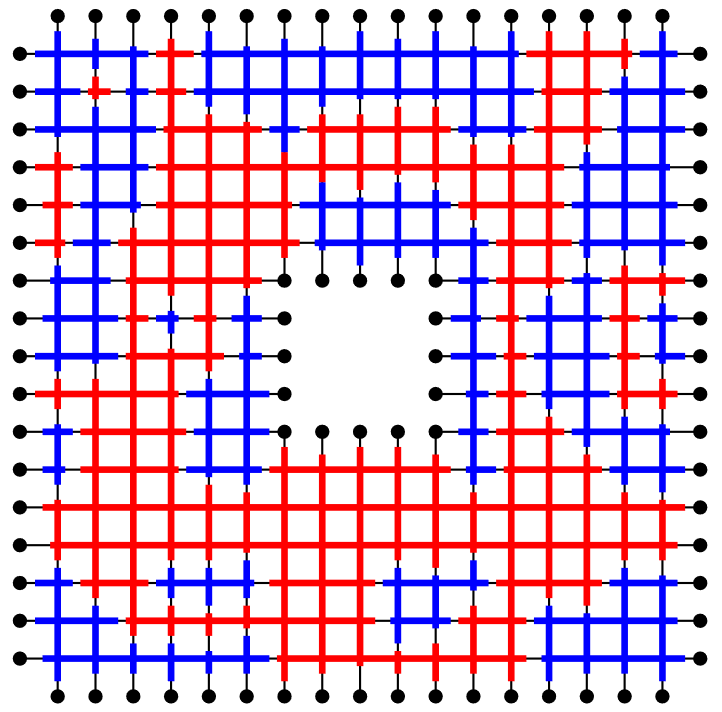}
\caption{
Conceptual depiction of the metric graph GFF $\tilde{\phi}$
on an annular domain.
The black dots represent the boundary
$V_{\partial}$.
The positive, resp. negative values of the fields are
in red, resp. blue.
Left: no sign cluster of $\tilde{\phi}$ surrounds the inner hole of the annulus.
Right: there is a sign cluster (in red) surrounding the inner hole.}
\label{Fig ex ann}
\end{figure}

We provide three different proofs for Theorem \ref{Thm main pres}.
The first proof is presented in Section \ref{Subsec topo}.
It proceeds entirely though Gaussian computations at the metric graph level,
and in particular does not rely on isomorphism identities between free fields and Markovian paths.
The second proof, presented in Section \ref{Subsec iso topo},
involves an isomorphism identity for $\tilde{\phi}_{\sigma}$
on the metric graph.
The third proof relies on the relation between the metric graph GFF and the FK-Ising random cluster model that has been observed in \cite{LupuWernerIsing}.
Actually, a result analogous to our Theorem \ref{Thm main pres} holds for the FK-Ising model (Theorem \ref{Thm disord Ising}).
The latter has been communicated to us by Marcin Lis (TU Wien, Vienna)
after the prepublication of the first version of this paper.
We are grateful to Marcin Lis for this.

We further derive some consequences of Theorem \ref{Thm main pres}
in the continuum limit in dimension 2.
First, Theorem \ref{Thm main pres} gives an expression for the probability that a certain CLE$_{4}$-type loop ensemble on an annular domain contains only contractible loops.
The same probability can be obtained by a different method, 
through the use of Schramm-Sheffield level lines of the 2D continuum GFF.
It is however non-obvious at a first glance that the two expressions obtained by two different methods are actually equal.
In Section \ref{Subsubsec Proba ann} we check that the two are indeed equal,
and this involves a remarkable computation with Jacobi Theta functions.
Then, in Section \ref{Subsubsec MS annulus} we derive a version of the Miller-Sheffield coupling, that involves on one hand the annular CLE$_{4}$
conditioned on contractibility,
and on the other hand the continuum gauge-twisted GFF.
We also describe how the renormalized hyperbolic cosine and the renormalized even powers of the 2D continuum GFF transform under topological conditioning. 

We also consider the case of higher dimensions, and in particular formulate an intensity doubling conjecture for high dimensions. Here we will state it only informally.
We refer to Section \ref{Sec conj high dim} for the precise statement and the heuristic behind.
\begin{conja}[Intensity doubling]
Assume that the space dimension is $d>8$.
Then the cycles in the sign clusters of the metric graph GFF are described in the scaling limit by a Brownian loop soup of intensity parameter 
$\alpha=1 = 2\times \dfrac{1}{2}$,
which is the double of the intensity parameter appearing in isomorphism theorems
(Theorems \ref{Thm Le Jan iso} and \ref{Thm Lupu iso}).
\end{conja}
The intensity doubling might also hold as soon as $d>6$.

\medskip

Let us mention some other works where the ratio
$Z_{\sigma}/Z$ as \eqref{Eq ratio partition func} appeared in a probabilistic context.
In \cite{BrugCamiaLis18winding} in a planar setting, the quantities of type $Z_{\sigma}/Z$ appeared as correlations of the Kramers-Wannier dual of the discrete GFF.
Again in a planar context in \cite{BenesLawlerViklund16LERW}, 
quantities of form $Z^{2}/Z_{\sigma}^{2}$
entered the expression of the probability that a loop-erased random walk goes through a particular edge.

Let us also mention that the effect of the gauge field has been already studied
on the fermionic side, the GFF being of course the Euclidean bosonic free field.
The Euclidean fermionic free field (without gauge field) is related to spanning trees and wired spanning forests; see \cite{FermionTree04}.
If one adds a gauge field $\sigma\in\{-1,1\}^{E}$,
then one gets the cycle-rooted spanning forests introduced by Kenyon \cite{Kenyon2011CRSF}.
These are spanning sub-graphs where a connected component can contain one cycle
of holonomy $-1$.
So, remarkably, while on the fermionic side the effect of a gauge field is to add
cycles of holonomy $-1$,
on the bosonic side the effect is to remove all possible cycles of holonomy $-1$.

\medskip

This article is organized as follows. 
Section \ref{Sec prelim} consists of preliminaries where we recall some background
that is maybe not common knowledge. 
In Section \ref{Subsec gauge} we recall the notions of gauge field,
gauge equivalence and holonomy in our particular setting
where the gauge group is $\{ -1,1\}$.
Section \ref{Subsec GFF gauge} deals with the discrete GFFs, 
the usual one and the gauge twisted.
Section \ref{Subsec paths} deals with random walk representations of these fields.
Section \ref{Subsec metric} recalls the method of the metric graph
and some results for the metric graph GFF.

In Section \ref{Sec results} we present the results of this article and
provide the corresponding proofs.
In Section \ref{Subsec subdivision} we introduce the natural
extrapolation $\tilde{\phi}_{\sigma}$ 
of the gauge-twisted GFF $\phi_{\sigma}$ to the metric graph.
In Section \ref{Subsec double cover} we consider the double cover of the discrete graph induced by the gauge field and observe that the usual discrete GFF
and the gauge twisted one are projections on two orthogonal subspaces
of the discrete GFF on the double cover.
In Section \ref{Subsec cov metric} we do the same at the level of the metric graph,
which in particular provides us a continuous extension of $\tilde{\phi}_{\sigma}$
to the double cover of the metric graph.
In Section \ref{Subsec topo} we give a more detailed statement of 
Theorem \ref{Thm main pres} and then prove it through direct Gaussian computations.
In Section \ref{Subsec iso topo} we provide an alternative proof of 
Theorem \ref{Thm main pres} through the metric graph version of the isomorphism identity between  $\tilde{\phi}_{\sigma}$ and Markovian loop soups 
due to Kassel and Lévy \cite{KasselLevy16CovSym}.
We also give a stronger version for this isomorphism identity.

In Section \ref{Sec rel Ising} we relate the topological conditioning for the metric graph GFF with the topological conditioning in FK-Ising random cluster model.
In Section \ref{Subsec intro Ising} we recall the spin Ising model, the FK-Ising model,
and the Edwards-Sokal coupling between the two. 
In Section \ref{Subsec disord Ising} we present the analogue of
Theorem \ref{Thm main pres} in the Ising setting.
In Section \ref{Subsec rel Ising GFF} we recall the relation between the Ising setting and the GFF setting.
In Section \ref{Subsec 3rd proof} we give the third proof of Theorem 
\ref{Thm main pres}, where it appears as a consequence of the analogous result for
FK-Ising.

Section \ref{Sec interpret} is a discussion around Theorem \ref{Thm main pres},
where we present some interpretations and implications.
In Section \ref{Subsec annular domain} we explain how metric graph annular domains naturally appear.
In Section \ref{Subsec annular continuum} we derive some consequences 
of Theorem \ref{Thm main pres} on 2D continuum annular domains
in the scaling limit.
In Section \ref{Subsec many holes} we briefly mention what Theorem \ref{Thm main pres}
provides for planar continuum domains with multiple holes, and what it does not.
In Section \ref{Subsec dim geq 3} we explain how the non-planar setting differs from the planar setting, and what Theorem \ref{Thm main pres} provides for the non-planar setting.

In Section \ref{Sec conj high dim} we present our intensity doubling conjecture for high dimensions and our reasoning behind it.

\section{Preliminaries}
\label{Sec prelim}

\subsection{On gauge fields, holonomy and gauge equivalence}
\label{Subsec gauge}

Let $\LG=(V,E)$ be a finite connected undirected graph.
We assume that there are no self-loops or multi-edges.
We also assume that the set of vertices consists of two disjoint parts,
$V=V_{\rm int}\cup V_{\partial}$,
$V_{\rm int}\cap V_{\partial}=\emptyset$,
with both $V_{\rm int}$ and $V_{\partial}$ non-empty.
We see $V_{\rm int}$ as the interior vertices and
$V_{\partial}$ as boundary vertices.
Each edge $\{x,y\}\in E$ is endowed with a conductance
$C(x,y)=C(y,x)>0$.
Thus, $\LG$ is an electrical network.

In this paper, a \textit{gauge field} will mean
a family $(\sigma(e))_{e\in E}\in\{ -1,1\}^{E}$.
This is the simplest case when the gauge group is $\{ -1,1\}$.
We will also use the notation
$\sigma(x,y) = \sigma(\{ x,y\})$,
when $\{x,y\}\in E$.
Given an other collection of signs
$(\hat{\sigma}(x))_{x\in V}\in \{ -1,1\}^{V}$,
this time above the vertices, 
it induces a \textit{gauge transformation}
$\sigma\mapsto \hat{\sigma}\cdot\sigma$,
where $\hat{\sigma}\cdot\sigma\in \{ -1,1\}^{E}$
is the gauge field defined by
\begin{displaymath}
(\hat{\sigma}\cdot\sigma)(x,y) = 
\hat{\sigma}(x)\sigma(x,y)\hat{\sigma}(y).
\end{displaymath}
Two gauge fields $\sigma,\sigma'\in \{ -1,1\}^{E}$
are said to be \textit{gauge-equivalent} if there is
$\hat{\sigma}\in \{ -1,1\}^{V}$ such that
$\sigma'=\hat{\sigma}\cdot\sigma$.
A gauge field $\sigma\in \{ -1,1\}^{E}$
is said \textit{trivial} if it is gauge-equivalent
to the gauge field with value $1$ on every edge.

Given $\wp=(x_{1},x_{2},\dots,x_{n})$ a discrete nearest-neighbor path
in $\LG$, 
the \textit{holonomy} of $\sigma\in\{ -1,1\}^{E}$ 
along $\wp$ is the product
\begin{displaymath}
\hol^{\sigma}(\wp) = 
\sigma(x_{1},x_{2})\sigma(x_{2},x_{3})
\dots \sigma(x_{n-1},x_{n}).
\end{displaymath}
If the nearest-neighbor path $\wp$ with finitely many jumps is parametrized by
continuous time, then $\hol^{\sigma}(\wp)$ is the holonomy along the
discrete skeleton of $\wp$.

\begin{lemma}
\label{Lem hol gauge equive}
\begin{enumerate}
\item Assume that two gauge fields $\sigma,\sigma'\in \{ -1,1\}^{E}$
are gauge-equivalent.
Then for every
loop (i.e. closed path) 
$\wp=(x_{1},x_{2},\dots ,x_{n-1},x_{n},x_{1})$ in $\LG$,
$\hol^{\sigma}(\wp)=\hol^{\sigma'}(\wp)$.
\item Conversely, assume that there is $x_{1}\in V$
such that for every
loop $\wp=(x_{1},x_{2},\dots ,x_{n-1},x_{n},x_{1})$ rooted in $x_{1}$,
$\hol^{\sigma}(\wp)=\hol^{\sigma'}(\wp)$.
Then $\sigma$ and $\sigma'$ are gauge equivalent.
\item In particular, a gauge field $\sigma\in \{ -1,1\}^{E}$ is
trivial if and only if for every loop $\wp$,
$\hol^{\sigma}(\wp)=1$.
\end{enumerate}
\end{lemma}

\begin{proof}
(1) If $\sigma' = \hat{\sigma}\cdot\sigma$, then
for every path $\wp=(x_{1},x_{2},\dots ,x_{n-1},x_{n},x_{n+1})$,
\begin{displaymath}
\hol^{\sigma'}(\wp) = \hat{\sigma}(x_{1})\hol^{\sigma}(\wp)\hat{\sigma}(x_{n+1}).
\end{displaymath}
In particular, if $x_{n+1}=x_{1}$, $\hol^{\sigma'}(\wp)=\hol^{\sigma}(\wp)$.

(2) For each $x\in V$, take a path $\wp^{x_{1},x}$ from $x_{1}$ to $x$
in $\LG$ and set
\begin{displaymath}
\hat{\sigma}(x)=\hol^{\sigma}(\wp^{x_{1},x})\hol^{\sigma'}(\wp^{x_{1},x}).
\end{displaymath}
The value of $\hat{\sigma}(x)$ does not depend on the particular choice of the
path $\wp^{x_{1},x}$.
Indeed, if $\bar{\wp}^{x_{1},x}$ is an other path from $x_{1}$ to $x$,
then one can concatenate $\bar{\wp}^{x_{1},x}$
with the time reversal $\overleftarrow{\wp^{x_{1},x}}$ of 
$\wp^{x_{1},x}$,
so as to get a loop $\bar{\wp}^{x_{1},x}\wedge \overleftarrow{\wp^{x_{1},x}}$
from $x_{1}$ to $x_{1}$, and then
\begin{displaymath}
\hol^{\sigma}(\wp^{x_{1},x})\hol^{\sigma}(\bar{\wp}^{x_{1},x})
=\hol^{\sigma}(\bar{\wp}^{x_{1},x}\wedge \overleftarrow{\wp^{x_{1},x}})
=\hol^{\sigma'}(\bar{\wp}^{x_{1},x}\wedge \overleftarrow{\wp^{x_{1},x}})
=\hol^{\sigma'}(\wp^{x_{1},x})\hol^{\sigma'}(\bar{\wp}^{x_{1},x}).
\end{displaymath}
With $\hat{\sigma}$ defined in this way, we have that
$\sigma'=\hat{\sigma}\cdot\sigma$.
Note that $\hat{\sigma}$ is not the only element of
$\{ -1,1\}^{V}$ to relate
$\sigma$ and $\sigma'$ through induced gauge transformation.
With $\LG$ being connected, there are exactly two such 
elements of $\{ -1,1\}^{V}$, $\hat{\sigma}$ and $-\hat{\sigma}$.
\end{proof}

Figures \ref{Fig triv ann}, \ref{Fig non triv ann} and \ref{Fig 2 holes}
provide examples of gauge-equivalent gauge fields together with the corresponding gauge transformation.

\begin{figure}
\includegraphics[scale=0.3]{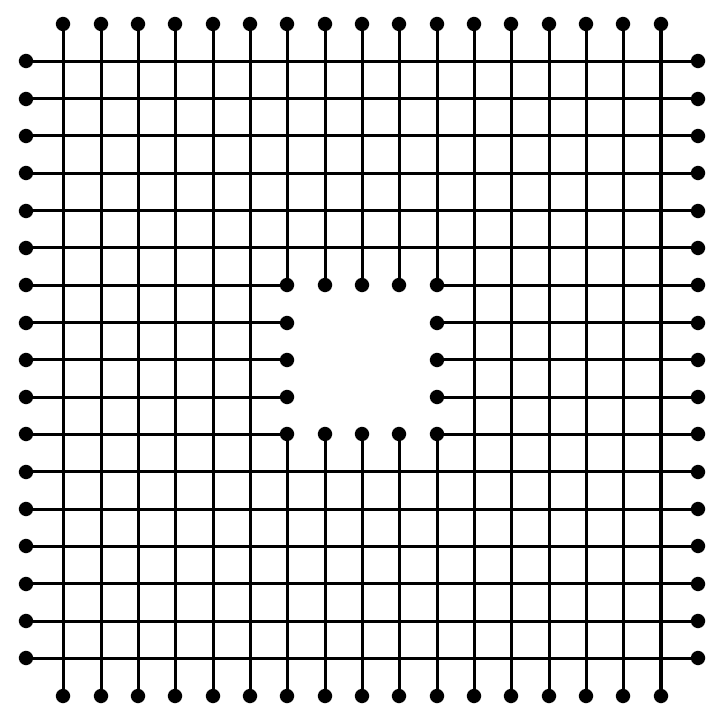}
\qquad
\includegraphics[scale=0.3]{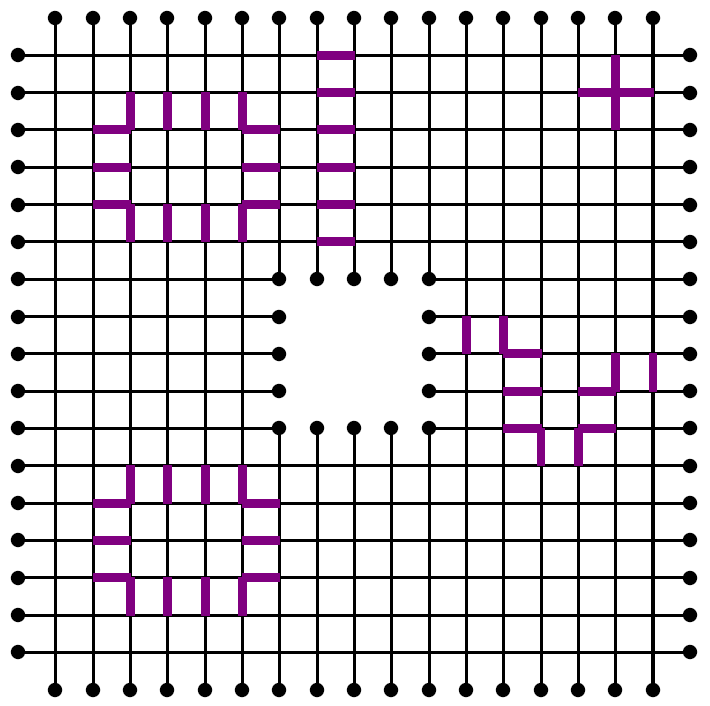}

\vspace{0.5cm}

\includegraphics[scale=0.3]{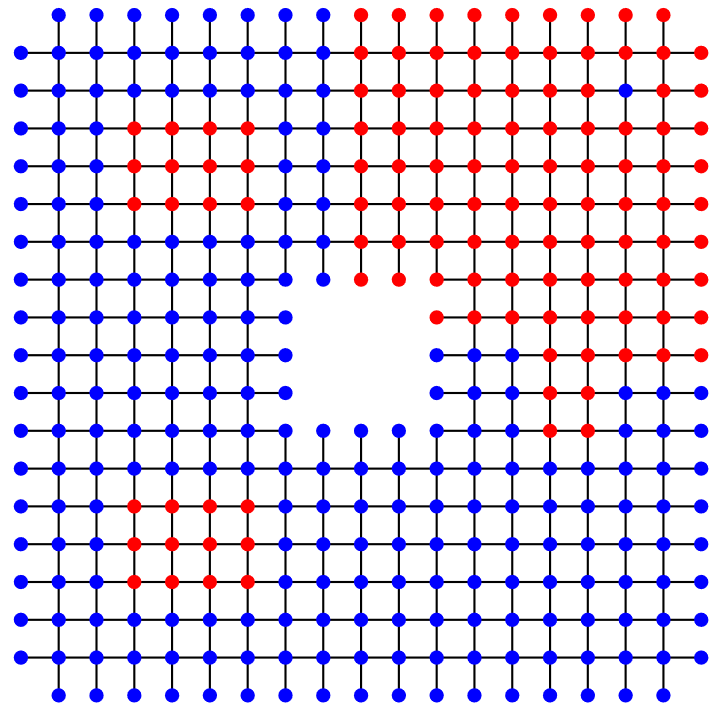}
\caption{Top: two trivial gauge fields. The edges with sign $-1$ are in violet.
Bottom: a gauge transformation relating the above gauge fields. In red are the $+1$ vertices and in blue the $-1$ vetices. The black dots represent the boundary
$V_{\partial}$.}
\label{Fig triv ann}
\end{figure}

\begin{figure}
\includegraphics[scale=0.3]{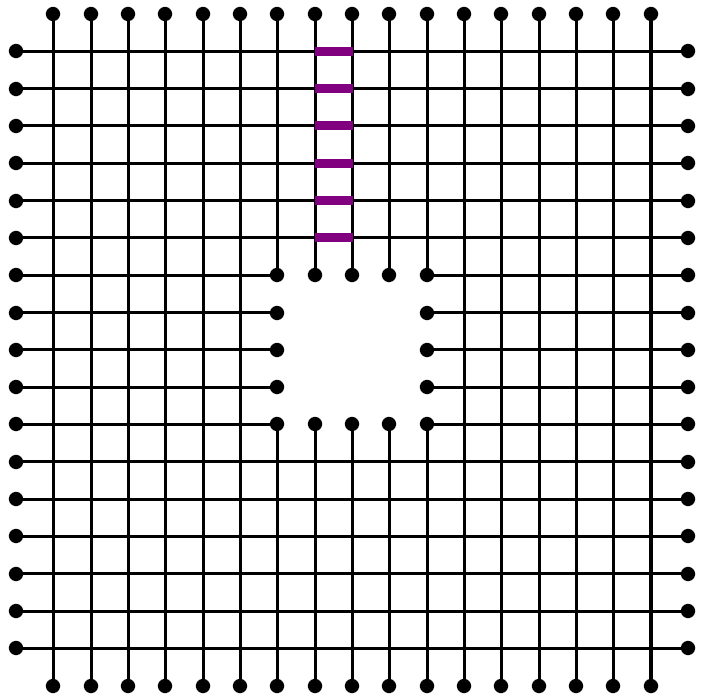}
\qquad
\includegraphics[scale=0.3]{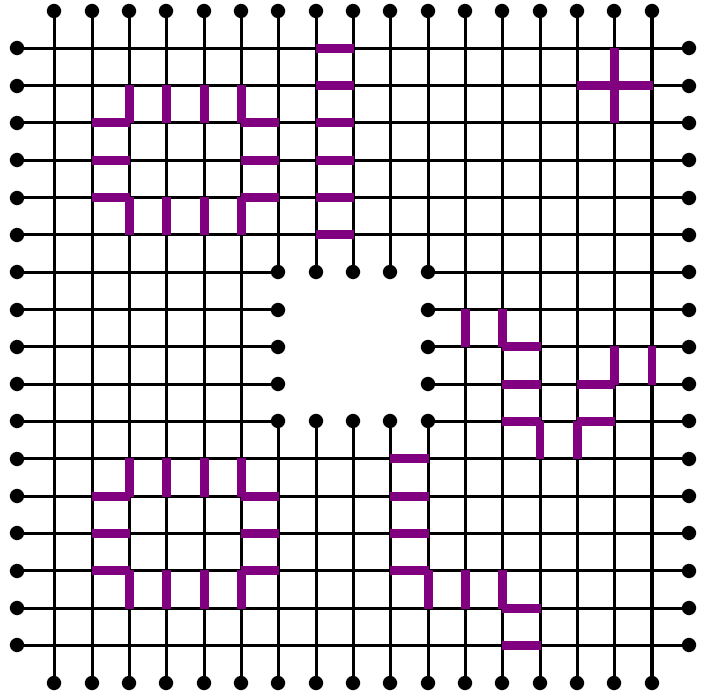}

\vspace{0.5cm}

\includegraphics[scale=0.32]{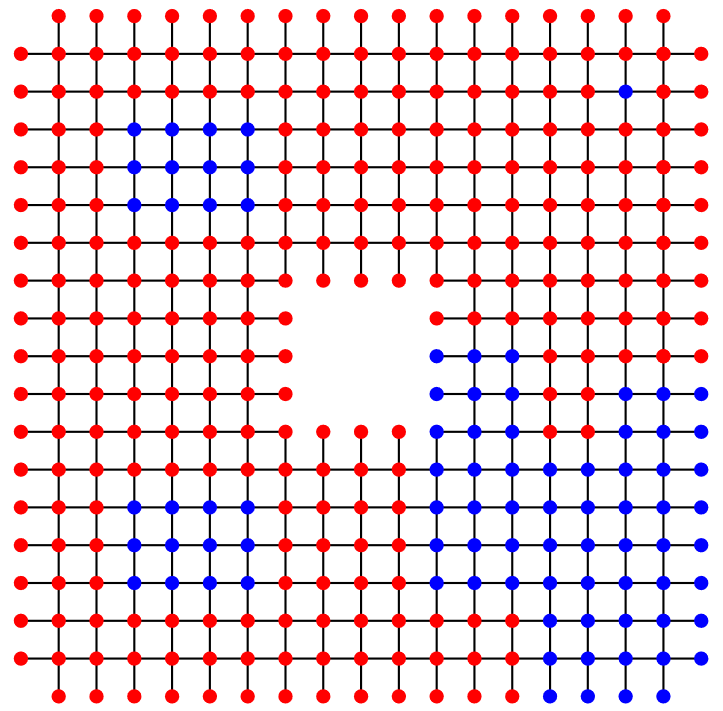}
\caption{Top: two non-trivial and gauge-equivalent gauge fields. 
The edges with sign $-1$ are in violet.
The holonomy of a loop is $-1$ to the power the number of turns it performs around the inner hole.
Bottom: a gauge transformation relating the above gauge fields. In red are the $+1$ vertices and in blue the $-1$ vetices.
The black dots represent the boundary $V_{\partial}$.}
\label{Fig non triv ann}
\end{figure}

\begin{figure}
\includegraphics[scale=0.3]{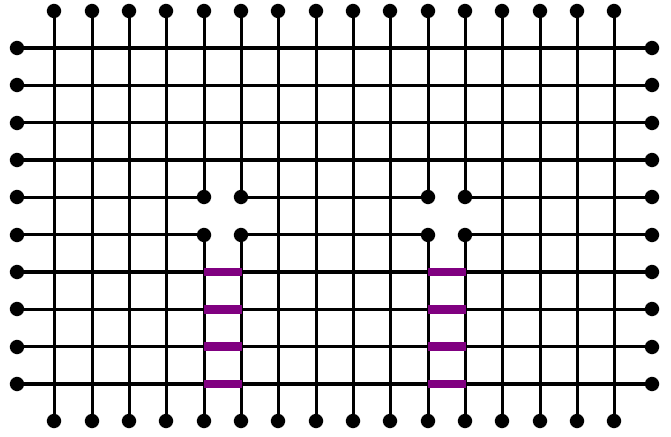}
\qquad
\includegraphics[scale=0.3]{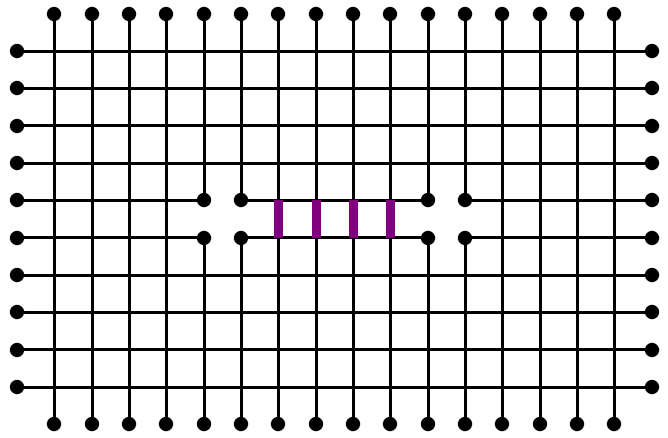}

\vspace{0.5cm}

\includegraphics[scale=0.32]{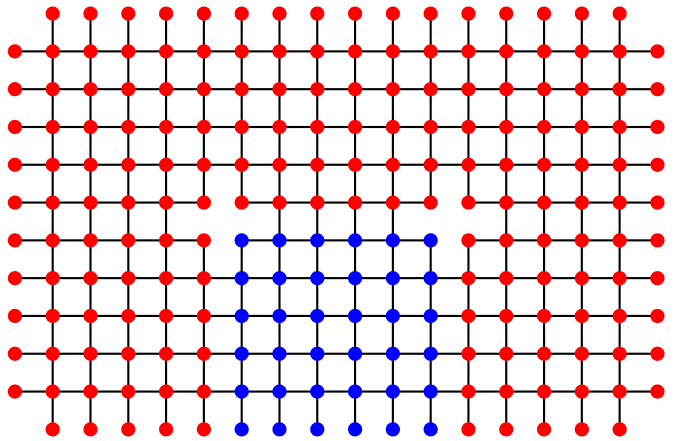}
\caption{Top: two non-trivial and gauge-equivalent gauge fields. 
The edges with sign $-1$ are in violet.
In the holonomy of a loop one multiplies by $-1$ each time one surrounds one of the two holes but not both.
Bottom: a gauge transformation relating the above gauge fields. In red are the $+1$ vertices and in blue the $-1$ vetices.
The black dots represent the boundary $V_{\partial}$.}
\label{Fig 2 holes}
\end{figure}

\subsection{Discrete scalar GFF twisted by a gauge field}
\label{Subsec GFF gauge}

Let $\Delta_{\LG}$ denote the discrete Laplacian on $\LG$:
\begin{displaymath}
(\Delta_{\LG} f)(x) = 
\sum_{y\sim x} C(x,y)(f(y)-f(x)),
\qquad x\in V,
\end{displaymath}
where $y\sim x$ means that $y$ is a neighbor of $x$,
i.e. $\{ x,y\}\in E$.
Let $(G(x,y))_{x,y\in V}$ be the Green's function of 
$-\Delta_{\LG}$ with $0$ boundary conditions on $V_{\partial}$.

Further, if $\sigma\in\{ -1,1\}^{E}$ is a gauge field,
the discrete Laplacian twisted by $\sigma$ is
\begin{displaymath}
(\Delta_{\LG,\sigma} f)(x) = 
\sum_{y\sim x} C(x,y)(\sigma(x,y)f(y)-f(x)),
\qquad x\in V.
\end{displaymath}
The twisted operator $-\Delta_{\LG,\sigma}$ is still non-negative
in the sense that for every $f\in\R^{V}$,
\begin{displaymath}
\mathcal{E}_{\sigma}(f,f) := 
\sum_{x\in V}(-\Delta_{\LG,\sigma} f)(x)f(x)=
\sum_{\{ x,y\}\in E} C(x,y)(\sigma(x,y)f(y)-f(x))^{2}\geq 0.
\end{displaymath}
Moreover, if $f$ is a function such that
$\mathcal{E}_{\sigma}(f,f)=0$ and $f$ is $0$ on $V_{\partial}$,
then $f$ is uniformly $0$ on $V$.
Indeed, one can see this by induction on the graph distance of a vertex $x\in V$
from $V_{\partial}$.
Therefore, one can consider the inverse of the restriction of the operator
$-\Delta_{\LG,\sigma}$ to the functions that are $0$ on $V_{\partial}$.
This is the gauge-twisted Green's function $(G_{\sigma}(x,y))_{x,y\in V}$
with $0$ boundary conditions on $V_{\partial}$.
It is defined by
\begin{displaymath}
\forall x\in V, \forall y\in V_{\partial},~~G_{\sigma}(x,y)=0,
\end{displaymath}
and
\begin{displaymath}
\forall x_{0}\in V_{\rm int}, \forall x\in  V_{\rm int},~~
\sum_{y\sim x} C(x,y)(\sigma(x,y)G(x_{0},y)-G(x_{0},x))
= -\ind_{x=x_{0}}.
\end{displaymath}
The function $G_{\sigma}$ is symmetric: 
$G_{\sigma}(x,y)=G_{\sigma}(y,x)$.
However, unlike for the usual Green's function $G(x,y)$,
some of the values $G_{\sigma}(x,y)$ may be negative.
Still, the operator induced by $G_{\sigma}$ is non-negative: 
for every $f\in\R^{V}$,
\begin{displaymath}
\sum_{x,y\in V} f(x) G_{\sigma}(x,y) f(y) \geq 0.
\end{displaymath}
Further, if $\sigma,\sigma'\in \{ -1,1\}^{E}$
are two gauge-equivalent gauge fields, with a gauge transformation between
$\sigma$ and $\sigma'$ induced by a $\hat{\sigma}\in \{ -1,1\}^{V}$,
then for every $x,y\in V$,
\begin{displaymath}
G_{\sigma'}(x,y) = \hat{\sigma}(x) G_{\sigma}(x,y) \hat{\sigma}(y).
\end{displaymath}

The discrete scalar \textit{Gaussian free field} (\textit{GFF})
on $\LG$ with boundary conditions $0$ on $V_{\partial}$ is
is the random Gaussian field $(\phi(x))_{x\in V}$,
with $\phi(x) = 0$ for $x\in V_{\partial}$, 
and with the probability distribution
\begin{equation}
\label{Eq phi}
\dfrac{1}{Z}
\exp
\Big(
-\dfrac{1}{2}
\sum_{\{ x,y\}\in E} C(x,y)(\varphi(y)-\varphi(x))^{2}
\Big)
\prod_{z\in V_{\rm int}} d\varphi(z).
\end{equation}
The expectation of $\phi$ is $0$.
The covariance function of $\phi$ is the Green's function
$(G(x,y))_{x,y\in V}$.

Further, given $\sigma\in\{ -1,1\}^{E}$ a gauge field,
the GFF twisted by $\sigma$ is the random Gaussian field 
$(\phi_{\sigma}(x))_{x\in V}$,
with $\phi_{\sigma}(x) = 0$ for $x\in V_{\partial}$, 
and with the probability distribution
\begin{equation}
\label{Eq phi sigma}
\dfrac{1}{Z_{\sigma}}
\exp
\Big(
-\dfrac{1}{2}
\sum_{\{ x,y\}\in E} C(x,y)(\sigma(x,y)\varphi(y)-\varphi(x))^{2}
\Big)
\prod_{z\in V_{\rm int}} d\varphi(z).
\end{equation}
The expectation of $\phi_{\sigma}$ is $0$.
Actually, $\phi_{\sigma}\stackrel{(\text{law})}{=}-\phi_{\sigma}$.
The covariance function of $\phi_{\sigma}$ is the gauge-twisted Green's function
$(G_{\sigma}(x,y))_{x,y\in V}$.
If $\sigma,\sigma'\in \{ -1,1\}^{E}$
are two gauge-equivalent gauge fields, with a gauge transformation between
$\sigma$ and $\sigma'$ induced by a $\hat{\sigma}\in \{ -1,1\}^{V}$,
then
\begin{equation}
\label{Eq cov GFF gauge transfo}
(\phi_{\sigma'}(x))_{x\in V}
\stackrel{(\text{law})}{=}
(\hat{\sigma}(x)\phi_{\sigma}(x))_{x\in V}.
\end{equation}
In particular, if the gauge field $\sigma$ is trivial, 
then the field $\phi_{\sigma}$ can be obtained via a deterministic transformation
from the usual GFF $\phi$.

\subsection{Measures on paths}
\label{Subsec paths}

Let $(X_{t})_{t\geq 0}$ be the nearest-neighbor Markov jump process on $\LG$
with the jump rates given by the conductances $C(x,y)$.
Let $T_{V_{\partial}}$ denote the first hitting time of $V_{\partial}$.
We will denote by $p(t,x,y)$ the transition probabilities of the killed process
$(X_{t})_{0\leq t\leq T_{V_{\partial}}}$, so that
\begin{displaymath}
\sum_{y\in V_{\rm int}}p(t,x,y) = \mathbb{P}_{x}(T_{V_{\partial}}> t).
\end{displaymath}
Since the process is symmetric, $p(t,x,y)=p(t,y,x)$.
Moreover,
\begin{displaymath}
\int_{0}^{+\infty} p(t,x,y) dt = G(x,y).
\end{displaymath}
Given $x,y\in V_{\rm int}$ and $t>0$,
let $\mathbb{P}^{x,y}_{t}$ denote the bridge probability measure from $x$ to $y$ of duration $t$, where one conditions on $T_{V_{\partial}}> t$.
The \textit{excursion measure} from $x$ to $y$ is
\begin{displaymath}
\mu^{x,y}(d\wp) = 
\int_{0}^{+\infty}\mathbb{P}^{x,y}_{t}(d\wp) p(t,x,y) dt.
\end{displaymath}
So $\mu^{x,y}$ is a measure on nearest-neighbor paths from $x$ to $y$,
parametrized by continuous time, of finite duration.
The total mass of $\mu^{x,y}$ is $G(x,y)$.
The measure $\mu^{y,x}$ is the image of $\mu^{x,y}$ by time-reversal.
The induced measure on discrete skeletons is
\begin{displaymath}
\mu^{x,y}(x\rightarrow x_{1}\rightarrow x_{2}\rightarrow\dots \rightarrow x_{n-1}
\rightarrow y) = 
\dfrac{C(x,x_{1}) C(x_{1},x_{2})\dots C(x_{n-2},x_{n-1}) C(x_{n-1},y)}
{W(x)W(x_{1}) W(x_{2}) \dots W(x_{n-1})W(y)},
\end{displaymath}
where
\begin{equation}
\label{Eq W sum}
W(z) = \sum_{w\sim z} C(z,w).
\end{equation}

Consider now $\sigma\in\{ -1,1\}^{E}$ a gauge field.
Kassel and Lévy showed in \cite[Theorem 5.1]{KasselLevy16CovSym} the following.

\begin{thm}[Kassel-Lévy, \cite{KasselLevy16CovSym}]
\label{Thm Green hol gauge}
For every $x,y\in V_{\rm int}$,
\begin{displaymath}
G_{\sigma}(x,y) = 
\int \hol^{\sigma}(\wp)
\mu^{x,y}(d\wp).
\end{displaymath}
In other words,
\begin{displaymath}
\dfrac{G_{\sigma}(x,y)}{G(x,y)} =
\mathbb{E}[\hol^{\sigma}(\wp^{x,y})],
\end{displaymath}
where $\wp^{x,y}$ is a random excursion from $x$ to $y$ distributed according to
the probability measure  $\mu^{x,y}/G(x,y)$.
In particular, for every $x\in V_{\rm int}$,
\begin{displaymath}
\int \hol^{\sigma}(\wp)
\mu^{x,x}(d\wp) = G_{\sigma}(x,x)>0.
\end{displaymath}
\end{thm}

The measure on continuous time random walk loops introduced 
by Le Jan \cite{LeJan2010LoopsRenorm,LeJan2011Loops} is
\begin{equation}
\label{Eq mu loop}
\mu^{\rm loop}(d\wp)
=
\sum_{x\in V_{\rm int}}
\int_{0}^{+\infty}\mathbb{P}^{x,x}_{t}(d\wp) p(t,x,x) \dfrac{dt}{t}
= \dfrac{1}{T(\wp)}\sum_{x\in V_{\rm int}} \mu^{x,x}(d\wp),
\end{equation}
where $T(\wp)$ denotes the duration of a path $\wp$.
There are two types of loops, those that visit at least two different vertices,
and those that stay in one vertex and do not perform jumps.
For $n\geq 2$ jumps, the measure induced on discrete skeletons is
\begin{displaymath}
\mu^{\rm loop}
(x_{1}\rightarrow x_{2}\rightarrow\dots \rightarrow x_{n-1}\rightarrow x_{n}
\rightarrow x_{1}) = 
\dfrac{1}{n}\dfrac{C(x_{1},x_{2})\dots C(x_{n-1},x_{n})C(x_{n},x_{1})}
{W(x_{1})W(x_{2})\dots W(x_{n-1})W(x_{n})}.
\end{displaymath}
So this is the same measure on discrete-time loops as the one that
appeared in \cite{LawlerFerreras07RWLoopSoup} and
\cite[Chapter 9]{LawlerLimic2010RW}.
The total mass of the loops that visit at least two vertices is
\begin{equation}
\label{Eq loop det}
\mu^{\rm loop}(\{\text{Loops that visit at least 2 vertices}\})
=
\log\Big(
\det((G(x,y))_{x,y\in V_{\rm int}})
\prod_{x\in V_{\rm int}} W(x)
\Big);
\end{equation}
see \cite[Equation (2.5)]{LeJan2011Loops}.
Besides the loops that visit at least two vertices,
$\mu^{\rm loop}$ also puts weight on loops that stay in one vertex and do not jump.
For every $x\in V_{\rm int}$, the induced measure on the duration of loops that only stay in $x$ is
\begin{displaymath}
e^{-t/G(x,x)} \dfrac{dt}{t}.
\end{displaymath}

Given $\sigma\in\{ -1,1\}^{E}$ a gauge field,
the (signed) measure on loops twisted by $\sigma$ is
\begin{displaymath}
\mu^{\rm loop}_{\sigma}(d\wp)
=
\hol^{\sigma}(\wp)
\mu^{\rm loop}(d\wp).
\end{displaymath}
The measure $\mu^{\rm loop}_{\sigma}$ is signed and its total variation measure is
$\mu^{\rm loop}$.
The measure $\mu^{\rm loop}_{\sigma}$ is the same for the whole
gauge-equivalence class of $\sigma$.
The signed measure $\mu^{\rm loop}_{\sigma}$ can be decomposed as
\begin{displaymath}
\mu^{\rm loop}_{\sigma} =\mu^{\rm loop}_{\sigma ,+} - 
\mu^{\rm loop}_{\sigma ,-},
\end{displaymath}
where $\mu^{\rm loop}_{\sigma ,+}$ is the restriction of
$\mu^{\rm loop}_{\sigma}$ to loops $\wp$ with
$\hol^{\sigma}(\wp)=1$, and 
$\mu^{\rm loop}_{\sigma ,-}$ is the restriction of
$\mu^{\rm loop}_{\sigma}$ to loops $\wp$ with
$\hol^{\sigma}(\wp)= -1$.
According to \cite[Corollary 3.7]{KasselLevy16CovSym},
a formula similar to \eqref{Eq loop det} also holds in the presence of a gauge field:
\begin{equation}
\label{Eq loop det sigma}
\mu^{\rm loop}_{\sigma}(\{\text{Loops that visit at least 2 vertices}\})
=
\log\Big(
\det((G_{\sigma}(x,y))_{x,y\in V_{\rm int}})
\prod_{x\in V_{\rm int}} W(x)
\Big).
\end{equation}
So in particular, one gets the following.

\begin{cor}
\label{Cor ratio det}
The following identity holds:
\begin{displaymath}
\dfrac{\det((G_{\sigma}(x,y))_{x,y\in V_{\rm int}})}
{\det((G(x,y))_{x,y\in V_{\rm int}})}
=
\exp\big(-2\mu^{\rm loop}(\{\text{Loops } \wp \text{ with } \hol^{\sigma}(\wp)=-1\})\big).
\end{displaymath}
\end{cor}

\begin{proof}
By combining the formulas \eqref{Eq loop det} and
\eqref{Eq loop det sigma}, one obtains
\begin{eqnarray*}
\log\Big(\dfrac{\det((G_{\sigma}(x,y))_{x,y\in V_{\rm int}})}
{\det((G(x,y))_{x,y\in V_{\rm int}})}\Big)
&=&
(\mu^{\rm loop}_{\sigma}-\mu^{\rm loop})
(\{\text{Loops that visit at least 2 vertices}\})
\\&=&
-2\mu^{\rm loop}_{\sigma ,-}(\{\text{Loops that visit at least 2 vertices}\}).
\end{eqnarray*}
Further, a loop $\wp$ with $\hol^{\sigma}(\wp)=-1$ has to visit at least two vertices. Thus,
\begin{displaymath}
\mu^{\rm loop}_{\sigma ,-}(\{\text{Loops that visit at least 2 vertices}\})
=
\mu^{\rm loop}(\{\text{Loops } \wp \text{ with } \hol^{\sigma}(\wp)=-1\}).
\qedhere
\end{displaymath}
\end{proof}

Given a path $\wp$ on $\LG$ parametrized by continuous time,
$\ell^{x}(\wp)$ will denote its total time spent at vertex $x$, $x\in V$.
Given $\alpha>0$,
$\LL^{\alpha}$ will denote the Poisson point process of loops with intensity
measure $\alpha \mu^{\rm loop}$.
This is the continuous time \textit{random walk loop soup}
\cite{LeJan2010LoopsRenorm,LeJan2011Loops}.
For $x\in V$, $\ell^{x}(\LL^{\alpha})$ will denote its \textit{occupation field} 
in $x$:
\begin{displaymath}
\ell^{x}(\LL^{\alpha}) = \sum_{\wp\in \LL^{\alpha}} \ell^{x}(\wp).
\end{displaymath}
For the particular value $\alpha=1/2$, this occupation field is Gaussian.
More precisely one has the Le Jan's isomorphism.

\begin{thm}[Le Jan,\cite{LeJan2010LoopsRenorm,LeJan2011Loops}]
\label{Thm Le Jan iso}
For $\alpha=1/2$, the following identity in law holds:
\begin{displaymath}
(\ell^{x}(\LL^{1/2}))_{x\in V}
\stackrel{(\text{law})}{=}
\Big(\dfrac{1}{2}\phi(x)^{2}\Big)_{x\in V},
\end{displaymath}
where $\phi$ is the GFF \eqref{Eq phi}.
\end{thm}

Now take $\sigma\in\{ -1,1\}^{E}$ a gauge field.
In \cite[Theorem 6.6]{KasselLevy16CovSym},
Kassel and Lévy proved the following extension of Le Jan's isomorphism.

\begin{thm}[Kassel-Lévy, \cite{KasselLevy16CovSym}]
\label{Thm KL Le Jan gauge}
Denote by $\LL^{1/2}_{\sigma ,+}$, resp. $\LL^{1/2}_{\sigma ,-}$,
a Poisson point process with intensity measure
$\frac{1}{2}\mu^{\rm loop}_{\sigma ,+}$, resp.
$\frac{1}{2}\mu^{\rm loop}_{\sigma ,-}$.
Recall that $\phi_{\sigma}$ denotes the gauge-twisted GFF
\eqref{Eq phi sigma}.
Take $\phi_{\sigma}$ and $\LL^{1/2}_{\sigma ,-}$
to be independent.
Then the following identity in law holds
\begin{equation}
\label{Eq KL Le Jan gauge}
(\ell^{x}(\LL^{1/2}_{\sigma ,+}))_{x\in V}
\stackrel{(\text{law})}{=}
\Big(\dfrac{1}{2}\phi_{\sigma}(x)^{2}
+\ell^{x}(\LL^{1/2}_{\sigma ,-})\Big)_{x\in V}.
\end{equation}
In particular, the field $\phi_{\sigma}^{2}$ is stochastically
dominated bu the field $\phi^{2}$.
\end{thm}

The identity \eqref{Eq KL Le Jan gauge} tells that in some sense
the field $\frac{1}{2}\phi_{\sigma}^{2}$ is distributed as the occupation
field of a Poisson point process with intensity measure
$\frac{1}{2}\mu^{\rm loop}_{\sigma}$.
However, the measure $\frac{1}{2}\mu^{\rm loop}_{\sigma}$ is signed,
unless the gauge field $\sigma$ is trivial,
and thus cannot be an intensity for a Poisson.
Actually, the loops $\wp$ with $\hol^{\sigma}(\wp)=-1$ have to be counted
negatively.

\subsection{GFF on metric graphs}
\label{Subsec metric}

Here we will briefly recall the notion of the GFF on metric graphs.
For details we refer to \cite{Lupu2016Iso}.

The \textit{metric graph} associated to the electrical network $\LG$,
which we will denote by $\widetilde{\LG}$,
is obtained by replacing each edge $e=\{x,y\}\in E$
by a continuous line segment $I_{e} = I_{\{x,y\}}$,
with endpoints $x$ and $y$.
Moreover, $\widetilde{\LG}$ is endowed with a metric,
by setting the length of $I_{e}$ to be
$C(x,y)^{-1}$
(i.e. the resistance, the inverse of the conductance),
and with the corresponding length measure, which we will denote by
$\tilde{m}$. So $\widetilde{\LG}$ is a continuous and connected metric space.

The discrete GFF $\phi$ on $\LG$ \eqref{Eq phi} can be interpolated to a continuous Gaussian field $\tilde{\phi}$.
The restriction of $\tilde{\phi}$ to the vertices $V$ is the discrete GFF $\phi$.
Conditionally on $\phi$, the values of $\tilde{\phi}$ along an edge-line
$I_{\{x,y\}}$ are distributed as a standard one-dimensional Brownian bridge
between $\phi(x)$ and $\phi(y)$ of length $C(x,y)^{-1}$,
with values on different edge-lines being independent.
The field $\tilde{\phi}$ is the \textit{metric graph GFF}.
It satisfies a strong Markov property when cutting not only along the vertices,
but also inside edge-lines;
see \cite[Section 3]{Lupu2016Iso}.
The covariance of $\tilde{\phi}$ is given by the extension of the Green's function
to $\widetilde{\LG}\times\widetilde{\LG}$,
which we will still denote $G(x,y)$.

The Markov jump process $(X_{t})_{t\geq 0}$ on $\LG$ can be extended to a continuous Markov diffusion process on $\LG$,
which we will denote by $(\widetilde{X}_{t})_{t\geq 0}$.
Inside an edge-line $I_{e}$, $\widetilde{X}_{t}$ behaves as a one-dimensional Brownian motion, and once the process reaches a vertex $x\in V$,
it performs Brownian excursions inside each adjacent edge-line.
See \cite[Section 2]{Lupu2016Iso} for details.
In order to fit with the isomorphism identities,
we normalize $\widetilde{X}_{t}$ so that inside every edge-line
$I_{e}$ it behaves like a Brownian motion with quadratic variation $2 dt$.
Just as a one-dimensional Brownian motion,
the process $(\widetilde{X}_{t})_{t\geq 0}$ admits a family of local times
$(\ell^{x}_{t}(\widetilde{X}))_{x\in \widetilde{\LG}, t\geq 0}$,
continuous in $(x,t)$, characterized by
\begin{displaymath}
\int_{0}^{t} f(\widetilde{X}_{s})\, ds =
\int_{\widetilde{\LG}} f(x) \ell^{x}_{t}(\widetilde{X})\,
\tilde{m}(dx).
\end{displaymath}
Consider the continuous additive functional $A(t)$ and its inverse
$A^{-1}(t)$:
\begin{equation}
\label{Eq CAF}
A(t) = \sum_{x\in V} \ell^{x}_{t}(\widetilde{X}),
\qquad 
A^{-1}(t) = \inf\{ s\geq 0 \vert A(s)>t\}.
\end{equation}
Then $(\widetilde{X}_{A^{-1}(t)})_{t\geq 0}$ is distributed as the Markov jump process $(X_{t})_{t\geq 0}$.

Let $\widetilde{T}_{V_{\partial}}$ denote the first time  $\widetilde{X}_{t}$
hits $V_{\partial}$.
By construction, $A(\widetilde{T}_{V_{\partial}}) = T_{V_{\partial}}$.
Considered the process 
$(\widetilde{X}_{t})_{0\leq t\leq \widetilde{T}_{V_{\partial}}}$
killed upon hitting $V_{\partial}$.
Let $\tilde{p}(t,x,y)$ be the transition densities of the killed process,
so that
\begin{displaymath}
\int_{\widetilde{\LG}}\tilde{p}(t,x,y)\,\tilde{m}(dy) =
\mathbb{P}_{x}(\widetilde{T}_{V_{\partial}}>t).
\end{displaymath}
For $x, y\in \widetilde{\LG}\setminus V_{\partial}$ and $t>0$,
let $\widetilde{\mathbb{P}}^{x,y}_{t}$ denote the bridge probability measure
from $x$ to $y$ in time $t$, where one conditions on 
$\widetilde{T}_{V_{\partial}}>t$.
The measure on loops on the metric graph is
\begin{equation}
\label{Eq mu metric}
\tilde{\mu}^{\rm loop}(d\wp) =
\int_{\widetilde{\LG}}
\int_{0}^{+\infty} \widetilde{\mathbb{P}}^{x,x}_{t}(d\wp)
\tilde{p}(t,x,x)\dfrac{dt}{t}\,\tilde{m}(dx).
\end{equation}
For $\alpha>0$, $\widetilde{\mathcal{L}}^{\alpha}$
will denote the Poisson point process of loops on $\widetilde{\LG}$
of intensity $\alpha\tilde{\mu}^{\rm loop}$.
This is the \textit{metric graph loop soup}.
The loops in $\widetilde{\mathcal{L}}^{\alpha}$ can be divided into two:
those that visit vertices in $V$ and those that only stay in the interior of
an edge-line.
If one takes the trace on $V$ of the loops in $\widetilde{\mathcal{L}}^{\alpha}$ 
that intersect $V$,
by applying the time change $A^{-1}$ \eqref{Eq CAF},
one gets the continuous time random walk loop soup 
$\mathcal{L}^{\alpha}$,
actually up to a rerooting of the loops.
See  \cite[Section 2]{Lupu2016Iso} for details.
Given $\wp\in\widetilde{\mathcal{L}}^{\alpha}$ and $x\in \widetilde{\LG}$,
$\ell^{x}(\wp)$ will denote the cumulative local time of $\wp$ in $x$.
The \textit{occupation field} of $\widetilde{\mathcal{L}}^{\alpha}$ is
\begin{displaymath}
\ell^{x}(\widetilde{\mathcal{L}}^{\alpha})=
\sum_{\wp \in \widetilde{\mathcal{L}}^{\alpha}}\ell^{x}(\wp).
\end{displaymath}
For $x\in V$, we have that
$\ell^{x}(\widetilde{\mathcal{L}}^{\alpha})=
\ell^{x}(\mathcal{L}^{\alpha})$.

Next we will consider the clusters of loops in $\widetilde{\mathcal{L}}^{\alpha}$.
Two loops $\wp,\wp'\in \widetilde{\mathcal{L}}^{\alpha}$ 
belong to the same \textit{cluster}
if they are connected by a finite chain of loops in 
$\widetilde{\mathcal{L}}^{\alpha}$.
We will also see a cluster $\mathcal{C}$ of $\widetilde{\mathcal{L}}^{\alpha}$
as a subset of $\widetilde{\LG}$ by taking the union of the ranges of loops
forming the cluster.
The clusters of $\widetilde{\mathcal{L}}^{\alpha}$ are exactly the connected components of
$\{x\in\widetilde{\LG}\vert \ell^{x}(\widetilde{\mathcal{L}}^{\alpha})>0\}$.
In \cite[Theorem 1]{Lupu2016Iso} is shown that for $\alpha=1/2$,
there is a one to one correspondence between the clusters of
$\widetilde{\mathcal{L}}^{1/2}$ and the sign components of $\tilde{\phi}$.

\begin{thm}[Lupu, \cite{Lupu2016Iso}]
\label{Thm Lupu iso}
Take $\widetilde{\mathcal{L}}^{1/2}$ a metric graph loop soup on 
$\widetilde{\LG}$ of parameter $\alpha=1/2$.
For each cluster $\mathcal{C}$ additionally sample a conditionally independent sign
$\zeta(\mathcal{C})$ with
\begin{displaymath}
\mathbb{P}(\zeta(\mathcal{C})=1\vert \widetilde{\mathcal{L}}^{1/2})=
\mathbb{P}(\zeta(\mathcal{C})=-1\vert \widetilde{\mathcal{L}}^{1/2})=1/2.
\end{displaymath}
For $x\in \widetilde{\LG}$ such that
$\ell^{x}(\widetilde{\mathcal{L}}^{1/2})>0$,
let $\mathcal{C}(x)$ denote the cluster of $\widetilde{\mathcal{L}}^{1/2}$
containing $x$.
Then the field
\begin{displaymath}
\big(\zeta(\mathcal{C}(x))\sqrt{2\ell^{x}(\widetilde{\mathcal{L}}^{1/2})}\big)
_{x\in \widetilde{\LG}}
\end{displaymath}
is distributed as a metric graph GFF $\tilde{\phi}$.
\end{thm}

The result above comes from an application of the Le Jan's isomorphism 
(Theorem \ref{Thm Le Jan iso})
at the metric graph level and the use of the intermediate value property of the
continuous field $\tilde{\phi}$.

Beside the exact correspondence between the sign components of $\tilde{\phi}$
and the clusters of $\widetilde{\mathcal{L}}^{1/2}$,
the metric graph GFF $\tilde{\phi}$ satisfies many other exact solvability
properties that one does not get for the discrete GFF $\phi$.
Next we give a brief list of these exact indentities,
and further in Section \ref{Sec results} we will prove a new one,
related to a special kind of topological conditioning;
see Theorem \ref{Thm main 1}.
\begin{itemize}
\item As observed in \cite[Proposition 5.2]{Lupu2016Iso},
given $x,y\in \widetilde{\LG}$,
the probability that $x$ and $y$ belong to the same sign component of
$\tilde{\phi}$, 
or alternatively to the same cluster of $\widetilde{\mathcal{L}}^{1/2}$,
equals
\begin{displaymath}
\mathbb{E}\big[
\operatorname{sign}(\tilde{\phi}(x))\operatorname{sign}(\tilde{\phi}(y))
\big]
=
\dfrac{2}{\pi}\arcsin\Big(
\dfrac{G(x,y)}{\sqrt{G(x,x)G(y,y)}}
\Big).
\end{displaymath}
\item If $V_{\partial}$ is divided into 3 parts
$V_{\partial,0}$, $V_{\partial,1}$ and $V_{\partial,2}$,
with $V_{\partial,0}$ allowed to be empty,
and if the boundary condition of a metric graph GFF $\tilde{\phi}$
is strictly positive on $V_{\partial,1}$ and $V_{\partial,2}$,
and zero on $V_{\partial,0}$ (rather than zero on the whole $V_{\partial}$),
then there is an explicit formula for the existence of a positive crossing
from $V_{\partial,1}$ to $V_{\partial,2}$:
\begin{displaymath}
\mathbb{P}\Big(
V_{\partial,1}\stackrel{\tilde{\phi}>0}{\longleftrightarrow} V_{\partial,2}
\Big) =
1- 
\exp\Big(
-2\sum_{x\in V_{\partial,1}}\sum_{y\in V_{\partial,2}}
\tilde{\phi}(x) H(x,y) \tilde{\phi}(y)
\Big),
\end{displaymath}
where $H(x,y)$ is the boundary Poisson kernel on $V_{\partial}\times V_{\partial}$.
This is \cite[Formula (18)]{LupuWerner2016Levy}.
If the boundary condition mixes values of different sign, no analogous explicit formula is known.
\item In the articles \cite{DrewitzPrevostRodriguez22PTRF,DrewitzPrevostRodriguez21CriticalExp}
the authors provide the exact law for the effective conductance (called capacity there) between the boundary $V_{\partial}$ and
the connected component containing $x_{0}$ of the level set
\begin{displaymath}
\{x\in \widetilde{\LG} \vert \tilde{\phi}(x)\geq h\}
\end{displaymath}
($x_{0}\in \widetilde{\LG}$ and $h\geq 0$ fixed).
This also can be derived from \cite[Proposition 5]{LupuWerner2016Levy}.
\item The Lévy transformation for the one-dimensional Brownian motion
can be extended to the general metric graph GFFs, as explained in
\cite{LupuWerner2016Levy}.
\end{itemize}

\section{Gauge-twisted GFF on metric graph, double cover and equivalence to topological conditioning}
\label{Sec results}

\subsection{Gauge-twisted GFF and subdivision of edges}
\label{Subsec subdivision}

Consider the electrical network $\LG=(V,E)$ as in 
Section \ref{Subsec gauge}.
For $N\in\N\setminus\{0\}$, we will denote by
$\LG^{(N)}=(V^{(N)},E^{(N)})$ the electrical network obtained from 
$\LG=(V,E)$ by subdividing each edge $e\in E$ into $N$.
In this way, $\LG^{(1)} = \LG$.
In general,
\begin{displaymath}
\vert E^{(N)}\vert = N \vert E\vert,
\qquad
\vert V^{(N)}\vert = \vert V \vert + (N-1) \vert E\vert .
\end{displaymath}
Moreover, if $N$ is a divisor of $N'$, then
$V^{(N)}$ is naturally a subset of $V^{(N')}$.
In particular, we always have $V\subset V^{(N)}$.
Given $e\in E$, we will denote $E^{(N)}(e)$ the subset of $E^{(N)}$
made of edges obtained by subdivision of $e$.
We also denote by $V^{(N)}(e)$ the endpoints of edges in $E^{(N)}(e)$.
So $\vert E^{(N)}(e)\vert = N$,
$\vert V^{(N)}(e)\vert = N+1$ and
$\vert V^{(N)}(e)\setminus V\vert = N-1$.
We endow $\LG^{(N)}$ with the following conductances:
for every $e\in E$ and $e'\in E^{(N)}(e)$,
$C^{(N)}(e') = N C(e)$.
Let be the energy
\begin{displaymath}
\mathcal{E}^{(N)}(f,f) = 
\sum_{\{x,y\}\in E^{(N)}}C^{(N)}(x,y)(f(y)-f(x))^{2}.
\end{displaymath}
We also see the electrical network
$\LG^{(N)}$, more precisely $V^{(N)}$,
as a subset of the metric graph
$\widetilde{\LG}$.
In this way, $V^{(N)}(e)$ is a set of $N+1$ points on $I_{e}$ with equal spacing
$(N C(e))^{-1}$, which includes the two endpoints of $e$.

If $(\phi^{(N)}(x))_{x\in V^{(N)}}$ is a GFF on $\LG^{(N)}$,
then it's restriction to $V=V^{(1)}$ is a GFF on $\LG$.
Indeed,
\begin{displaymath}
\mathcal{E}^{(N)}(f,f) = \mathcal{E}^{(1)}(f_{\vert V},f_{\vert V}) 
+ N(N-1)\sum_{e\in E}C(e)\sum_{i=1}^{N}(f(x_{i})-f(x_{i-1})-(f(x_{N})-f(x_{0}))/N)^{2},
\end{displaymath}
where $x_{0}, x_{1},\dots, x_{N}$ is the ordered set of points in $V^{(N)}(e)$.
Moreover, if $(\tilde{\phi}(x))_{x\in\widetilde{\LG}}$ is a metric graph GFF,
then its restriction $\tilde{\phi}_{\vert V^{(N)}}$ is distributed as a GFF 
$\phi^{(N)}$ on $\LG^{(N)}$.
So, in a sense, $\tilde{\phi}$ is the limit of $\phi^{(N)}$ as $N\to +\infty$.

Now, given a gauge field $\sigma\in\{ -1,1\}^{E}$,
we similarly want a natural extension of a $\sigma$-twisted GFF $\phi_{\sigma}$ to the subdivided graphs
$\LG^{(N)}$, and ultimately to the metric graph $\widetilde{\LG}$.
So, given $\sigma\in\{ -1,1\}^{E}$, we will denote by
$\sigma^{(N)}$ the following element of $\{ -1,1\}^{E^{(N)}}$.
If $e\in E$ and $\sigma(e)=1$,
then for every $e'\in E^{(N)}(e)$, $\sigma^{(N)}(e')=1$.
However, if $e\in E$ and $\sigma(e)= -1$, 
then among the edges in $E^{(N)}(e)$ there is exactly one with sign $-1$ under 
$\sigma^{(N)}$,
all the $N-1$ other having sign $+1$ (so there are actually choices to make for 
$\sigma^{(N)}$,
$N$ choices for each $e\in E$ with $\sigma(e)= -1$).

\begin{prop}
\label{Prop sigma N}
Let $N\in\N\setminus\{0\}$, $\sigma\in \{ -1,1\}^{E}$
and $\sigma^{(N)}\in\{ -1,1\}^{E^{(N)}}$ as above.
Let $(\phi^{(N)}_{\sigma^{(N)}}(x))_{x\in V^{(N)}}$ be the $\sigma^{(N)}$-twisted GFF on 
$\LG^{(N)}$ with $0$ boundary conditions on $V_{\partial}\subset V\subset V^{(N)}$.
Then its restriction to $V$ is distributed as $(\phi_{\sigma}(x))_{x\in V}$,
the $\sigma$-twisted GFF on $\LG$ with $0$ boundary conditions.
\end{prop}

\begin{proof}
Consider the energy
\begin{displaymath}
\mathcal{E}^{(N)}_{\sigma}(f,f) = 
\sum_{\{x,y\}\in E^{(N)}}C^{(N)}(x,y)(\sigma^{(N)}(x,y)f(y)-f(x))^{2}.
\end{displaymath}
Let $e\in E$ and let $x_{0}, x_{1},\dots, x_{N}$ be the ordered set of points in $V^{(N)}(e)$.
For $i\in \{ 1,\dots, N\}$, let be
\begin{displaymath}
\tilde{\sigma}_{i} = 
\sigma^{(N)}(x_{0},x_{1})\dots\sigma^{(N)}(x_{i-1},x_{i}).
\end{displaymath}
Then
\begin{multline}
\label{Eq decomp edge}
\sum_{i=1}^{N}C^{(N)}(x_{i-1},x_{i})
(\sigma^{(N)}(x_{i-1},x_{i})f(x_{i})-f(x_{i-1}))^{2}
= C(x_{0},x_{N})(\sigma(x_{0},x_{N})f(x_{N})-f(x_{0}))^{2}
\\
+
N(N-1)C(x_{0},x_{N})
\sum_{i=1}^{N}(\tilde{\sigma}_{i}f(x_{i})-\tilde{\sigma}_{i-1}f(x_{i-1})
-(\sigma(x_{0},x_{N})f(x_{N})-f(x_{0}))/N)^{2}.
\end{multline}
This induces a decomposition of the energy $\mathcal{E}^{(N)}_{\sigma}$.
This implies that $\phi^{(N)}_{\sigma^{(N)}}$ restricted to
$V$ is distributed as $\phi_{\sigma}$.
Moreover, 
the fields 
\begin{displaymath}
(\tilde{\sigma}_{i}\phi^{(N)}_{\sigma^{(N)}}(x_{i}) + 
i(\sigma(x_{0},x_{N})\phi^{(N)}_{\sigma^{(N)}}(x_{N})
-\phi^{(N)}_{\sigma^{(N)}}(x_{0}))/N)_{1\leq i\leq N-1, e \in E},
\end{displaymath}
are independent from $\phi^{(N)}_{\sigma^{(N)}}$ restricted to $V$,
with also independence across $e\in E$.
\end{proof}

In the sequel, $\sigma\in \{ -1,1\}^{E}$ and,
for the sake of simplicity,
$N$ will be odd and $\sigma^{(N)}$ is chosen such that
for $e\in E$ with $\sigma(e)= -1$, 
the middle edge in $E^{(N)}(e)$ has the sign $-1$ under $\sigma^{(N)}$.
We will define a random Gaussian field $\tilde{\phi}_{\sigma}$ on the metric
graph $\widetilde{\LG}$ as follows.
We consider the following independent objects.
\begin{itemize}
\item The twisted discrete GFF $(\phi_{\sigma}(x))_{x\in V}$.
\item For every edge $e\in E$, and independent Brownian bridge
$(W_{e}(u))_{0\leq u \leq C(e)^{-1}}$ of length $C(e)^{-1}$, from $0$ to $0$.
\end{itemize}
We will also chose for each edge $e\in E$ and arbitrary orientation and denote
its endpoints $e_{-}$ and $e_{+}$. 
For $u\in [0,C(e)^{-1}]$, $x_{e,u}$ will denote the point of $I_{e}$ at distance 
$u$ from $e_{+}$.
We define 
$(\tilde{\phi}_{\sigma}(x))_{x\in \widetilde{\LG}}$ as follows.
\begin{itemize}
\item For $x\in V$, $\tilde{\phi}_{\sigma}(x)=\phi_{\sigma}(x)$.
\item If $e\in E$ and $\sigma(e)=1$,
then for every $u\in (0,C(e)^{-1})$, 
\begin{equation}
\label{Eq tilde phi sigma 1}
\tilde{\phi}_{\sigma}(x_{e,u}) = W_{e}(u) + 
C(e)(u\phi_{\sigma}(e_{-})+ (C(e)^{-1}-u)\phi_{\sigma}(e_{+})).
\end{equation}
\item If $e\in E$ and $\sigma(e)= -1$,
then for $u\in (0,C(e)^{-1}/2)$,
\begin{equation}
\label{Eq tilde phi sigma 2}
\tilde{\phi}_{\sigma}(x_{e,u}) = W_{e}(u) + 
C(e)(-u\phi_{\sigma}(e_{-})+ (C(e)^{-1}-u)\phi_{\sigma}(e_{+})),
\end{equation}
and for $u\in (C(e)^{-1}/2,C(e)^{-1})$,
\begin{equation}
\label{Eq tilde phi sigma 3}
\tilde{\phi}_{\sigma}(x_{e,u}) = -W_{e}(u) + 
C(e)(u\phi_{\sigma}(e_{-}) - (C(e)^{-1}-u)\phi_{\sigma}(e_{+})).
\end{equation}
The value for $u=C(e)^{-1}/2$ is not specified and can be chosen arbitrary.
\end{itemize}
Defined in this way, the field $\tilde{\phi}_{\sigma}$ is discontinuous
in the middle of each $I_{e}$ for 
$e\in \{e\in E\vert \sigma(e)= -1\}$, with
\begin{displaymath}
\lim_{u\to (C(e)^{-1}/2)_{-}}\tilde{\phi}_{\sigma}(x_{e,u})
= -
\lim_{u\to (C(e)^{-1}/2)_{+}}\tilde{\phi}_{\sigma}(x_{e,u}).
\end{displaymath}
However, the absolute value 
$\vert \tilde{\phi}_{\sigma}\vert$ extends to a continuous field on
$\widetilde{\LG}$.
Note that the law of $\tilde{\phi}_{\sigma}$ does not depend on the arbitrary choice of orientations for the edges, since
$\phi_{\sigma} \stackrel{(\text{law})}{=} - \phi_{\sigma}$ and
$W_{e} \stackrel{(\text{law})}{=} -W_{e}$.

\begin{prop}
\label{Prop sigma N metric}
For every $N$ odd,
the restriction of $\tilde{\phi}_{\sigma}$ to $V^{(N)}$ is distributed as
$\phi^{(N)}_{\sigma^{(N)}}$.
\end{prop}

\begin{proof}
This is an immediate consequence of the decomposition \eqref{Eq decomp edge}.
\end{proof}

So the metric graph field $\tilde{\phi}_{\sigma}$ appears as the limit of
$\phi^{(N)}_{\sigma^{(N)}}$ as $N\to +\infty$.
However, it is not continuous
(except for $\sigma$ equal to $1$ everywhere). 
To recover continuous fields we will work on a double cover of 
$\widetilde{\LG}$; see Section \ref{Subsec cov metric}.

\subsection{Double covers induced by gauge fields}
\label{Subsec double cover}

Let $\sigma\in\{ -1,1\}^{E}$.
We will introduce a double cover
$\LG^{\rm db}_{\sigma} = (V^{\rm db}, E^{\rm db}_{\sigma})$
of the graph $\LG$, which is as follows.
$V^{\rm db}$ consist of two copies of $V$,
$V^{\rm db} = V_{1}\cup V_{2}$, and the projection
$\bpi_{\sigma} : V^{\rm db} \rightarrow V$ induces a bijection
between $V_{1}$ and $V$ and between $V_{2}$ and $V$.
The set of edges $E^{\rm db}_{\sigma}$ is as follows.
Let $\{ x,y\}\in E$ and let $x_{1},y_{1}\in V_{1}$ and
$x_{2},y_{2}\in V_{2}$ such that
$\bpi_{\sigma}(x_{1})=\bpi_{\sigma}(x_{2})=x$ and
$\bpi_{\sigma}(y_{1})=\bpi_{\sigma}(y_{2})=y$.
\begin{itemize}
\item If $e\in E$ and $\sigma(e)=1$, 
then $\{x_{1},y_{1}\}, \{x_{2},y_{2}\} \in E^{\rm db}_{\sigma}$.
\item If $e\in E$ and $\sigma(e)= -1$, then
$\{x_{1},y_{2}\}, \{x_{2},y_{1}\} \in E^{\rm db}_{\sigma}$.
\end{itemize}
See Figure \ref{Fig double cover} for an example.
We also have a projection
$\bpi_{\sigma} : E^{\rm db}_{\sigma} \rightarrow E$ such that
\begin{displaymath}
\bpi_{\sigma}(\{\hat{x},\hat{y}\})
=\{\bpi_{\sigma}(\hat{x}),\bpi_{\sigma}(\hat{y})\}.
\end{displaymath}
In this way $\bpi_{\sigma}$ is a graph covering map from $\LG^{\rm db}_{\sigma}$
to $\LG$.
Given $x_{1}\in V_{1}$ and $x_{2}\in V_{2}$ such that
$\bpi_{\sigma}(x_{1})=\bpi_{\sigma}(x_{2})$,
we will define $\psi(x_{1})=x_{2}$ and 
$\psi(x_{2})=x_{1}$. 
In this way, $\psi$ is a bijection from $V^{\rm db}$ to $V^{\rm db}$.
It is also a graph automorphism of $\LG^{\rm db}_{\sigma}$,
that is to say if $\{\hat{x},\hat{y}\}\in E^{\rm db}_{\sigma}$
then $\{\psi(\hat{x}),\psi(\hat{y})\}\in E^{\rm db}_{\sigma}$.
The map $\psi$ is also a covering automorphism,
because $\bpi_{\sigma}\circ\psi = \bpi_{\sigma}$.

\begin{figure}
\belowbaseline[0pt]{
\includegraphics[scale=0.48]{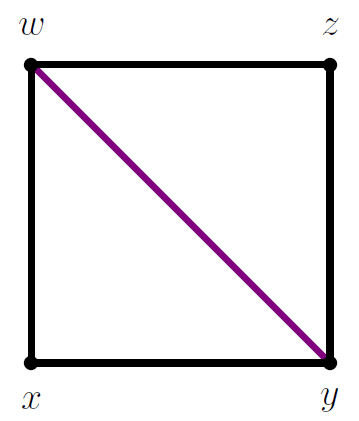}}
\qquad
\belowbaseline[0pt]{
\includegraphics[scale=0.48]{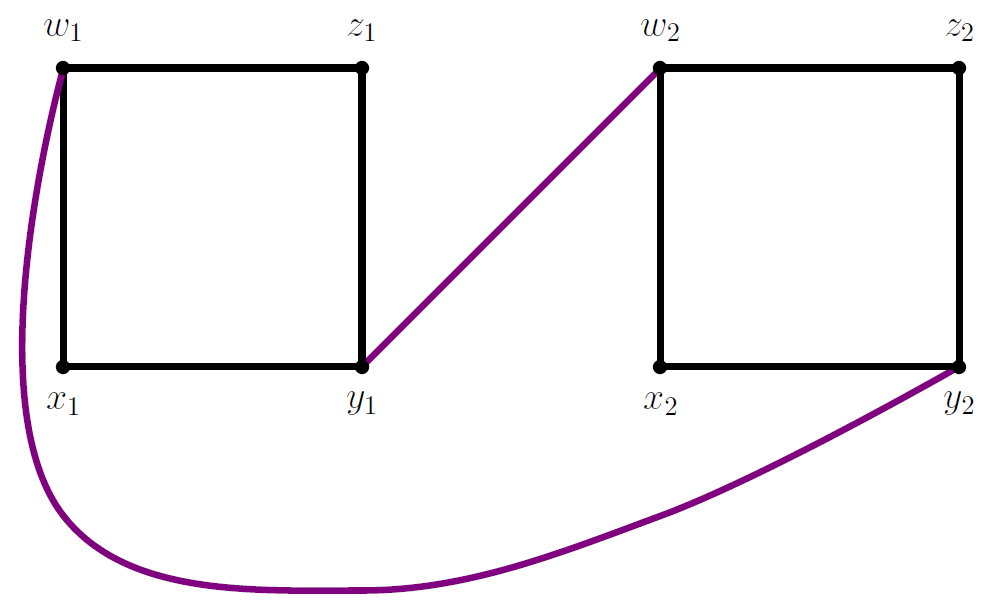}}
\caption{Left: a graph with 4 vertices and 5 edges. Right: its double cover.
The edges in violet correspond to $\sigma(e) = -1$.}
\label{Fig double cover}
\end{figure}

\begin{lemma}
\label{Lem cov hol}
Let $\wp=(x_{1},x_{2},\dots, x_{n-1},x_{n},x_{1})$ be a nearest neighbor loop on 
$\LG$ and let
$\hat{\wp}=(\hat{x}_{1},\hat{x}_{2},\dots, \hat{x}_{n-1},\hat{x}_{n},\hat{x}_{n+1})$ be a lift of $\wp$ on $\LG^{\rm db}_{\sigma}$ with respect to the covering map 
$\bpi_{\sigma}$:
$\bpi_{\sigma}(\hat{x}_{i})=x_{i}$ for $i\in\{1,2,\dots, n\}$,
$\bpi_{\sigma}(\hat{x}_{n+1})=x_{1}$,
and $\{\hat{x}_{i},\hat{x}_{i+1}\}\in E^{\rm db}_{\sigma}$
for $i\in\{1,2,\dots, n\}$.
Then $\hol^{\sigma}(\wp)=1$ if and only if $\hat{x}_{n+1}=\hat{x}_{1}$.
Otherwise $\hol^{\sigma}(\wp)= -1$ and 
$\hat{x}_{n+1}=\psi(\hat{x}_{1})$.
\end{lemma}

\begin{proof}
By construction, whenever $\sigma(x_{i},x_{i+1})=1$,
then $\hat{x}_{i}$ and $\hat{x}_{i+1}$ belong to the same copy $V_{1}$ or $V_{2}$.
Whenever $\sigma(x_{i},x_{i+1})= -1$ then either
$\hat{x}_{i}\in V_{1}$ and $\hat{x}_{i+1}\in V_{2}$,
or vice-versa.
So $\hat{x}_{n+1}=\hat{x}_{1}$ if and only even there is an even number of transitions from $V_{1}$ to $V_{2}$ or from $V_{2}$ to $V_{1}$,
that is to say that the number of $i\in\{1,2,\dots, n\}$
such that $\sigma(x_{i},x_{i+1})= -1$ is even, 
which exactly means that $\hol^{\sigma}(\wp)=1$.
\end{proof}

\begin{lemma}
\label{Lem iso cov}
Let $\sigma, \sigma'\in\{ -1,1\}^{E}$.
\begin{enumerate}
\item Assume that $\sigma$ and $\sigma'$ are gauge equivalent, and let
$\hat{\sigma}\in \{ -1,1\}^{V}$ such that
$\sigma' = \hat{\sigma}\cdot\sigma$.
Define $\psi_{\hat{\sigma}} :V^{\rm db}\rightarrow V^{\rm db}$
as follows. Let $\hat{x}\in V^{\rm db}$.
If $\hat{\sigma}(\bpi_{\sigma}(\hat{x})) = 1$,
then $\psi_{\hat{\sigma}}(\hat{x}) = \hat{x}$.
If $\hat{\sigma}(\bpi_{\sigma}(\hat{x})) = -1$,
then $\psi_{\hat{\sigma}}(\hat{x}) = \psi(\hat{x})$.
Then $\psi_{\hat{\sigma}}$ induces a graph isomorphism between 
$\LG^{\rm db}_{\sigma}$ and $\LG^{\rm db}_{\sigma'}$,
which is moreover a covering isomorphism, i.e.
$\bpi_{\sigma'} = \bpi_{\sigma}\circ \psi_{\hat{\sigma}}$.
\item Conversely, if there exists a covering isomorphism between 
$\LG^{\rm db}_{\sigma}$ et $\LG^{\rm db}_{\sigma'}$, 
then $\sigma$ and $\sigma'$ are gauge equivalent.
\end{enumerate}
\end{lemma}

\begin{proof}
(1) The map $\psi_{\hat{\sigma}}$ is clearly a bijection from 
$V^{\rm db}$ to itself, and it is clear that 
$\bpi_{\sigma'} = \bpi_{\sigma}\circ \psi_{\hat{\sigma}}$.
One needs only to check that whenever $\{\hat{x},\hat{y}\}\in E^{\rm db}_{\sigma}$,
then also $\{\psi_{\hat{\sigma}}(\hat{x}),\psi_{\hat{\sigma}}(\hat{y})\}
\in E^{\rm db}_{\sigma'}$.
But this is clear from the definition of $\psi_{\hat{\sigma}}$
and the fact that
\begin{displaymath}
\sigma'(\bpi_{\sigma'}(\psi_{\hat{\sigma}}(\hat{x})),
\bpi_{\sigma'}(\psi_{\hat{\sigma}}(\hat{y})))
=
\sigma'(\bpi_{\sigma}(\hat{x}),
\bpi_{\sigma}(\hat{y}))
=
\hat{\sigma}(\bpi_{\sigma}(\hat{x}))
\sigma(\bpi_{\sigma}(\hat{x}),
\bpi_{\sigma}(\hat{y}))
\hat{\sigma}(\bpi_{\sigma}(\hat{y})).
\end{displaymath}

(2) Let $\hat{\psi}$ be a covering a covering isomorphism between 
$\LG^{\rm db}_{\sigma}$ et $\LG^{\rm db}_{\sigma'}$.
Since $\bpi_{\sigma'} = \bpi_{\sigma}\circ \hat{\psi}$,
for every $\hat{x}\in V^{\rm db}$,
\begin{itemize}
\item either $(\hat{\psi}(\hat{x}),\hat{\psi}(\psi(\hat{x})))
=(\hat{\psi}(\hat{x}),\hat{\psi}(\psi(\hat{x})))$,
\item or $(\hat{\psi}(\hat{x}),\hat{\psi}(\psi(\hat{x})))
=(\hat{\psi}(\psi(\hat{x})),\hat{\psi}(\hat{x}))$.
\end{itemize}
Note that here $\hat{x}$ and $\psi(\hat{x})$ play symmetric roles.
In the first case we set $\hat{\sigma}(\bpi_{\sigma}(\hat{x}))=1$.
In the second case we set $\hat{\sigma}(\bpi_{\sigma}(\hat{x}))= -1$.
Then it is easy to check that $\sigma' = \hat{\sigma}\cdot\sigma$.
\end{proof}

\begin{lemma}
\label{Lem cov triv}
Let $\sigma\in\{ -1,1\}^{E}$ be a gauge field and let $\LG^{\rm db}_{\sigma}$
be the corresponding double cover.
Then the graph $\LG^{\rm db}_{\sigma}$ is connected if and only if
$\sigma$ is non-trivial.
Otherwise $\LG^{\rm db}_{\sigma}$ consists of two disconnected copies of
$\LG$.
\end{lemma}

\begin{proof}
If $\sigma(e)=1$ for every $e\in E$,
then by construction, $V_{1}$ and $V_{2}$ are not connected in 
$\LG^{\rm db}_{\sigma}$ and one has just two copies $\LG$ with vertex sets
$V_{1}$ and $V_{2}$ respectively.
In the more general case of $\sigma$ trivial, one can apply 
Lemma \ref{Lem iso cov} to get that $\LG^{\rm db}_{\sigma}$
is isomorphic as a covering to the previous case.

Assume now that $\sigma$ is non-trivial. Let $x\in V$,
and let $x_{1}\in V_{1}$, $x_{2}\in V_{2}$ such that
$\bpi_{\sigma}(x_{1}) = \bpi_{\sigma}(x_{2}) = x$.
According to Lemma \ref{Lem hol gauge equive}, there is a loop $\wp$ rooted in $x$
such that $\hol^{\sigma}(\wp)= -1$.
Let $\hat{\wp}$ be a lift in $\LG^{\rm db}_{\sigma}$ of the loop $\wp$.
By Lemma \ref{Lem cov hol}, $\hat{\wp}$ is a path joining $x_{1}$ and $x_{2}$.
Thus $x_{1}$ and $x_{2}$ belong to the same connected component.
Further, by construction, if $\{x,y\}\in E$, then the two sets
$\bpi_{\sigma}^{-1}(\{ x\})$ and $\bpi_{\sigma}^{-1}(\{ y\})$
are connected to each other. 
Since $\LG$ is connected, then so is $\LG^{\rm db}_{\sigma}$.
\end{proof}

Further, we will endow the double covers $\LG^{\rm db}_{\sigma}$
with conductances as follows.
Given $\{\hat{x},\hat{y}\}\in E^{\rm db}_{\sigma}$,
\begin{equation}
\label{Eq conductance cov}
C(\hat{x},\hat{y}) = C(\bpi_{\sigma}(\hat{x}),\bpi_{\sigma}(\hat{y})).
\end{equation} 
Let be $V^{\rm db}_{\partial} = \bpi_{\sigma}^{-1}(V_{\partial})$ and
$V^{\rm db}_{\rm int} = \bpi_{\sigma}^{-1}(V_{\rm int})$.
We consider the nearest-neighbor Markov jump process on $\LG^{\rm db}_{\sigma}$
with the transition rates given by the conductances,
and killed upon hitting $V^{\rm db}_{\partial}$.
Let $p_{\sigma}(t,\hat{x},\hat{y})$ be the transition probabilities of this killed process.
The following is immediate.

\begin{lemma}
\label{Lem heat ker cov}
\begin{enumerate}
\item For every $\hat{x},\hat{y}\in V^{\rm db}_{\rm int}$,
\begin{displaymath}
p_{\sigma}(t,\hat{x},\hat{y}) + p_{\sigma}(t,\hat{x},\psi(\hat{y}))
= p(t,\bpi_{\sigma}(\hat{x}),\bpi_{\sigma}(\hat{y})).
\end{displaymath}
\item For every $\hat{x},\hat{y}\in V^{\rm db}_{\rm int}$,
such that $\hat{x}$ and $\hat{y}$ are both in $V_{1}$ or both in $V_{2}$,
\begin{displaymath}
\mathbb{P}_{t}^{\bpi_{\sigma}(\hat{x}),\bpi_{\sigma}(\hat{y})}
(\hol^{\sigma}(\wp)= -1)
= \dfrac{p_{\sigma}(t,\hat{x},\psi(\hat{y}))}
{p(t,\bpi_{\sigma}(\hat{x}),\bpi_{\sigma}(\hat{y}))},
\end{displaymath}
where $\mathbb{P}_{t}^{\bpi_{\sigma}(\hat{x}),\bpi_{\sigma}(\hat{y})}$
is the bridge probability measure for the Markov jump process on $\LG$
conditioned on staying in $V_{\rm int}$.
\end{enumerate}
\end{lemma}

A \textit{fundamental domain} of $\bpi_{\sigma}$ is a subset
$\mathcal{D}\subset V^{\rm db}$ such that $\bpi_{\sigma}$
induces a bijection from $\mathcal{D}$ to $V$.
By construction, $V_{1}$ and $V_{2}$
are both fundamental domains.
The following is an immediate consequence of
\eqref{Eq mu loop} and Lemma \ref{Lem heat ker cov}.

\begin{cor}
\label{Cor mu neg hol cov}
Let be a gauge field $\sigma\in\{ -1,1\}^{E}$.
Then
\begin{displaymath}
\mu^{\rm loop}(\{\text{Loops } \wp \text{ with } \hol^{\sigma}(\wp)=-1\} =
\sum_{\hat{x}\in \mathcal{D}\cap V_{\rm int}^{\rm db}}
\int_{0}^{+\infty} p_{\sigma}(t,\hat{x},\psi(\hat{x})) \dfrac{dt}{t},
\end{displaymath}
where $\mathcal{D}$ is any fundamental domain.
\end{cor}

Let be the space $\mathcal{S}_{0}^{\rm db}$:
\begin{displaymath}
\mathcal{S}_{0}^{\rm db} = 
\{ f\in\R^{V^{\rm db}} \vert \forall \hat{x}\in V^{\rm db}_{\partial},
f(\hat{x})=0\}.
\end{displaymath}
We endow $\mathcal{S}_{0}^{\rm db}$ with the positive definite inner product
\begin{displaymath}
\mathcal{E}^{\rm db}_{\sigma}(f,g)=
\sum_{\{\hat{x},\hat{y}\}\in E^{\rm db}_{\sigma}}
C(\hat{x},\hat{y})(f(\hat{y})-f(\hat{x}))(g(\hat{y})-g(\hat{x})).
\end{displaymath}
Let $\mathcal{S}_{0,+}^{\rm db}$ and $\mathcal{S}_{0,-}^{\rm db}$
be the following subspaces of $\mathcal{S}_{0}^{\rm db}$:
\begin{displaymath}
\mathcal{S}_{0,+}^{\rm db} = 
\{f\in \mathcal{S}_{0}^{\rm db}\vert 
f\circ\psi = f
\},
\qquad
\mathcal{S}_{0,-}^{\rm db} = 
\{f\in \mathcal{S}_{0}^{\rm db}\vert 
f\circ\psi = -f
\}.
\end{displaymath}
Note that $\dim \mathcal{S}_{0,+}^{\rm db} = \dim \mathcal{S}_{0,-}^{\rm db}
= \vert V_{\rm int}\vert$, and
$\dim \mathcal{S}_{0}^{\rm db} = 2\vert V_{\rm int}\vert$.

\begin{lemma}
\label{Lem decomp S db}
The space $\mathcal{S}_{0}^{\rm db}$ is a direct orthogonal 
(for $\mathcal{E}^{\rm db}_{\sigma}$) sum of the subspaces 
$\mathcal{S}_{0,+}^{\rm db}$ and $\mathcal{S}_{0,-}^{\rm db}$.
\end{lemma}

\begin{proof}
It is clear that 
$\mathcal{S}_{0,+}^{\rm db}\cap\mathcal{S}_{0,-}^{\rm db}=\{ 0\}$,
and every $f\in \mathcal{S}_{0}^{\rm db}$ can be decomposed as
\begin{displaymath}
f = \dfrac{1}{2}(f+f\circ\psi) + \dfrac{1}{2}(f-f\circ\psi).
\end{displaymath}
So $\mathcal{S}_{0}^{\rm db}$ is a direct sum of 
$\mathcal{S}_{0,+}^{\rm db}$ and $\mathcal{S}_{0,-}^{\rm db}$.
To check that the decomposition is orthogonal,
consider $f\in \mathcal{S}_{0,+}^{\rm db}$
and $g\in \mathcal{S}_{0,-}^{\rm db}$.
Given $\{ \hat{x},\hat{y}\}\in E^{\rm db}_{\sigma}$,
then $\{\psi(\hat{x}),\psi(\hat{y})\}\in E^{\rm db}_{\sigma}$ and
\begin{displaymath}
(f(\hat{y})-f(\hat{x}))(g(\hat{y})-g(\hat{x})) =
- (f(\psi(\hat{y}))-f(\psi(\hat{x})))(g(\psi(\hat{y}))-g(\psi(\hat{x})).
\end{displaymath}
So $\mathcal{E}^{\rm db}_{\sigma}(f,g)=0$.
\end{proof}

Let $\Delta_{\sigma}^{\rm db}$ denote the discrete Laplacian on 
$\LG_{\sigma}^{\rm db}$:
\begin{displaymath}
(\Delta_{\sigma}^{\rm db} f)(\hat{x}) = 
\sum_{\substack{\hat{y}\in V^{\rm db}\\
\{\hat{x},\hat{y}\}\in E^{\rm db}_{\sigma}}} 
C(\hat{x},\hat{y})(f(\hat{y})-f(\hat{x})),
\qquad \hat{x}\in V^{\rm db}.
\end{displaymath}
One can see $\Delta_{\sigma}^{\rm db}$ as an operator on
$\mathcal{S}_{0}^{\rm db}$ by taking
$((\Delta_{\sigma}^{\rm db} f)(\hat{x}))_{\hat{x}\in V^{\rm db}_{\rm int}}$
and extending to $V^{\rm db}_{\partial}$ by $0$.
Then, clearly, the subspaces 
$\mathcal{S}_{0,+}^{\rm db}$ and $\mathcal{S}_{0,-}^{\rm db}$
are stable by $\Delta_{\sigma}^{\rm db}$.
Therefore,
\begin{displaymath}
\operatorname{det}_{\mathcal{S}_{0}^{\rm db}}(-\Delta_{\sigma}^{\rm db})
=\operatorname{det}_{\mathcal{S}_{0,+}^{\rm db}}(-\Delta_{\sigma}^{\rm db}) 
\operatorname{det}_{\mathcal{S}_{0,-}^{\rm db}}(-\Delta_{\sigma}^{\rm db}),
\end{displaymath}
where the subscripts indicate the space considered.

Let $f: V^{\rm db} \rightarrow \R$, 
and let $f_{1}$ denote the following function from $V$ to $\R$.
Given $x\in V$ and $x_{1}\in V_{1}$ such that
$\bpi_{\sigma}(x_{1})=x$,
\begin{displaymath}
f_{1}(x)=f(x_{1}).
\end{displaymath}
If $f\in \mathcal{S}_{0,+}^{\rm db}$, then
\begin{displaymath}
(\Delta_{\sigma}^{\rm db} f)(x_{1}) = (\Delta_{\LG} f_{1})(x),
\end{displaymath}
and if $f\in \mathcal{S}_{0,-}^{\rm db}$, then
\begin{displaymath}
(\Delta_{\sigma}^{\rm db} f)(x_{1}) = (\Delta_{\LG,\sigma} f_{1})(x).
\end{displaymath}
The following is an immediate consequence of the above.

\begin{cor}
\label{Cor dets db}
We have that
\begin{displaymath}
\operatorname{det}_{\mathcal{S}_{0,+}^{\rm db}}(-\Delta_{\sigma}^{\rm db}) =
\det((G(x,y))_{x,y\in V_{\rm int}})^{-1},
\qquad
\operatorname{det}_{\mathcal{S}_{0,-}^{\rm db}}(-\Delta_{\sigma}^{\rm db}) =
\det((G_{\sigma}(x,y))_{x,y\in V_{\rm int}})^{-1}.
\end{displaymath}
In particular,
\begin{displaymath}
\dfrac{\operatorname{det}_{\mathcal{S}_{0,+}^{\rm db}}(-\Delta_{\sigma}^{\rm db})}
{\operatorname{det}_{\mathcal{S}_{0,-}^{\rm db}}(-\Delta_{\sigma}^{\rm db})}
=
\exp\big(-2\mu^{\rm loop}(\{\text{Loops } \wp \text{ with } \hol^{\sigma}(\wp)=-1\})\big).
\end{displaymath}
\end{cor}

Let $(G^{\rm db}_{\sigma}(\hat{x},\hat{y}))_{\hat{x},\hat{y}\in V^{\rm db}}$
denote the Green's function on $\LG^{\rm db}_{\sigma}$ with
$0$ boundary conditions on $V^{\rm db}_{\partial}$.
Note that
$G^{\rm db}_{\sigma}(\psi(\hat{x}),\psi(\hat{y}))
=G^{\rm db}_{\sigma}(\hat{x},\hat{y})$

\begin{lemma}
\label{Lem Green cov}
Let $x,y\in V$ and let $x_{1},y_{1}\in V_{1}$ and
$y_{2}\in V_{2}$ such that
$\bpi_{\sigma}(x_{1})=x$ and
$\bpi_{\sigma}(y_{1})=\bpi_{\sigma}(y_{2})=y$.
Then
\begin{displaymath}
G(x,y) = G^{\rm db}_{\sigma}(x_{1},y_{1}) + G^{\rm db}_{\sigma}(x_{1},y_{2}),
\qquad
G_{\sigma}(x,y) = G^{\rm db}_{\sigma}(x_{1},y_{1}) 
- G^{\rm db}_{\sigma}(x_{1},y_{2}),
\end{displaymath}
or equivalently
\begin{displaymath}
G^{\rm db}_{\sigma}(x_{1},y_{1}) 
=\dfrac{1}{2}(G(x,y) + G_{\sigma}(x,y)),
\qquad
G^{\rm db}_{\sigma}(x_{1},y_{2}) 
=\dfrac{1}{2}(G(x,y) - G_{\sigma}(x,y)).
\end{displaymath}
\end{lemma}

\begin{proof}
We have that
\begin{displaymath}
G^{\rm db}_{\sigma}(x_{1},y_{i})
=\int_{0}^{+\infty} p_{\sigma} (t,x_{1},y_{i})\, dt,
\qquad 
G(x,y) = \int_{0}^{+\infty} p(t,x,y)\,dt,
\end{displaymath}
and according to Theorem \ref{Thm Green hol gauge},
\begin{displaymath}
G_{\sigma}(x,y) = \int_{0}^{+\infty} p(t,x,y)
\mathbb{P}_{t}^{x,y}(\hol^{\sigma}(\wp)= -1)\,
dt.
\end{displaymath}
We conclude by Lemma \ref{Lem heat ker cov}.
\end{proof}

Let $(\phi^{\rm db}_{\sigma}(\hat{x}))_{\hat{x}\in V^{\rm db}}$
be the discrete GFF on the double cover $\LG^{\rm db}_{\sigma}$.
We have that
\begin{displaymath}
\E[\phi^{\rm db}_{\sigma}(\hat{x})] = 0,
\qquad
\E[\phi^{\rm db}_{\sigma}(\hat{x})\phi^{\rm db}_{\sigma}(\hat{y})]
= G^{\rm db}_{\sigma}(\hat{x},\hat{y}).
\end{displaymath}
Let $\phi^{\rm db}_{\sigma,+}$ and $\phi^{\rm db}_{\sigma,-}$
be the fields
\begin{displaymath}
\phi^{\rm db}_{\sigma,+} = 
\dfrac{1}{\sqrt{2}}(\phi^{\rm db}_{\sigma} + \phi^{\rm db}_{\sigma}\circ\psi),
\qquad
\phi^{\rm db}_{\sigma,-} = 
\dfrac{1}{\sqrt{2}}(\phi^{\rm db}_{\sigma} - \phi^{\rm db}_{\sigma}\circ\psi).
\end{displaymath}
The field $\phi^{\rm db}_{\sigma,+}$,
respectively $\phi^{\rm db}_{\sigma,-}$,
is a random element of $\mathcal{S}_{0,+}^{\rm db}$,
respectively $\mathcal{S}_{0,-}^{\rm db}$.
Since $\mathcal{S}_{0,+}^{\rm db}$ and $\mathcal{S}_{0,-}^{\rm db}$
are orthogonal for $\mathcal{E}^{\rm db}_{\sigma}$, the fields
$\phi^{\rm db}_{\sigma,+}$ and $\phi^{\rm db}_{\sigma,-}$
are independent.
Lemma \ref{Lem Green cov} immediately implies the following.

\begin{cor}
\label{Cor GFF cov}
Let $(\phi(x))_{x\in V}$ and $(\phi_{\sigma}(x))_{x\in V}$
be the following fields:
\begin{displaymath}
\phi(x) = \phi^{\rm db}_{\sigma,+}(x_{1}),
\qquad
\phi_{\sigma}(x) = \phi^{\rm db}_{\sigma,-}(x_{1}), 
\end{displaymath}
where $x_{1}\in V_{1}$ and $\bpi_{\sigma}(x_{1})=x$.
Then $\phi$ is distributed as the GFF on $\LG$ with $0$
boundary conditions \eqref{Eq phi},
and $\phi_{\sigma}$ is distributed as the $\sigma$-twisted GFF on $\LG$ with $0$
boundary conditions \eqref{Eq phi sigma}.
\end{cor}

\subsection{Double covers of metric graphs and related GFFs}
\label{Subsec cov metric}

Consider a gauge field $\sigma \in \{ -1,1\}^{E}$.
Let be the associated double cover $\LG^{\rm db}_{\sigma}$,
endowed with the conductances \eqref{Eq conductance cov}.
We will consider the metric graph associated to $\LG^{\rm db}_{\sigma}$,
denoted by $\widetilde{\LG}^{\rm db}_{\sigma}$.
The projection $\bpi_{\sigma}$ can be extended into a covering map
$\bpi_{\sigma}: \widetilde{\LG}^{\rm db}_{\sigma}
\rightarrow \widetilde{\LG}$
which is locally an isometry.
In this way, $\widetilde{\LG}^{\rm db}_{\sigma}$ is a double cover of the 
metric graph $\widetilde{\LG}$.
Similarly, the map $\psi : V^{\rm db} \rightarrow V^{\rm db}$
extends into a covering automorphism
$\psi_{\sigma} : \widetilde{\LG}^{\rm db}_{\sigma} \rightarrow 
\widetilde{\LG}^{\rm db}_{\sigma}$
($\bpi_{\sigma}\circ \psi_{\sigma} = \bpi_{\sigma}$)
which interchanges the two sheets of the covering.
Lemma \ref{Lem cov triv} extends to the metric graph setting:
the metric space $\widetilde{\LG}^{\rm db}_{\sigma}$ is connected 
if and only if the gauge field $\sigma$ is non-trivial.
Otherwise it consists of two disjoint copies of $\widetilde{\LG}$.

Let $\tilde{\phi}^{\rm db}_{\sigma}$ be the metric graph GFF on 
$\widetilde{\LG}^{\rm db}_{\sigma}$ with $0$ boundary conditions on
$V^{\rm db}_{\partial}$. 
The field $\tilde{\phi}^{\rm db}_{\sigma}$ is continuous.
It's restriction to $V^{\rm db}$ is the field 
$\phi^{\rm db}_{\sigma}$ introduced previously.
Let $\tilde{\phi}^{\rm db}_{\sigma,+}$ and 
$\tilde{\phi}^{\rm db}_{\sigma,-}$ be the fields
\begin{displaymath}
\tilde{\phi}^{\rm db}_{\sigma,+} = 
\dfrac{1}{\sqrt{2}}(\tilde{\phi}^{\rm db}_{\sigma} 
+ \tilde{\phi}^{\rm db}_{\sigma}\circ\psi_{\sigma}),
\qquad
\tilde{\phi}^{\rm db}_{\sigma,-} = 
\dfrac{1}{\sqrt{2}}(\tilde{\phi}^{\rm db}_{\sigma} - 
\tilde{\phi}^{\rm db}_{\sigma}\circ\psi_{\sigma}).
\end{displaymath}
The fields $\tilde{\phi}^{\rm db}_{\sigma,+}$ and 
$\tilde{\phi}^{\rm db}_{\sigma,-}$ are again continuous and their restrictions
to $V^{\rm db}$ are 
$\phi^{\rm db}_{\sigma,+}$ and $\phi^{\rm db}_{\sigma,-}$ respectively.

\begin{lemma}
\label{Lem indep metric cov}
The fields $\tilde{\phi}^{\rm db}_{\sigma,+}$ 
and $\tilde{\phi}^{\rm db}_{\sigma,-}$
are independent.
\end{lemma}

\begin{proof}
We already know that the fields 
$\phi^{\rm db}_{\sigma,+}$ and $\phi^{\rm db}_{\sigma,-}$
are independent.
To conclude, 
we need to look at what happens inside the edge-lines.
Given $(W_{1}(u))_{0\leq u\leq C(e)^{-1}}$
and $(W_{2}(u))_{0\leq u\leq C(e)^{-1}}$
two independent Brownian bridges from $0$ to $0$ of the same length,
then $(W_{1}+W_{2})/\sqrt{2}$ and $(W_{1}-W_{2})/\sqrt{2}$
are indeed independent and again distributed as Brownian
bridges from $0$ to $0$.
\end{proof}

Since $\tilde{\phi}^{\rm db}_{\sigma,+}\circ\psi_{\sigma}
= \tilde{\phi}^{\rm db}_{\sigma,+}$,
there is a continuous field
$\tilde{\phi}$ on $\widetilde{\LG}$
(i.e. on the base of the covering)
such that 
\begin{displaymath}
\tilde{\phi}^{\rm db}_{\sigma,+}
=
\tilde{\phi}\circ\bpi_{\sigma}.
\end{displaymath}

\begin{lemma}
\label{Lem usual GFF}
The field $\tilde{\phi}$ is distributed as the metric graph GFF on
$\widetilde{\LG}$ with $0$ boundary conditions on
$V_{\partial}$ (see Section \ref{Subsec metric}),
hence the same notation, 
and in particular its distribution is the same whatever the
gauge field $\sigma$.
\end{lemma}

\begin{proof}
By Corollary \ref{Cor GFF cov} we already know that
the restriction of the field $\tilde{\phi}$ to $V$
is distributed as the discrete GFF with $0$ boundary conditions on $V_{\partial}$.
We need to check that the identity in law also extends inside the
edge-lines.
Given $(W_{1}(u))_{0\leq u\leq C(e)^{-1}}$
and $(W_{2}(u))_{0\leq u\leq C(e)^{-1}}$
two independent Brownian bridges from $0$ to $0$ of the same length,
then indeed $(W_{1}+W_{2})/\sqrt{2}$ is distributed as a Brownian
bridge from $0$ to $0$.
\end{proof}

Next we will introduce a section 
$\mathsf{s}_{\sigma} : 
\widetilde{\LG}\rightarrow \widetilde{\LG}^{\rm db}_{\sigma}$,
such that
$\bpi_{\sigma}\circ \mathsf{s}_{\sigma}
=\operatorname{Id}_{\widetilde{\LG}}$.
In general, $\mathsf{s}_{\sigma}$ will not be continuous
and will actually have finitely many discontinuity points.
The section $\mathsf{s}_{\sigma}$ will be defined by the following conditions:
\begin{enumerate}
\item $\bpi_{\sigma}\circ \mathsf{s}_{\sigma}
=\operatorname{Id}_{\widetilde{\LG}}$.
\item For every $x\in V$,
$\mathsf{s}_{\sigma}(x)\in V_{1}$.
\item For every edge $e\in E$,
the map $u\mapsto \mathsf{s}_{\sigma}(x_{e,u})$
(see Section \ref{Subsec subdivision} for the notations)
is continuous on the intervals
$[0,C(e)^{-1}/2)$ and $(C(e)^{-1}/2,C(e)^{-1}]$,
i.e. a discontinuity is possible only in the middle of the
edge-line at $u=C(e)^{-1}/2$.
\end{enumerate}
The above entirely specifies $\mathsf{s}_{\sigma}$
outside the middle points of edge-lines in $\widetilde{\LG}$.
Further, if $\{x,y\}\in E$ and $\sigma(x,y)=1$,
then $\{\mathsf{s}_{\sigma}(x),\mathsf{s}_{\sigma}(y)\}\in E^{\rm db}_{\sigma}$,
and thus $\mathsf{s}_{\sigma}$ extends continuously to the whole
edge-line $I_{\{x,y\}}$.
However, if $\sigma(x,y)= -1$, then 
$\{\mathsf{s}_{\sigma}(x),\mathsf{s}_{\sigma}(y)\}\not\in E^{\rm db}_{\sigma}$,
and thus $\mathsf{s}_{\sigma}$ has a discontinuity in the middle of the edge-line
$I_{\{x,y\}}$. 
We will not specify which is the value assumed by $\mathsf{s}_{\sigma}$
at such a discontinuity point, i.e. which of the left or right limit, 
as this will not be of any importance.

Now, consider on $\widetilde{\LG}$ the field
\begin{equation}
\label{Eq phi sigma cov}
\tilde{\phi}_{\sigma} = \tilde{\phi}^{\rm db}_{\sigma,-}\circ\mathsf{s}_{\sigma}.
\end{equation}

\begin{lemma}
The field $\tilde{\phi}_{\sigma}$ has the same distribution
as the Gaussian field on $\widetilde{\LG}$ introduced in Section
\ref{Subsec subdivision}, hence the same notation.
\end{lemma}

\begin{proof}
By Corollary \ref{Cor GFF cov} we already know that
the restriction of the field $\tilde{\phi}_{\sigma}$ to $V$
is distributed as the discrete $\sigma$-twisted GFF $\phi_{\sigma}$ with $0$ 
boundary conditions on $V_{\partial}$.
We need only to check what happens inside the edge lines.
Given $(W_{1}(u))_{0\leq u\leq C(e)^{-1}}$
and $(W_{2}(u))_{0\leq u\leq C(e)^{-1}}$
two independent Brownian bridges from $0$ to $0$ of the same length,
then $W_{e}=(W_{1}-W_{2})/\sqrt{2}$ is again distributed as a Brownian
bridge from $0$ to $0$.
If $\sigma(e)=1$,
then $\mathsf{s}_{\sigma}$ is continuous on the edge-line $I_{e}$ and the restriction of $\tilde{\phi}_{\sigma}$
to $I_{e}$ is obtained by interpolating between
$\tilde{\phi}_{\sigma}(e_{-})$ and $\tilde{\phi}_{\sigma}(e_{+})$
with $W_{e}$ (plus the deterministic linear part).
If $\sigma(e)= -1$,
then $\mathsf{s}_{\sigma}$ has a discontinuity in the middle of 
$I_{e}$, and 
the restriction of $\tilde{\phi}_{\sigma}$
to $I_{e}$ is obtained by interpolating between
$-\tilde{\phi}_{\sigma}(e_{-})$ and $\tilde{\phi}_{\sigma}(e_{+})$
with $W_{e}$ (plus the deterministic linear part),
then by reflecting the bridge on half of $I_{e}$.
So this is indeed the same construction as the one given by the formulas
\eqref{Eq tilde phi sigma 1}, \eqref{Eq tilde phi sigma 2}
and \eqref{Eq tilde phi sigma 3}.
\end{proof}

The main point of the Sections \ref{Subsec double cover} and
\ref{Subsec cov metric} was actually to introduce the field
$\tilde{\phi}^{\rm db}_{\sigma,-}$.
The advantage of $\tilde{\phi}^{\rm db}_{\sigma,-}$ over $\tilde{\phi}_{\sigma}$
is that the former is always continuous, and in particular satisfies the intermediate value property.
This will be exploited in the forthcoming Section \ref{Subsec topo}.
See Lemma \ref{Lem T sigma}.

\subsection{Topological events for the metric graph GFF}
\label{Subsec topo}

Here we consider a gauge field $\sigma\in\{ -1,1\}^{E}$.
Given $U$ an open non-empty connected subset of the metric graph 
$\widetilde{\LG}$,
the inverse image $\bpi^{-1}_{\sigma}(U)\subset
\widetilde{\LG}_{\sigma}^{\rm db}$ is either connected or
consists of two isometric connected components.
It is easy to verify (see e.g. Lemmas \ref{Lem cov hol} and \ref{Lem cov triv})
that  $\bpi^{-1}_{\sigma}(U)$ is connected if and only there is a continuous
loop $\wp$ contained in $U$ such that
$\hol^{\sigma}(\wp) = -1$.
The holonomy of a continuous loop on the metric graph is given by the parity
of the number of crossings of edge-lines $I_{e}$ with $\sigma(e) =-1$.

In the example of Figure \ref{Fig non triv ann},
$\bpi^{-1}_{\sigma}(U)$ is connected if and only if
$U$ surrounds the inner hole of the domain.
In the example of Figure \ref{Fig 2 holes},
$\bpi^{-1}_{\sigma}(U)$ is connected if and only if
$U$ separates one hole from the other.
By \textit{separate} we mean separate as a subset of the continuum plane,
since the planar metric graph on the figure can be embedded into the plane.
However, if $U$ surrounds both holes without separating one form the other,
then $\bpi^{-1}_{\sigma}(U)$ is not connected.

So whether $\bpi^{-1}_{\sigma}(U)$ is connected or not is a topological,
or rather homotopical property of $U$.
It can be reformulated in a more abstract way as follows.
Let $\pi_{1}(\widetilde{\LG})$ be the fundamental group of the metric graph
$\widetilde{\LG}$
(the notation $\pi_{1}$ should not be confused with the notation 
$\bpi_{\sigma}$ for the covering map; 
$\pi_{1}$ is just the standard notation for the fundamental group).
Let $\pi_{1}(U)$ be the fundamental group of $U$.
Since $U$ is a subset of $\widetilde{\LG}$, 
there is a natural group homomorphism
$\theta_{U} : \pi_{1}(U)\rightarrow \pi_{1}(\widetilde{\LG})$,
obtained by considering the loops in $U$ as loops in $\widetilde{\LG}$
and rooting all the loops at a point $x_{0}\in U$.
The gauge field $\sigma$ induces a group homomorphism
$h_{\sigma} : \pi_{1}(\widetilde{\LG})\rightarrow \{-1,1\}$
obtained by taking the holonomy of loops.
The homomorphism $h_{\sigma}$ depends only on the gauge equivalence class of
$\sigma$. 
The gauge field $\sigma$ is trivial if and only if
$\ker h_{\sigma} = \pi_{1}(\widetilde{\LG})$.
Further, $\bpi^{-1}_{\sigma}(U)$ is \textbf{not} connected if and only if
$\ker h_{\sigma}\circ \theta_{U} = \pi_{1}(U)$.
This depends on $\sigma$ only through its gauge equivalence class,
and if $\sigma$ is trivial, 
$\bpi^{-1}_{\sigma}(U)$ cannot be connected.

Let $\mathcal{C}(\widetilde{\LG})$ denote the space of
continuous functions $\widetilde{\LG}\rightarrow\R$.
Given a function $f\in \mathcal{C}(\widetilde{\LG})$,
$\{ f\neq 0\}$ will be a short notation for the non-zero set of $f$:
\begin{displaymath}
\{ f\neq 0\} = \{x\in \widetilde{\LG}\vert f(x)\neq 0\}.
\end{displaymath}
Let $\LT_{\sigma}$ be the following subset of $\mathcal{C}(\widetilde{\LG})$:
\begin{displaymath}
\LT_{\sigma} = 
\{
f\in \mathcal{C}(\widetilde{\LG})\vert
~\forall~U \text{ connected component of } \{ f\neq 0\},
\bpi^{-1}_{\sigma}(U) \text{ is \textbf{not} connected}
\}.
\end{displaymath}
It is easy to see that $\LT_{\sigma}$ is a closed subset of 
$\mathcal{C}(\widetilde{\LG})$
for the uniform norm.
In the example of Figure \ref{Fig non triv ann},
$f\in \LT_{\sigma}$ if and only if
no connected component of $\{ f\neq 0\}$ surrounds the inner hole of the domain;
see Figure \ref{Fig ex ann}.
In the example of Figure \ref{Fig 2 holes},
$f\in \LT_{\sigma}$ if and only if no connected component of $\{ f\neq 0\}$
separates one hole from the other.
Again, $\LT_{\sigma}$ depends only on the gauge-equivalence class of
$\sigma$, and if $\sigma$ is trivial, then
$\LT_{\sigma}=\mathcal{C}(\widetilde{\LG})$.

We recall that, although the field $\tilde{\phi}_{\sigma}$ is usually not continuous, its absolute value $\vert\tilde{\phi}_{\sigma}\vert$ can always be 
continuously extended to $\widetilde{\LG}$.
Thus, we see $\vert\tilde{\phi}_{\sigma}\vert$ as a random element of
$\mathcal{C}(\widetilde{\LG})$.

\begin{lemma}
\label{Lem T sigma}
Almost surely, $\vert\tilde{\phi}_{\sigma}\vert\in \LT_{\sigma}$.
\end{lemma}

\begin{proof}
We will show that a.s., for every $U$ connected component of
$\{\vert\tilde{\phi}_{\sigma}\vert\neq 0\}$,
and for every $\hat{x}\in \bpi_{\sigma}^{-1}(U)$,
the points $\hat{x}$ and $\psi_{\sigma}(\hat{x})$
are not connected inside $\bpi_{\sigma}^{-1}(U)$.
If this were not the case, there would be a point
$\hat{x}\in \widetilde{\LG}_{\sigma}^{\rm db}$
and a continuous path $(\hat{\wp}(t))_{0\leq t\leq 1}$
joining $\hat{x}$ and $\psi_{\sigma}(\hat{x})$
such that for every $t\in [0,1]$,
$\vert\tilde{\phi}_{\sigma}\vert (\bpi_{\sigma}(\hat{\wp}(t)))>0$.
Consider on $\widetilde{\LG}_{\sigma}^{\rm db}$ the field
$\tilde{\phi}^{\rm db}_{\sigma,-}$,
related to $\tilde{\phi}_{\sigma}$ through \eqref{Eq phi sigma cov}.
Then, for ever $t\in [0,1]$,
$\tilde{\phi}^{\rm db}_{\sigma,-}(\hat{\wp}(t))\neq 0$.
But the field $\tilde{\phi}^{\rm db}_{\sigma,-}$ is continuous, and
$\tilde{\phi}^{\rm db}_{\sigma,-}(\hat{\wp}(0)) =
-\tilde{\phi}^{\rm db}_{\sigma,-}(\hat{\wp}(1))$.
So this is a contradiction.
\end{proof}

Now consider the usual metric graph GFF $\tilde{\phi}$,
without the gauge field $\sigma$.

\begin{lemma}
\label{Lem phi norm T sigma}
We have that $\mathbb{P}(\tilde{\phi}\in\LT_{\sigma})>0$.
Moreover, if $\sigma$ is non-trivial,
$\mathbb{P}(\tilde{\phi}\in\LT_{\sigma})< 1$.
\end{lemma}

\begin{proof}
Let us first prove that $\mathbb{P}(\tilde{\phi}\in\LT_{\sigma})>0$.
Let $A$ be the event that the field $\tilde{\phi}$ has zeroes inside every edge-line $I_{e}$ for $e\in E$. With the construction with Brownian bridges, it is clear that $\mathbb{P}(A)>0$.
Let us argue that on the event $A$, we have that $\tilde{\phi}\in\LT_{\sigma}$,
whatever the gauge field $\sigma$.
Indeed, on the event $A$, no connected component of 
$\{\tilde{\phi}\neq 0\}$ contains a full edge-line, and
thus cannot contain a loop with holonomy $-1$.

Assume now that the gauge field $\sigma$ is non-trivial.
Then there is $(\wp(t))_{0\leq t\leq 1}$ a continuous (deterministic) loop in 
$\widetilde{\LG}$ such that $\hol^{\sigma}(\wp) = -1$.
With positive probability, $\tilde{\phi}$ has no zeroes on the range of $\wp$.
On this event, the range of $\wp$ belongs to the same connected component of
$\{\tilde{\phi}\neq 0\}$, and because $\hol^{\sigma}(\wp) = -1$
we have that $\tilde{\phi}\not\in\LT_{\sigma}$.
\end{proof}

We are ready to state the main result of this paper.

\begin{thmm}
\label{Thm main 1}
Let be a gauge field $\sigma\in\{-1,1\}^{E}$.
Then
\begin{equation}
\label{Eq prob LT sigma}
\mathbb{P}(\tilde{\phi}\in\LT_{\sigma})
=
\exp\big(-\mu^{\rm loop}(\{\text{Loops } \wp \text{ with } \hol^{\sigma}(\wp)=-1\})\big).
\end{equation}
So in particular
(see Corollaries  \ref{Cor ratio det},
\ref{Cor mu neg hol cov} and \ref{Cor dets db}),
\begin{eqnarray*}
\mathbb{P}(\tilde{\phi}\in\LT_{\sigma})&=&
\dfrac{\det((G_{\sigma}(x,y))_{x,y\in V_{\rm int}})^{1/2}}
{\det((G(x,y))_{x,y\in V_{\rm int}})^{1/2}}
\\
&=&
\dfrac{\operatorname{det}_{\mathcal{S}_{0,+}^{\rm db}}(-\Delta_{\sigma}^{\rm db})^{1/2}}
{\operatorname{det}_{\mathcal{S}_{0,-}^{\rm db}}(-\Delta_{\sigma}^{\rm db})^{1/2}}
\\
&=&
\exp
\Big(
- \sum_{\hat{x}\in \mathcal{D}\cap V_{\rm int}^{\rm db}}
\int_{0}^{+\infty} p_{\sigma}(t,\hat{x},\psi(\hat{x})) \dfrac{dt}{t}
\Big),
\end{eqnarray*}
where $\mathcal{D}$ is any fundamental domain of the covering map $\bpi_{\sigma}$.
Moreover, conditionally on the event
$\{\tilde{\phi}\in\LT_{\sigma}\}$,
the field $\vert \tilde{\phi}\vert$ has the same distribution as
$\vert\tilde{\phi}_{\sigma}\vert$.
\end{thmm}

\begin{rem}
\label{Rem not iso}
We would like to emphasize that
Theorem \ref{Thm main 1} is not a direct consequence of
the isomorphism identity for $\tilde{\phi}$ (Theorem \ref{Thm Lupu iso}). 
Indeed, if $\tilde{\phi}\in\LT_{\sigma}$, then
$\widetilde{\mathcal{L}}^{1/2}$ cannot contain loops
$\wp$ with $\hol^{\sigma}(\wp)=-1$.
However, even if 
$\widetilde{\mathcal{L}}^{1/2}$ does not contain loops
$\wp$ with $\hol^{\sigma}(\wp)=-1$,
it can still form clusters $\mathcal{C}$ out of loops 
$\wp$ with $\hol^{\sigma}(\wp)= +1$
such that $\bpi_{\sigma}^{-1}(\mathcal{C})$ 
is connected.
See Figure \ref{Fig chain ann}.
In particular,
\begin{displaymath}
\mathbb{P}(\tilde{\phi}\in\LT_{\sigma})
\neq
\mathbb{P}\big(\forall \wp\in \widetilde{\mathcal{L}}^{1/2}, \hol^{\sigma}(\wp)= 1
\big).
\end{displaymath}
More precisely,
\begin{displaymath}
\mathbb{P}(\tilde{\phi}\in\LT_{\sigma})
=
\mathbb{P}\big(\forall \wp\in \widetilde{\mathcal{L}}^{1/2}, \hol^{\sigma}(\wp)= 1
\big)^{2}.
\end{displaymath}
However, this difference in exponents can be explained through the isomorphism
identity for the gauge-twisted GFF $\tilde{\phi}_{\sigma}$:
see Section \ref{Subsec iso topo} and the identity \eqref{Eq KL metric alpha 1}.
This factor $2$ in the exponent plays a crucial role in our intensity doubling conjecture; 
see Section \ref{Sec conj high dim} and Conjecture \ref{Conj double}.
\end{rem}

\begin{figure}
\includegraphics[scale=0.3]{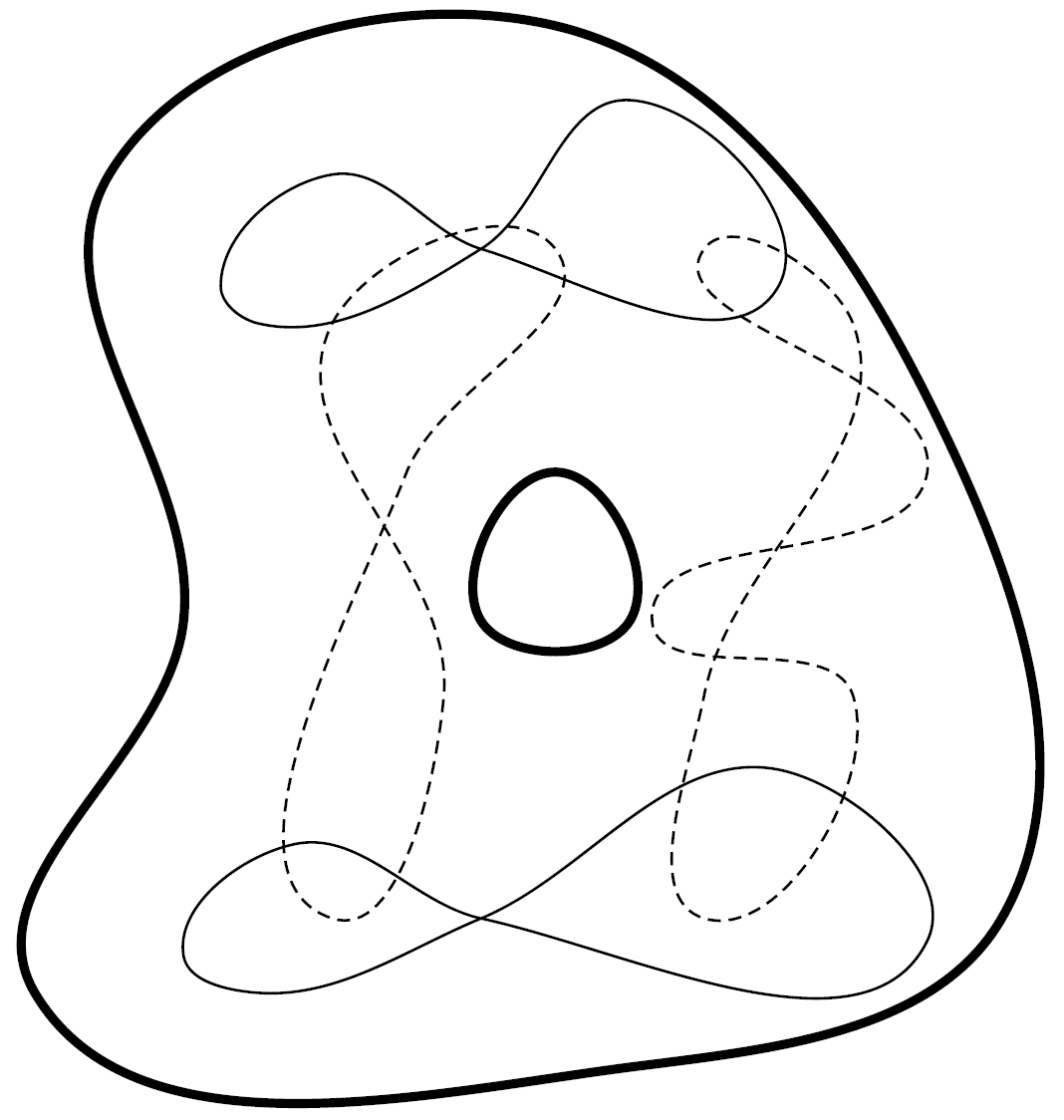}
\caption{
A non-contractible chain in an annular domain formed by contractible loops.}
\label{Fig chain ann}
\end{figure}

\begin{rem}
\label{Rem moments}
Theorem \ref{Thm main 1} implies in particular that for every
$x\in \widetilde{\LG}$,
\begin{displaymath}
\E[\tilde{\phi}(x)^{2}\vert \tilde{\phi}\in\LT_{\sigma}]
= G_{\sigma}(x,x).
\end{displaymath}
Further, for any $k\in\mathbb{N}\setminus\{0\}$
and $x_{1},x_{2},\dots, x_{k}\in \widetilde{\LG}$,
for the conditional moments (products of squares)
\begin{displaymath}
\E[\tilde{\phi}(x_{1})^{2}\tilde{\phi}(x_{2})^{2}
\dots \tilde{\phi}(x_{k})^{2}\vert \tilde{\phi}\in\LT_{\sigma}]
\end{displaymath}
the Wick's formula with kernel $G_{\sigma}$ holds.
\end{rem}

\begin{rem}
\label{Rem interact bos}
Theorem \ref{Thm main 1} immediately extends to interacting bosonic fields
on the metric graph where the interaction potential depends only on the absolute value of the field.
For instance, one can consider the $\varphi^{4}$ fields on $\widetilde{\LG}$.
Let $g>0$ be a coupling constant and consider the relative partition functions
\begin{displaymath}
Z^{\varphi^{4}}_{g} = 
\E
\Big[
\exp\Big(
-g\int_{\widetilde{\LG}}\tilde{\phi}(x)^{4}\,\tilde{m}(dx)
\Big)
\Big],
\qquad
Z^{\varphi^{4}}_{g,\sigma} = 
\E
\Big[
\exp\Big(
-g\int_{\widetilde{\LG}}\tilde{\phi}_{\sigma}(x)^{4}\,\tilde{m}(dx)
\Big)
\Big].
\end{displaymath}
Let $\tilde{\rho}$ be the field on $\widetilde{\LG}$
with density
\begin{displaymath}
\dfrac{1}{Z^{\varphi^{4}}_{g}}
\exp\Big(
-g\int_{\widetilde{\LG}}\tilde{\varphi}(x)^{4}\,\tilde{m}(dx)
\Big)
\end{displaymath}
with respect to the GFF $\tilde{\phi}$.
This is the $\varphi^{4}$ field on the metric graph.
Similarly, let $\tilde{\rho}_{\sigma}$ be the field on $\widetilde{\LG}$
with density
\begin{displaymath}
\dfrac{1}{Z^{\varphi^{4}}_{g,\sigma}}
\exp\Big(
-g\int_{\widetilde{\LG}}\tilde{\varphi}(x)^{4}\,\tilde{m}(dx)
\Big)
\end{displaymath}
with respect to the GFF $\tilde{\phi}_{\sigma}$.
This is the $\sigma$-twisted $\varphi^{4}$ field.
Then Theorem \ref{Thm main 1} implies that
\begin{equation}
\label{Eq ratio phi 4}
\mathbb{P}(\tilde{\rho}\in \LT_{\sigma})
=
\dfrac{\det((G_{\sigma}(x,y))_{x,y\in V_{\rm int}})^{1/2} 
Z^{\varphi^{4}}_{g,\sigma}}
{\det((G(x,y))_{x,y\in V_{\rm int}})^{1/2} Z^{\varphi^{4}}_{g}},
\end{equation}
and that conditionally on $\tilde{\rho}\in \LT_{\sigma}$,
the field $\vert \tilde{\rho}\vert$ is distributed as
$\vert \tilde{\rho}_{\sigma}\vert$.
In \eqref{Eq ratio phi 4} one again gets a ratio of partition functions,
this time for the $\varphi^{4}$ field and its $\sigma$-twisted version.
One can replace the interaction potential $\varphi^{4}$
by any other potential $\mathcal{V}(\vert \varphi\vert)$ bounded from below and 
depending only on the absolute value of the field.
\end{rem}

The rest of this section is dedicated to the proof of Theorem \ref{Thm main 1}.
Given a gauge field $\sigma\in\{-1,1\}^{E}$, we will denote
\begin{displaymath}
E_{\sigma,-} = 
\{e\in E\vert \sigma(e) = -1\}.
\end{displaymath}
The proof will be done under the following additional assumption.
\begin{assume}
\label{Assume pin}
Every connected component of the subgraph $(V, E\setminus E_{\sigma,-})$
of $\LG$ intersects $V_{\partial}$.
\end{assume}
The above assumption is needed because in the proof we will consider GFFs on subgraphs of $\LG$ (and sub-metric-graphs of $\widetilde{\LG}$), and thus each connected component of the subgraph needs to be connected to the boundary $V_{\partial}$ in order to pin the corresponding field.
The above assumption is however not restrictive at all.
Indeed, if a couple of graph and gauge field $(\LG,\sigma)$
does not satisfy Assumption \ref{Assume pin},
then one can enlarge $\LG$ by adding an extra site $\dagger$ to $V_{\partial}$
and connect $\dagger$ to every vertex in $x\in V_{\rm int}$ by an edge
$\{x,\dagger\}$, with a conductance
$C(x,\dagger)=\delta>0$;
and extend the gauge field $\sigma$ with
$\sigma(x,\dagger) = 1$ for every $x\in V_{\rm int}$.
Then the extended graph and gauge field now satisfy Assumption \ref{Assume pin}.
Further, if one has Theorem  \ref{Thm main 1} for the extended graph and gauge field,
one gets Theorem \ref{Thm main 1} for the initial graph and gauge field by letting
$\delta\to 0$ (recall that $\delta$ is the conductance to the extra boundary vertex 
$\dagger$).
In essence, adding $\dagger$ is equivalent to considering massive GFFs with a small square-mass $\delta$.

\medskip

Let be
\begin{displaymath}
\underline{r} = \min_{e\in E} C(e)^{-1}.
\end{displaymath}
Given $e\in E$, $x_{e}^{\rm m}$ will denote the middle point of the edge-line
$I_{e}$:
\begin{displaymath}
x_{e}^{\rm m} = x_{e, C(e)^{-1}/2}.
\end{displaymath}
For $\varepsilon \leq C(e)^{-1}/2$, we will denote by
$x_{e,\varepsilon}^{-}$ and $x_{e,\varepsilon}^{+}$ the two points of
$I_{e}$ at distance $\varepsilon$ from $x_{e}^{\rm m}$:
\begin{displaymath}
x_{e,\varepsilon}^{-} = x_{e, C(e)^{-1}/2 + \varepsilon},
\qquad 
x_{e,\varepsilon}^{+} = x_{e, C(e)^{-1}/2 - \varepsilon}.
\end{displaymath}
The points $x_{e,\varepsilon}^{-}$ and $x_{e,\varepsilon}^{+}$
will play symmetric roles, the distinction is just for the notations.
Let $J_{e,\varepsilon}$ denote the open subinterval of $I_{e}$
delimited by $x_{e,\varepsilon}^{-}$ and $x_{e,\varepsilon}^{+}$:
\begin{displaymath}
J_{e,\varepsilon} = 
\{ x_{e,u} \vert C(e)^{-1}/2 - \varepsilon < u < C(e)^{-1}/2 + \varepsilon\}.
\end{displaymath}
See Figure \ref{Fig interval}.

\begin{figure}
\includegraphics[scale=0.5]{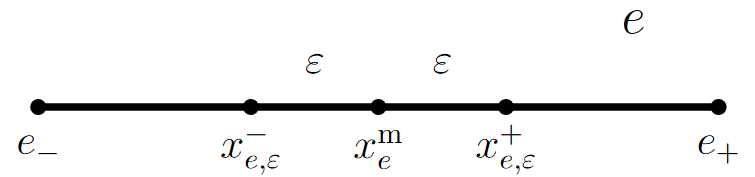}
\caption{Inside an edge-line.}
\label{Fig interval}
\end{figure}

For $\varepsilon\in (0,\underline{r}/2)$, 
let $\widetilde{\LG}_{\sigma,\varepsilon}$ denote
\begin{displaymath}
\widetilde{\LG}_{\sigma,\varepsilon} = 
\widetilde{\LG}\setminus \bigcup_{e\in E_{\sigma,-}}J_{e,\varepsilon}.
\end{displaymath}
In this way, $\widetilde{\LG}_{\sigma,\varepsilon}$ is a closed subset of the metric graph $\widetilde{\LG}$, and is itself a metric graph, 
but not necessarily connected.
Consider the discrete graph 
$\LG_{\sigma,\varepsilon} = (V_{\sigma,\varepsilon},E_{\sigma,\varepsilon})$
constructed as follows.
The set of vertices is given by
\begin{displaymath}
V_{\sigma,\varepsilon} = 
V \cup \{x_{e,\varepsilon}^{-}\vert e\in E_{\sigma,-}\}
\cup\{x_{e,\varepsilon}^{+}\vert e\in E_{\sigma,-}\},
\end{displaymath}
and the set of edges given by
\begin{displaymath}
E_{\sigma,\varepsilon} =
(E\setminus  E_{\sigma,-})\cup
\{\{x_{e,\varepsilon}^{-},e_{-}\}\vert e\in E_{\sigma,-}\}
\cup
\{\{x_{e,\varepsilon}^{+},e_{+}\}\vert e\in E_{\sigma,-}\}.
\end{displaymath}
The graph $\LG_{\sigma,\varepsilon}$ is endowed with the following conductances:
for $e\in E\setminus  E_{\sigma,-}$, we keep the same conductance
$C(e)$. 
Further, for $e\in E_{\sigma,-}$, we set
\begin{displaymath}
C(x_{e,\varepsilon}^{-},e_{-}) = 
C(x_{e,\varepsilon}^{+},e_{+})
= (C(e)^{-1}/2 - \varepsilon)^{-1}.
\end{displaymath}
Then $\widetilde{\LG}_{\sigma,\varepsilon}$ is actually the metric graph associated to the electrical network $\LG_{\sigma,\varepsilon}$.
We will denote by $\tilde{\phi}_{\sigma,\varepsilon}$
the metric graph GFF on $\widetilde{\LG}_{\sigma,\varepsilon}$
with $0$ boundary conditions on $V_{\partial}$.
Assumption \ref{Assume pin} ensures that $\tilde{\phi}_{\sigma,\varepsilon}$
is well defined.
The field $\tilde{\phi}_{\sigma,\varepsilon}$ is a usual metric graph GFF,
i.e. \textbf{not} gauge-twisted.
We will also consider the restrictions of the metric graph
GFFs $\tilde{\phi}$ and $\tilde{\phi}_{\sigma}$ to
$\widetilde{\LG}_{\sigma,\varepsilon}$.

For $x,y\in \widetilde{\LG}$, we will denote
\begin{displaymath}
G(x,y) = \E[\tilde{\phi}(x),\tilde{\phi}(y)],
\qquad
G_{\sigma}(x,y) = \E[\tilde{\phi}_{\sigma}(x),\tilde{\phi}_{\sigma}(y)].
\end{displaymath}
In this way we have a natural extension of Green's functions 
$G$ and $G_{\sigma}$ beyond $V\times V$.
note that $G$ defined in this way is continuous on 
$\widetilde{\LG}\times \widetilde{\LG}$, but 
$G_{\sigma}$ has discontinuities at the middle points 
$x_{e}^{\rm m}$ for $e\in E_{\sigma,-}$.
We will also denote by $G_{\sigma,\varepsilon}(x,y)$ the
covariance kernel of $\tilde{\phi}_{\sigma,\varepsilon}$ on 
$\widetilde{\LG}_{\sigma,\varepsilon}\times \widetilde{\LG}_{\sigma,\varepsilon}$.

\begin{lemma}
\label{Lem abs cont}
The restrictions $\tilde{\phi}_{\vert\widetilde{\LG}_{\sigma,\varepsilon}}$ and 
$\tilde{\phi}_{\sigma \vert \widetilde{\LG}_{\sigma,\varepsilon}}$
are both absolutely continuous with respect to $\tilde{\phi}_{\sigma,\varepsilon}$.
The Radon-Nikodym derivatives are given by 
\begin{equation}
\label{Eq density phi eps}
\dfrac{d\mathbb{P}_{\tilde{\phi}_{\vert\widetilde{\LG}_{\sigma,\varepsilon}}}}
{d\mathbb{P}_{\tilde{\phi}_{\sigma,\varepsilon}}}
(\tilde{\varphi})
=
\dfrac{(\det (G_{\sigma,\varepsilon}(x,y))_{x,y\in V_{\sigma,\varepsilon}\setminus V})^{1/2}}
{(\det (G(x,y))_{x,y\in V_{\sigma,\varepsilon}\setminus V})^{1/2}}
\exp\Big(
-
\dfrac{1}{4\varepsilon}\sum_{e\in E_{\sigma,-}}
(\tilde{\varphi}(x_{e,\varepsilon}^{+})-\tilde{\varphi}(x_{e,\varepsilon}^{-}))^{2}
\Big),
\end{equation}
\begin{equation}
\label{Eq density phi sigma eps}
\dfrac{d\mathbb{P}_{\tilde{\phi}_{\sigma\vert\widetilde{\LG}_{\sigma,\varepsilon}}}}
{d\mathbb{P}_{\tilde{\phi}_{\sigma,\varepsilon}}}
(\tilde{\varphi})
=
\dfrac{(\det (G_{\sigma,\varepsilon}(x,y))_{x,y\in V_{\sigma,\varepsilon}\setminus V})^{1/2}}
{(\det (G_{\sigma}(x,y))_{x,y\in V_{\sigma,\varepsilon}\setminus V})^{1/2}}
\exp\Big(
-
\dfrac{1}{4\varepsilon}\sum_{e\in E_{\sigma,-}}
(\tilde{\varphi}(x_{e,\varepsilon}^{+})+\tilde{\varphi}(x_{e,\varepsilon}^{-}))^{2}
\Big).
\end{equation}
\end{lemma}

\begin{proof}
Let us introduce the auxiliary electrical network
$\LG_{\sigma,\varepsilon}^{\ast}$ as follows.
Its set of vertices is $V_{\sigma,\varepsilon}$, the same as for
$\LG_{\sigma,\varepsilon}$.
Its set of edges is
\begin{displaymath}
E_{\sigma,\varepsilon}^{\ast} = 
E_{\sigma,\varepsilon}\cup
\{\{x_{e,\varepsilon}^{-},x_{e,\varepsilon}^{+}\}\vert e\in E_{\sigma,-} \}.
\end{displaymath}
The conductances on $E_{\sigma,\varepsilon}^{\ast}$ are as follows.
If $e\in E_{\sigma,\varepsilon}$, then we keep the same conductance $C(e)$.
Moreover, we set
\begin{displaymath}
C(x_{e,\varepsilon}^{-},x_{e,\varepsilon}^{+})
= \dfrac{1}{2\varepsilon},
\end{displaymath}
for $e\in E_{\sigma,-}$.
Then the metric graph associated to the elctrical network
$\LG_{\sigma,\varepsilon}^{\ast}$ is again 
$\widetilde{\LG}$, the same as for $\LG$.
This is because for every $e\in E_{\sigma,-}$,
\begin{multline*}
C(e_{-},x_{e,\varepsilon}^{-})^{-1}
+ C(x_{e,\varepsilon}^{-},x_{e,\varepsilon}^{+})^{-1}
+ C(x_{e,\varepsilon}^{+},e_{+})^{-1}
\\
= (C(e)^{-1}/2 - \varepsilon) + 2\varepsilon + (C(e)^{-1}/2 - \varepsilon)
= C(e)^{-1},
\end{multline*}
as for resistors in series.
Therefore, the restriction of $\tilde{\phi}$ to 
$V_{\sigma,\varepsilon}$ is the discrete GFF on the electrical network 
$\LG_{\sigma,\varepsilon}^{\ast}$,
with $0$ boundary conditions on $V_{\partial}$.
The right-hand side of \eqref{Eq density phi eps}
is then the density of $(\tilde{\phi}(x))_{x\in V_{\sigma,\varepsilon}}$
with respect to 
$(\tilde{\phi}_{\sigma,\varepsilon}(x))_{x\in V_{\sigma,\varepsilon}}$,
both being discrete GFFs, on $\LG_{\sigma,\varepsilon}^{\ast}$ and
$\LG_{\sigma,\varepsilon}$ respectively.
Indeed, compared to $\LG_{\sigma,\varepsilon}$, 
$\LG_{\sigma,\varepsilon}^{\ast}$ contains only the extra edges
$\{x_{e,\varepsilon}^{-},x_{e,\varepsilon}^{+}\}$ for $e\in E_{\sigma,-}$,
with conductance $1/(2\varepsilon)$ each, 
so the corresponding factors appear in the density.
The ratio of determinants of Green's functions to the power $1/2$ is the normalization factor. 
Finally, the expression of the density is the same when we consider the whole metric graph $\widetilde{\LG}_{\sigma,\varepsilon}$,
because both for $\tilde{\phi}$ and $\tilde{\phi}_{\sigma,\varepsilon}$
we interpolate with the same independent Brownian bridges. 

Now let us deal with $\tilde{\phi}_{\sigma}$.
We consider on $\LG_{\sigma,\varepsilon}^{\ast}$
the gauge-field $\sigma^{\ast}\in\{-1,1\}^{E_{\sigma,\varepsilon}^{\ast}}$
defined as follows.
For $e\in E\setminus E_{\sigma,-}$,
$\sigma^{\ast}(e) = \sigma(e) = 1$.
Further,
\begin{displaymath}
\sigma^{\ast}(e_{-},x_{e,\varepsilon}^{-})
=
\sigma^{\ast}(e_{+},x_{e,\varepsilon}^{+})
= 1,
\end{displaymath}
and
\begin{displaymath}
\sigma^{\ast}(x_{e,\varepsilon}^{-},x_{e,\varepsilon}^{+}) = -1.
\end{displaymath}
In this way, for every $\wp$
nearest-neighbor path in $\LG_{\sigma,\varepsilon}^{\ast}$
joining two points, $x,y\in V$,
\begin{displaymath}
\hol^{\sigma^{\ast}}(\wp) = 
\hol^{\sigma}(\wp_{\LG}),
\end{displaymath}
where $\wp_{\LG}$ is the nearest-neighbor path in $\LG$ from $x$ to $y$
obtained by taking the trace of $\wp$ on $V$.
The restriction of $\tilde{\phi}_{\sigma}$ to $V_{\sigma,\varepsilon}$
is the $\sigma^{\ast}$-twisted GFF on $\LG_{\sigma,\varepsilon}^{\ast}$
with $0$ boundary conditions on $V_{\partial}$.
This is similar to Proposition \ref{Prop sigma N metric}.
Then \eqref{Eq density phi sigma eps} is just the density of
 $(\tilde{\phi}_{\sigma}(x))_{x\in V_{\sigma,\varepsilon}}$
with respect to 
$(\tilde{\phi}_{\sigma,\varepsilon}(x))_{x\in V_{\sigma,\varepsilon}}$.
Because of $\sigma^{\ast}$,
there is a $+$ instead of a $-$ in 
\eqref{Eq density phi sigma eps} compared to
\eqref{Eq density phi eps}.
\end{proof}

Note that Lemma \ref{Lem abs cont} implies in particular that
$\tilde{\phi}_{\sigma \vert \widetilde{\LG}_{\sigma,\varepsilon}}$ is absolutely continuous with respect to 
$\tilde{\phi}_{\vert \widetilde{\LG}_{\sigma,\varepsilon}}$.
However, if one considers the whole metric graph $\widetilde{\LG}$,
then $\tilde{\phi}_{\sigma}$ is singular with respect to 
$\tilde{\phi}$, unless $E_{\sigma,-}=\emptyset$.
For instance, $\tilde{\phi}_{\sigma}$ is a.s. discontinuous at the middle points
$x_{e}^{\rm m}$ for $e\in E_{\sigma,-}$.
So that is the reason why we restrict to $\varepsilon$-far away from these middle points.

Let $\mathcal{C}(\widetilde{\LG}_{\sigma,\varepsilon})$ denote the space of
continuous real-valued functions on $\widetilde{\LG}_{\sigma,\varepsilon}$.
For $f\in \mathcal{C}(\widetilde{\LG}_{\sigma,\varepsilon})$,
$\{ f\neq 0\}$ will denote the non-zero set of $f$,
with $\{ f\neq 0\}\subset\widetilde{\LG}_{\sigma,\varepsilon}$.
Let $\LT_{\sigma,\varepsilon}$ be the following subset of 
$\mathcal{C}(\widetilde{\LG}_{\sigma,\varepsilon})$:
\begin{displaymath}
\LT_{\sigma,\varepsilon} = 
\Big\{
f \in \mathcal{C}(\widetilde{\LG}_{\sigma,\varepsilon})
\Big\vert~\forall~ U
\text{ connected component of }
\Big(\{ f\neq 0\}\bigcup_{e\in E_{\sigma,-}} J_{e,\varepsilon}\Big),
\bpi_{\sigma}^{-1}(U)
\text{ is \textbf{not} connected}
\Big\}.
\end{displaymath}
In the above definition, we enlarge $\{ f\neq 0\}$ with the 
intervals $J_{e,\varepsilon}$ for $e\in E_{\sigma,-}$,
and then consider the connected components.

Given $f\in \mathcal{C}(\widetilde{\LG}_{\sigma,\varepsilon})$,
we will define an equivalence relation $\sim_{f}$ on the set
\begin{equation}
\label{Eq x e pm}
\{x^{-}_{e,\varepsilon}\vert e\in E_{\sigma,-}\}\cup
\{x^{+}_{e,\varepsilon}\vert e\in E_{\sigma,-}\}
=V_{\sigma,\varepsilon}\setminus V.
\end{equation}
Two points $x,x'$ in the above set are in the same class if they are in the same connected component of $\{ f\neq 0\}$.
If $f(x)=0$, then $x$ is alone in its class.
We will denote by $\mathtt{V}_{f}$ the set \eqref{Eq x e pm}
quotiented by the equivalence relation $\sim_{f}$,
and by $[x]_{f}$ the equivalence class of a point $x$.
Next we will introduce an undirected multigraph $\mathtt{\Gamma}_{f}$ 
induced by the function $f$.
Note that in general $\mathtt{\Gamma}_{f}$ is not connected
and may have self-loops and multiple edges.
The set of vertices of $\mathtt{\Gamma}_{f}$ is the quotient set
$\mathtt{V}_{f}$.
Further, for each $e\in E_{\sigma,-}$, 
add to the multigraph $\mathtt{\Gamma}_{f}$
and edge with the ends $[x^{-}_{e,\varepsilon}]_{f}$ and 
$[x^{+}_{e,\varepsilon}]_{f}$.
So the number of edges of $\mathtt{\Gamma}_{f}$ is $\vert E_{\sigma,-}\vert$.

\begin{lemma}
\label{Lem bipartite graph}
Let $f\in \mathcal{C}(\widetilde{\LG}_{\sigma,\varepsilon})$.
Then the multigraph $\mathtt{\Gamma}_{f}$ is bipartite,
i.e. does not contain cycles with odd number of edges,
if and only if $f\in \LT_{\sigma,\varepsilon}$.
\end{lemma}

\begin{proof}
Given a continuous loop $\wp$ in the subset
\begin{equation}
\label{Eq f neq 0 augmented}
\Big(\{ f\neq 0\}\bigcup_{e\in E_{\sigma,-}} J_{e,\varepsilon}\Big),
\end{equation}
it induces a nearest-neighbor loop in the multigraph $\mathtt{\Gamma}_{f}$,
which we will denote by $[\wp]_{f}$ by a slight abuse of notation.
If $\wp$ does not visit any point in $V_{\sigma,\varepsilon}\setminus V$
\eqref{Eq x e pm},
then $[\wp]_{f}$ is just the empty loop.
Conversely, every nearest-neighbor loop in $\mathtt{\Gamma}_{f}$
can be lifted to a continuous loop in \eqref{Eq f neq 0 augmented}.
This follows from the way $\mathtt{\Gamma}_{f}$ has been constructed.
The edges visited by $[\wp]_{f}$ correspond to the intervals
$J_{e,\varepsilon}$ (with $e\in E_{\sigma,-}$) crossed by $\wp$ from one end to the opposite.
The holonomy $\hol^{\sigma}(\wp)$ is given by the parity of the number of
edge-lines $I_{e}$, for $e\in E_{\sigma,-}$,
crossed by $\wp$.
This parity is the same as for the crossings of $J_{e,\varepsilon}$,
for $e\in E_{\sigma,-}$, although the number of crossings itself might be different.
This is because each time $\wp$ crosses an edge-line $I_{e}$,
it will cross the corresponding subinterval $J_{e,\varepsilon}$
an odd number of times.
Thus, $\hol^{\sigma}(\wp)$ is also given by the parity of the number of edges of
multigraph loop $[\wp]_{f}$.
Thus, $\mathtt{\Gamma}_{f}$ is bipartite if and only if the subset
\eqref{Eq f neq 0 augmented} does not contain loops $\wp$ with
$\hol^{\sigma}(\wp) = -1$, 
which is the same as $f\in \LT_{\sigma,\varepsilon}$.
\end{proof}

Note that although $\mathcal{C}(\widetilde{\LG}_{\sigma,\varepsilon})$
is infinite, the number of different partitions of
$V_{\sigma,\varepsilon}\setminus V$ \eqref{Eq x e pm} 
induced by the functions $f$ can be only finite.
Given $f\in \LT_{\sigma,\varepsilon}$, 
one can choose a ``coloring''
$\bic_{f}: V_{\sigma,\varepsilon}\setminus V
\rightarrow \{ -1,1\}$ satisfying the following two properties.
\begin{enumerate}
\item The value of $\bic_{f}$ is the same across every
equivalence class for $\sim_{f}$.
\item For every $e\in E_{\sigma,-}$,
$\bic_{f}(x^{-}_{e,\varepsilon}) =  - \bic_{f}(x^{+}_{e,\varepsilon})$.
\end{enumerate}
The existence of such $\bic_{f}$ is ensured by the fact that
$\mathtt{\Gamma}_{f}$ is bipartite.
Moreover, one can choose $\bic_{f}$ to depend only on the partition
$\mathtt{V}_{f}$ induced by $\sim_{f}$.

Next we define a transformation $\TT_{\sigma,\varepsilon}$
acting on $\LT_{\sigma,\varepsilon}$.
Given $f\in \LT_{\sigma,\varepsilon}$,
the function $\TT_{\sigma,\varepsilon}f$ is defined as follows.
\begin{enumerate}
\item For every $x\in \widetilde{\LG}_{\sigma,\varepsilon}$
such that $f(x)=0$, we also have $(\TT_{\sigma,\varepsilon}f)(x)=0$.
\item On every connected component $U$ of $\{ f\neq 0\}$ such that $U$ does not intersect the set $V_{\sigma,\varepsilon}\setminus V$,
we have $(\TT_{\sigma,\varepsilon}f)_{\vert U} = f_{\vert U}$. 
\item On every connected component $U$ of $\{ f\neq 0\}$ intersecting 
$V_{\sigma,\varepsilon}\setminus V$,
we have $(\TT_{\sigma,\varepsilon}f)_{\vert U} = \bic_{f}(U)f_{\vert U}$,
where $\bic_{f}(U)$ is the common value of $\bic_{f}$ on $U$.
\end{enumerate}
So, in essence, $\TT_{\sigma,\varepsilon}f$ is obtained by flipping the signs
on some connected components of $\{ f\neq 0\}$.
In this way, $\TT_{\sigma,\varepsilon}f$ is continuous too, and
$\vert \TT_{\sigma,\varepsilon}f\vert \equiv f$.
Therefore, $\TT_{\sigma,\varepsilon}f\in \LT_{\sigma,\varepsilon}$ too.
Moreover, the equivalence relation 
$\sim_{\TT_{\sigma,\varepsilon}f}$ is the same as 
$\sim_{f}$.
Thus, 
$\TT_{\sigma,\varepsilon}^{2}f=\TT_{\sigma,\varepsilon}(\TT_{\sigma,\varepsilon}f) 
= f$.

Further, we will extend $\TT_{\sigma,\varepsilon}$ to the whole
$\mathcal{C}(\widetilde{\LG}_{\sigma,\varepsilon})$.
For $f\in \mathcal{C}(\widetilde{\LG}_{\sigma,\varepsilon})\setminus 
\LT_{\sigma,\varepsilon}$,
we set $\TT_{\sigma,\varepsilon}f = f$.

Let us denote by $\tilde{\xi}_{\varepsilon}$ and $\tilde{\eta}_{\varepsilon}$ 
the fields
\begin{displaymath}
\tilde{\xi}_{\varepsilon} = 
\TT_{\sigma,\varepsilon}(\tilde{\phi}_{\vert\widetilde{\LG}_{\sigma,\varepsilon}}),
\qquad 
\tilde{\eta}_{\varepsilon} =
\TT_{\sigma,\varepsilon}(\tilde{\phi}_{\sigma,\varepsilon}).
\end{displaymath}

\begin{lemma}
\label{Lem abs cont 2}
The field $\tilde{\xi}_{\varepsilon}$ is absolutely continuous with respect to 
$\tilde{\eta}_{\varepsilon}$, and the corresponding density is
\begin{displaymath}
\dfrac{d\mathbb{P}_{\tilde{\xi}_{\varepsilon}}}
{d\mathbb{P}_{\tilde{\eta}_{\varepsilon}}}
(\tilde{\varphi})
=
\dfrac{(\det (G_{\sigma,\varepsilon}(x,y))_{x,y\in V_{\sigma,\varepsilon}\setminus V})^{1/2}}
{(\det (G(x,y))_{x,y\in V_{\sigma,\varepsilon}\setminus V})^{1/2}}
\exp\Big(
-
\dfrac{1}{4\varepsilon}\sum_{e\in E_{\sigma,-}}
(\tilde{\varphi}(x_{e,\varepsilon}^{+})
-(-1)^{\ind{\tilde{\varphi}\in \LT_{\sigma,\varepsilon}}}
\tilde{\varphi}(x_{e,\varepsilon}^{-}))^{2}
\Big).
\end{displaymath}
\end{lemma}

\begin{proof}
Since $\tilde{\phi}_{\vert\widetilde{\LG}_{\sigma,\varepsilon}}$
is absolutely continuous with respect to $\tilde{\phi}_{\sigma,\varepsilon}$
(Lemma \ref{Lem abs cont})
and $\TT_{\sigma,\varepsilon}$ is a deterministic and measurable transformation,
we get that  $\tilde{\xi}_{\varepsilon}$ is absolutely continuous with respect to 
$\tilde{\eta}_{\varepsilon}$, with
\begin{displaymath}
\dfrac{d\mathbb{P}_{\tilde{\xi}_{\varepsilon}}}
{d\mathbb{P}_{\tilde{\eta}_{\varepsilon}}}
(\tilde{\varphi})
=
\dfrac{d\mathbb{P}_{\tilde{\phi}_{\vert\widetilde{\LG}_{\sigma,\varepsilon}}}}
{d\mathbb{P}_{\tilde{\phi}_{\sigma,\varepsilon}}}
(\TT_{\sigma,\varepsilon}\tilde{\varphi}).
\end{displaymath}
Further we use the expression \eqref{Eq density phi eps}.
Take $e\in E_{\sigma,-}$.
For $\tilde{\varphi}\not\in \LT_{\sigma,\varepsilon}$,
\begin{displaymath}
((\TT_{\sigma,\varepsilon}\tilde{\varphi})(x_{e,\varepsilon}^{+})
-
(\TT_{\sigma,\varepsilon}\tilde{\varphi})(x_{e,\varepsilon}^{-}))^{2}
=
(\tilde{\varphi}(x_{e,\varepsilon}^{+})
-
\tilde{\varphi}(x_{e,\varepsilon}^{-}))^{2}.
\end{displaymath}
For $\tilde{\varphi}\in \LT_{\sigma,\varepsilon}$,
\begin{displaymath}
((\TT_{\sigma,\varepsilon}\tilde{\varphi})(x_{e,\varepsilon}^{+})
-
(\TT_{\sigma,\varepsilon}\tilde{\varphi})(x_{e,\varepsilon}^{-}))^{2}
=
(\tilde{\varphi}(x_{e,\varepsilon}^{+})
+
\tilde{\varphi}(x_{e,\varepsilon}^{-}))^{2}.
\end{displaymath}
This because 
$\bic_{f}(x^{-}_{e,\varepsilon})\bic_{f}(x^{+}_{e,\varepsilon}) = -1$.
\end{proof}

\begin{lemma}
\label{Lem same law}
The field $\tilde{\eta}_{\varepsilon} =
\TT_{\sigma,\varepsilon}(\tilde{\phi}_{\sigma,\varepsilon})$
has the same distribution as $\tilde{\phi}_{\sigma,\varepsilon}$.
\end{lemma}

\begin{proof}
First, by construction,
$\vert\TT_{\sigma,\varepsilon}(\tilde{\phi}_{\sigma,\varepsilon})\vert
\equiv \vert \tilde{\phi}_{\sigma,\varepsilon}\vert$.
Further, the signs $\bic_{\tilde{\phi}_{\sigma,\varepsilon}}$
(that is to say $\bic_{f}$ with
$f=\tilde{\phi}_{\sigma,\varepsilon}$)
are by construction measurable with respect to 
$\vert \tilde{\phi}_{\sigma,\varepsilon}\vert$.
Then, we use the fact that the signs of $\tilde{\phi}_{\sigma,\varepsilon}$
(do not confuse them with $\bic_{\tilde{\phi}_{\sigma,\varepsilon}}$)
are distributed, conditionally on $\vert \tilde{\phi}_{\sigma,\varepsilon}\vert$,
as independent uniform $\{-1,1\}$-valued r.v.s, 
one for each connected component of
$\{\tilde{\phi}_{\sigma,\varepsilon}\neq 0\}$;
see \cite[Lemma 3.2]{Lupu2016Iso}.
This implies that the product of signs
$\bic_{\tilde{\phi}_{\sigma,\varepsilon}}
\operatorname{sign}(\tilde{\phi}_{\sigma,\varepsilon})$
has the same conditional distribution given 
$\vert \tilde{\phi}_{\sigma,\varepsilon}\vert$
as $\operatorname{sign}(\tilde{\phi}_{\sigma,\varepsilon})$.
Therefore, $\TT_{\sigma,\varepsilon}(\tilde{\phi}_{\sigma,\varepsilon})$
has the same distribution as $\tilde{\phi}_{\sigma,\varepsilon}$.
\end{proof}

\begin{lemma}
\label{Lem const ratio}
The function
\begin{displaymath}
\varepsilon\mapsto
\dfrac{(\det (G_{\sigma}(x,y))_{x,y\in V_{\sigma,\varepsilon}\setminus V})^{1/2}}
{(\det (G(x,y))_{x,y\in V_{\sigma,\varepsilon}\setminus V})^{1/2}}
\end{displaymath}
is constant on $(0,\underline{r}/2)$.
More precisely,
for every $\varepsilon\in (0,\underline{r}/2)$,
\begin{eqnarray*}
\dfrac{(\det (G_{\sigma}(x,y))_{x,y\in V_{\sigma,\varepsilon}\setminus V})^{1/2}}
{(\det (G(x,y))_{x,y\in V_{\sigma,\varepsilon}\setminus V})^{1/2}}
&=&
\dfrac{(\det (G_{\sigma}(x,y))_{x,y\in V_{\rm int}})^{1/2}}
{(\det (G(x,y))_{x,y\in V_{\rm int}})^{1/2}}
\\
&=&
\exp\big(-\mu^{\rm loop}(\{\text{Loops } \wp \text{ with } \hol^{\sigma}(\wp)=-1\})\big).
\end{eqnarray*}
\end{lemma}

\begin{proof}
Consider the measure $\tilde{\mu}^{\rm loop}$ \eqref{Eq mu metric}
on Brownian loops on the metric graph $\widetilde{\LG}$.
Similarly to Corollary \eqref{Cor ratio det},
we have that
\begin{displaymath}
\dfrac{(\det (G_{\sigma}(x,y))_{x,y\in V_{\sigma,\varepsilon}\setminus V})^{1/2}}
{(\det (G(x,y))_{x,y\in V_{\sigma,\varepsilon}\setminus V})^{1/2}}
=
\exp\big(-\tilde{\mu}^{\rm loop}(\{\text{Loops } \wp \text{ visiting }
V_{\sigma,\varepsilon}\setminus V \text{ and with } \hol^{\sigma}(\wp)=-1
\})\big).
\end{displaymath}
Moreover, 
\begin{eqnarray*}
\dfrac{\det((G_{\sigma}(x,y))_{x,y\in V_{\rm int}})^{1/2}}
{\det((G(x,y))_{x,y\in V_{\rm int}})^{1/2}}
&=&
\exp\big(-\mu^{\rm loop}(\{\text{Loops } \wp \text{ with } \hol^{\sigma}(\wp)=-1\})\big)
\\&=&
\exp\big(-\tilde{\mu}^{\rm loop}(\{\text{Loops } \wp \text{ visiting }
 V \text{ and with } \hol^{\sigma}(\wp) = -1
\})\big),
\end{eqnarray*}
where the second equality is due to the fact that the measure on discrete loops
$\mu^{\rm loop}$ can be obtained by taking the trace on $V$ of metric graph
loops under $\tilde{\mu}^{\rm loop}$.
We conclude by using the fact that any continuous loop $\wp$ on $\widetilde{\LG}$
with $\hol^{\sigma}(\wp) = -1$ has to visit both $V$ and 
$V_{\sigma,\varepsilon}\setminus V$. 
That is to say a loop not visiting either of two subsets has holonomy $1$,
because then it cannot cross any of the edge-lines $I_{e}$ for
$e\in E_{\sigma,-}$.
Therefore,
\begin{multline*}
\tilde{\mu}^{\rm loop}(\{\text{Loops } \wp \text{ visiting }
V_{\sigma,\varepsilon}\setminus V \text{ and with } \hol^{\sigma}(\wp)=-1\})
=
\tilde{\mu}^{\rm loop}(\{\text{Loops } \wp \text{ with } \hol^{\sigma}(\wp)=-1\})
\\
=
\tilde{\mu}^{\rm loop}(\{\text{Loops } \wp \text{ visiting }
 V \text{ and with } \hol^{\sigma}(\wp) = -1\})
=
\mu^{\rm loop}(\{\text{Loops } \wp \text{ with } \hol^{\sigma}(\wp)=-1\}).
\qedhere
\end{multline*}
\end{proof}

\begin{proof}[Proof of Theorem \ref{Thm main 1}]
By Lemma \ref{Lem abs cont},
the field $\tilde{\phi}_{\sigma \vert \widetilde{\LG}_{\sigma,\varepsilon}}$
is absolutely continuous with respect to the field 
$\tilde{\phi}_{\sigma,\varepsilon}$.
By Lemma \ref{Lem same law}, 
the field $\tilde{\eta}_{\varepsilon}$ has the same distribution
as $\tilde{\phi}_{\sigma,\varepsilon}$.
By Lemma \ref{Lem abs cont 2}, the fields
$\tilde{\eta}_{\varepsilon}$ and $\tilde{\xi}_{\varepsilon}$
are mutually absolutelly continuous, the corresponding Radon–Nikodym 
derivative being positive on $\mathcal{C}(\widetilde{\LG}_{\sigma,\varepsilon})$.
Therefore, the field 
$\tilde{\phi}_{\sigma \vert \widetilde{\LG}_{\sigma,\varepsilon}}$
is absolutely continuous with respect to the field $\tilde{\xi}_{\varepsilon}$
and the Radon–Nikodym derivative is given by
\begin{multline}
\label{Eq RN event eps}
\dfrac{d\mathbb{P}_{\tilde{\phi}_{\sigma\vert\widetilde{\LG}_{\sigma,\varepsilon}}}}
{d\mathbb{P}_{\tilde{\xi}_{\varepsilon}}}
(\tilde{\varphi})
= \\
\dfrac{(\det (G(x,y))_{x,y\in V_{\sigma,\varepsilon}\setminus V})^{1/2}}
{(\det (G_{\sigma}(x,y))_{x,y\in V_{\sigma,\varepsilon}\setminus V})^{1/2}}
\Big(\ind_{\tilde{\varphi}\in\LT_{\sigma,\varepsilon}}
+\ind_{\tilde{\varphi}\not\in\LT_{\sigma,\varepsilon}}
\exp\Big(
-\dfrac{1}{\varepsilon}\sum_{e\in E_{\sigma,-}}
\tilde{\varphi}(x_{e,\varepsilon}^{+})
\tilde{\varphi}(x_{e,\varepsilon}^{-})
\Big)
\Big).
\end{multline}
Note that $\tilde{\xi}_{\varepsilon}\in \LT_{\sigma,\varepsilon}$
if and only if 
$\tilde{\phi}_{\vert \widetilde{\LG}_{\sigma,\varepsilon}}
\in \LT_{\sigma,\varepsilon}$.
So in particular,
\begin{multline*}
\mathbb{P}(\tilde{\phi}_{\vert \widetilde{\LG}_{\sigma,\varepsilon}}
\in \LT_{\sigma,\varepsilon})
+
\E\Big[
\ind_{\tilde{\phi}_{\vert \widetilde{\LG}_{\sigma,\varepsilon}}
\not\in \LT_{\sigma,\varepsilon}}
\exp\Big(
-\dfrac{1}{\varepsilon}\sum_{e\in E_{\sigma,-}}
\tilde{\phi}(x_{e,\varepsilon}^{+})
\tilde{\phi}(x_{e,\varepsilon}^{-})
\Big)
\Big]
\\
=
\dfrac{(\det (G_{\sigma}(x,y))_{x,y\in V_{\sigma,\varepsilon}\setminus V})^{1/2}}
{(\det (G(x,y))_{x,y\in V_{\sigma,\varepsilon}\setminus V})^{1/2}}
=\exp\big(-\mu^{\rm loop}(\{\text{Loops } \wp \text{ with } \hol^{\sigma}(\wp)=-1\})\big),
\end{multline*}
where the second equality is due to Lemma \ref{Lem const ratio}.
Further,
\begin{displaymath}
\lim_{\varepsilon\to 0}
\mathbb{P}(\tilde{\phi}_{\vert \widetilde{\LG}_{\sigma,\varepsilon}}
\in \LT_{\sigma,\varepsilon})
=
\mathbb{P}(\tilde{\phi}\in \LT_{\sigma})
\end{displaymath}
and
\begin{displaymath}
\E\Big[
\ind_{\tilde{\phi}_{\vert \widetilde{\LG}_{\sigma,\varepsilon}}
\not\in \LT_{\sigma,\varepsilon}}
\exp\Big(
-\dfrac{1}{\varepsilon}\sum_{e\in E_{\sigma,-}}
\tilde{\phi}(x_{e,\varepsilon}^{+})
\tilde{\phi}(x_{e,\varepsilon}^{-})
\Big)
\Big]=
\dfrac{(\det (G_{\sigma}(x,y))_{x,y\in V_{\sigma,\varepsilon}\setminus V})^{1/2}}
{(\det (G(x,y))_{x,y\in V_{\sigma,\varepsilon}\setminus V})^{1/2}}
\mathbb{P}(
\vert\tilde{\phi}_{\sigma\vert \widetilde{\LG}_{\sigma,\varepsilon}}\vert
\not\in \LT_{\sigma,\varepsilon}),
\end{displaymath}
with
\begin{displaymath}
\lim_{\varepsilon\to 0}
\mathbb{P}(
\vert\tilde{\phi}_{\sigma\vert \widetilde{\LG}_{\sigma,\varepsilon}}\vert
\not\in \LT_{\sigma,\varepsilon})
=
\mathbb{P}(
\vert\tilde{\phi}_{\sigma}\vert\not\in \LT_{\sigma})
= 0;
\end{displaymath}
see Lemma \ref{Lem T sigma}. So we get \eqref{Eq prob LT sigma}.

Since $\vert\tilde{\xi}_{\varepsilon}\vert
=\vert\tilde{\phi}_{\vert \widetilde{\LG}_{\sigma,\varepsilon}}\vert$,
the field 
$\vert\tilde{\phi}_{\sigma\vert \widetilde{\LG}_{\sigma,\varepsilon}}\vert$
is absolutely continuous with respect to 
$\vert\tilde{\phi}_{\vert \widetilde{\LG}_{\sigma,\varepsilon}}\vert$,
and
\begin{displaymath}
\lim_{\varepsilon\to 0}
\dfrac{d\mathbb{P}_{\vert\tilde{\phi}_{\sigma\vert\widetilde{\LG}_{\sigma,\varepsilon}}\vert}}
{d\mathbb{P}_{\vert\tilde{\phi}_{\vert \widetilde{\LG}_{\sigma,\varepsilon}}\vert}}
(\tilde{\varphi})
= \dfrac{1}{\mathbb{P}(\tilde{\varphi}\in\LT_{\sigma})}
\ind_{\tilde{\varphi}\in\LT_{\sigma}},
\end{displaymath}
with convergence in $L^{1}$.
So $\mathbb{P}(\tilde{\varphi}\in\LT_{\sigma})^{-1}\ind_{\tilde{\varphi}\in\LT_{\sigma}}$ is the density of $\vert\tilde{\phi}_{\sigma}\vert$
with respect to $\vert\tilde{\phi}\vert$.
\end{proof}

Next we explain how to get the field $\tilde{\phi}_{\sigma}$
given the field $\tilde{\phi}$ conditionned on 
$\tilde{\phi}\in\LT_{\sigma}$,
that is to say how to get the signs, not just the absolute value.
In essence, one has to flip the signs on some of the connected components of
$\{\tilde{\phi}\neq 0\}\setminus\{x_{e}^{\rm m}\vert e\in E_{\sigma, -}\}$.

To avoid trivialities, we assume that the gauge field $\sigma$ is not uniformly $1$. Let us denote
\begin{displaymath}
\FM_{\sigma} = \{x_{e}^{\rm m}\vert e\in E_{\sigma, -}\}.
\end{displaymath}
We endow the subset $\widetilde{\LG}\setminus\FM_{\sigma}$ with the correspondent 
length metric $d'$, which is not the same as the distance inherited from
$\widetilde{\LG}$ since now we are not allowed to cross the points
in $\FM_{\sigma}$.
The metric space $(\widetilde{\LG}\setminus\FM_{\sigma},d')$ is not complete,
but we can consider its completion for $d'$, which is not $\widetilde{\LG}$.
This completion is
\begin{displaymath}
(\widetilde{\LG}\setminus\FM_{\sigma})\cup\FM_{\sigma}^{\pm},
\end{displaymath}
with
\begin{displaymath}
\FM_{\sigma}^{\pm} =
\{x_{e}^{\rm m,+}\vert e\in E_{\sigma, -}\}
\cup
\{x_{e}^{\rm m,-}\vert e\in E_{\sigma, -}\},
\end{displaymath}
where $x_{e}^{\rm m,-}$ and $x_{e}^{\rm m,+}$
are to be understood as left and right infinitesimal neighborhoods of 
$x_{e}^{\rm m}$:
\begin{displaymath}
x_{e}^{\rm m,-} = \lim_{\varepsilon\to 0} x^{-}_{e,\varepsilon},
\qquad
x_{e}^{\rm m,+} = \lim_{\varepsilon\to 0} x^{+}_{e,\varepsilon}.
\end{displaymath}
The fields $\tilde{\phi}$ and $\tilde{\phi}_{\sigma}$
can be both seen as continuous fields on 
$(\widetilde{\LG}\setminus\FM_{\sigma})\cup\FM_{\sigma}^{\pm}$,
with
\begin{displaymath}
\tilde{\phi}(x_{e}^{\rm m,-}) = \tilde{\phi}(x_{e}^{\rm m,+}),
\qquad 
\tilde{\phi}_{\sigma}(x_{e}^{\rm m,-}) = 
-\tilde{\phi}_{\sigma}(x_{e}^{\rm m,+}),
\qquad
e\in E_{\sigma, -}.
\end{displaymath}

A sample of $\tilde{\phi}$ induces a partition $\mathtt{V}_{\tilde{\phi}}$
of $\FM_{\sigma}^{\pm}$, where two points 
$x,x'\in \FM_{\sigma}^{\pm}$ are in the same class if they are in the same connected component of $\{\tilde{\phi}\neq 0\}$
seen as a subset of $\widetilde{\LG}\setminus\FM_{\sigma})\cup\FM_{\sigma}^{\pm}$
and not as a subset of $\widetilde{\LG}$.
We will denote by $[x]_{\tilde{\phi}}$
the equivalence class of $x$ for $x\in \FM_{\sigma}^{\pm}$.
The condition $\tilde{\phi}\in  \LT_{\sigma}$
is equivalent to $\mathtt{V}_{\tilde{\phi}}$ being bicolorable
in the following sense:
there is a map $\bic : \mathtt{V}_{\tilde{\phi}} \rightarrow \{ -1,1\}$,
such that for every $e\in E_{\sigma, -}$,
$\bic([x_{e}^{\rm m,-}]_{\tilde{\phi}}) 
= - \bic([x_{e}^{\rm m,+}]_{\tilde{\phi}})$.
This is similar to Lemma \ref{Lem bipartite graph}.
The number of different bicolorings is $2^{k}$ where $k$ is the number of
connected components of $\{\tilde{\phi}\neq 0\}$ in 
$\widetilde{\LG}$ 
(not in $(\widetilde{\LG}\setminus\FM_{\sigma})\cup\FM_{\sigma}^{\pm}$ !)
that intersect $\FM_{\sigma}$.
Indeed, there are two different colorings per such connected component, 
one being the opposite of the other.

Let $\operatorname{Bic}(\FM_{\sigma})$ denote the set of bicolorable
partitions of $\FM_{\sigma}$.
To each such partition $\mathbf{p}\in \operatorname{Bic}(\FM_{\sigma})$
we will associate (in a deterministic way)
a bicoloring $\bic_{\mathbf{p}}$.
We define the random field $\tilde{\xi}$ on $\widetilde{\LG}$ as follows.
On the event $\tilde{\phi}\not\in \LT_{\sigma}$,
we set $\tilde{\xi} = \tilde{\phi}$.
On the event $\tilde{\phi}\not\in \LT_{\sigma}$,
$\tilde{\xi}$ is defined by the following.
\begin{enumerate}
\item For every $x\in \widetilde{\LG}$,
$\vert \tilde{\xi}(x) \vert = \vert \tilde{\phi}(x) \vert$.
\item On every connected component $U$ of
$\{\tilde{\phi}\neq 0\}$, such that 
$U\cap \FM_{\sigma} = \emptyset$,
we have $\tilde{\xi}_{\vert U} =  \tilde{\phi}_{\vert U}$.
\item On every connected component $U$ of
$\{\tilde{\phi}\neq 0\}\setminus \FM_{\sigma}$ such that
$\overline{U}\cap \FM_{\sigma} \neq \emptyset$,
$\tilde{\xi}_{\vert U} = \bic_{\mathtt{V}_{\tilde{\phi}}}(U) 
\tilde{\phi}_{\vert U}$,
where $\bic_{\mathtt{V}_{\tilde{\phi}}}$ is
$\bic_{\mathbf{p}}$ with $\mathbf{p}$ being the random partition
$\mathtt{V}_{\tilde{\phi}}$, and
$\bic_{\mathtt{V}_{\tilde{\phi}}}(U)$ is the common value of 
$\bic_{\mathtt{V}_{\tilde{\phi}}}([x]_{\tilde{\phi}})$
for $x\in \overline{U}\cap \FM_{\sigma}$.
\end{enumerate}
So $\tilde{\xi}$ is obtain from $\tilde{\phi}$ through a deterministic
transformation. 
It corresponds to flipping the signs of some of the connected components 
of $\{\tilde{\phi}\neq 0\}\setminus \FM_{\sigma}$,
on the event $\tilde{\phi}\in \LT_{\sigma}$,
so as to achieve a bicoloring.
See Figure \ref{Fig bicoloring}.

\begin{figure}
\includegraphics[scale=0.48]{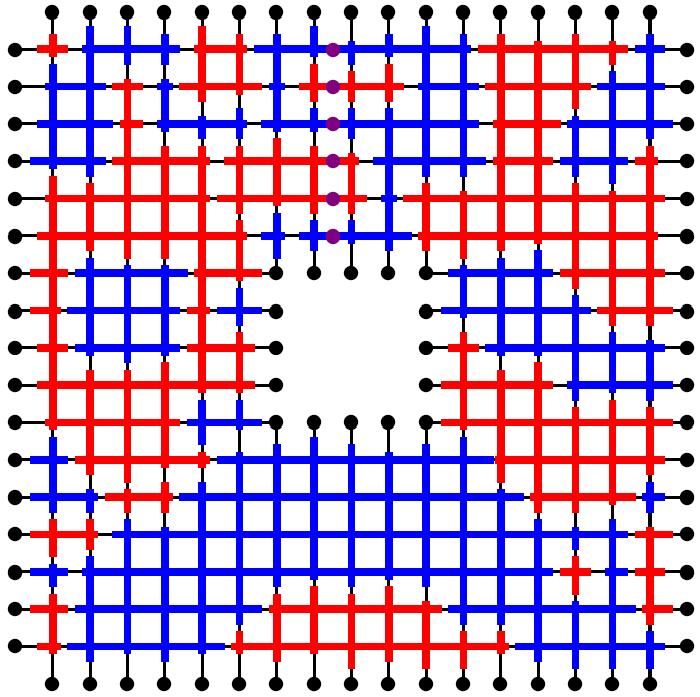}
\qquad
\includegraphics[scale=0.48]{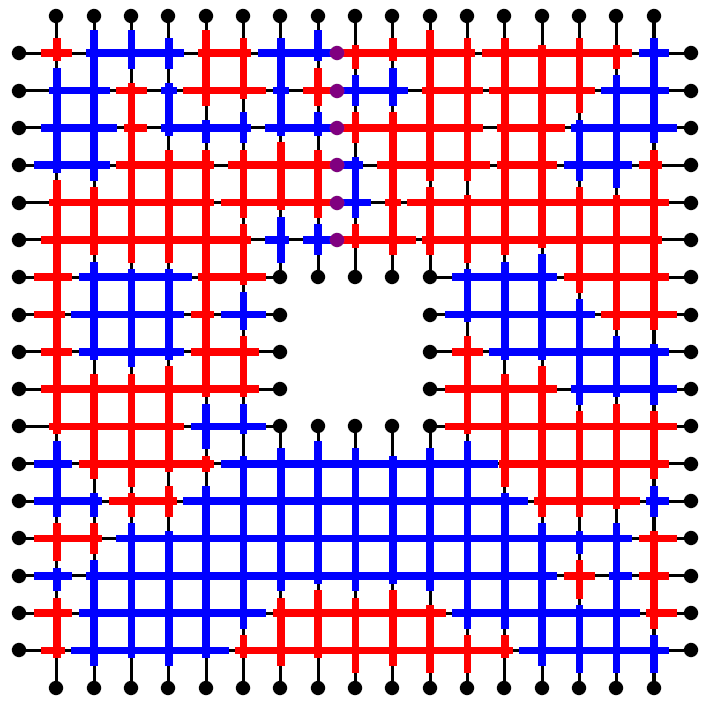}
\caption{Left: the field $\tilde{\phi}$ on the event 
$\tilde{\phi}\in \LT_{\sigma}$,
which in this example means that there are no sign clusters surrounding the inner hole of the domain.
Right: the associated $\tilde{\xi}$ field.
The black dots represent the boundary
$V_{\partial}$.
The violet dots represent $\FM_{\sigma}$.
The positive, resp. negative values of the fields are
in red, resp. blue.}
\label{Fig bicoloring}
\end{figure}

\begin{cor}
\label{Cor bicolor tilde phi}
The conditional distribution of $\tilde{\xi}$ on the event
$\tilde{\phi}\in \LT_{\sigma}$,
is that of $\tilde{\phi}_{\sigma}$.
\end{cor}

\begin{proof}
The couple $(\tilde{\phi},\tilde{\xi})$
can be obtained as limit in law as $\varepsilon\to 0$ of
$(\tilde{\phi}_{\vert \widetilde{\LG}_{\sigma,\varepsilon}},
\tilde{\xi}_{\varepsilon})$. Then the result is obtained by passing the density
\eqref{Eq RN event eps} to the limit.
\end{proof}

\subsection{Around the isomorphism for $\tilde{\phi}_{\sigma}$}
\label{Subsec iso topo}

Recall the measure on metric graph loops
$\tilde{\mu}^{\rm loop}$ \eqref{Eq mu metric}.
Let $\sigma\in\{-1,1\}^{E}$.
Let the signed measure on metric graph loops be
\begin{displaymath}
\tilde{\mu}^{\rm loop}_{\sigma}(d\wp) =
\hol^{\sigma}(\wp)\tilde{\mu}^{\rm loop}(d\wp),
\end{displaymath}
which can be decomposed according to the sign into
\begin{displaymath}
\tilde{\mu}^{\rm loop}_{\sigma}
=
\tilde{\mu}^{\rm loop}_{\sigma,+}
-
\tilde{\mu}^{\rm loop}_{\sigma,-}.
\end{displaymath}
Let $\widetilde{\LL}^{1/2}_{\sigma,+}$, respectively
$\widetilde{\LL}^{1/2}_{\sigma,-}$,
be the Poisson point processes of metric graph loops on
$\widetilde{\LG}$ with intensity measure 
$\frac{1}{2}\tilde{\mu}^{\rm loop}_{\sigma,+}$,
respectively $\frac{1}{2}\tilde{\mu}^{\rm loop}_{\sigma,-}$.
The version on Le Jan's isomorphism due to Kassel and Lévy 
\eqref{Eq KL Le Jan gauge}
extends in a straightforward way to the metric graphs:
\begin{equation}
\label{Eq Kassel Levy metric}
(\ell^{x}(\widetilde{\LL}^{1/2}_{\sigma ,+}))_{x\in \widetilde{\LG}}
\stackrel{(\text{law})}{=}
\Big(\dfrac{1}{2}\tilde{\phi}_{\sigma}(x)^{2}
+\ell^{x}(\widetilde{\LL}^{1/2}_{\sigma ,-})\Big)_{x\in \widetilde{\LG}},
\end{equation}
where $\tilde{\phi}_{\sigma}$ and $\widetilde{\LL}^{1/2}_{\sigma ,-}$
are taken independent,
and $\ell^{x}$ above denotes the Brownian local times.
One can use for instance an approximation from discrete through a subdivision of
edges as in Section \ref{Subsec subdivision}.
We also know that
\begin{equation}
\label{Eq Le Jan metric bis}
(\ell^{x}(\widetilde{\LL}^{1/2}))_{x\in \widetilde{\LG}}
=
(\ell^{x}(\widetilde{\LL}^{1/2}_{\sigma ,+})
+\ell^{x}(\widetilde{\LL}^{1/2}_{\sigma ,-}))_{x\in \widetilde{\LG}}
\stackrel{(\text{law})}{=}
\Big(\dfrac{1}{2}\tilde{\phi}(x)^{2}
\Big)_{x\in \widetilde{\LG}},
\end{equation}
where $\widetilde{\LL}^{1/2}_{\sigma ,+}$ and 
$\widetilde{\LL}^{1/2}_{\sigma ,-}$
are taken independent.
By combining \eqref{Eq Kassel Levy metric} and 
\eqref{Eq Le Jan metric bis} we get the following identity in law:
\begin{equation}
\label{Eq KL metric alpha 1}
\Big(\dfrac{1}{2}\tilde{\phi}(x)^{2}
\Big)_{x\in \widetilde{\LG}}
\stackrel{(\text{law})}{=}
\Big(\dfrac{1}{2}\tilde{\phi}_{\sigma}(x)^{2}
+\ell^{x}(\widetilde{\LL}^{1}_{\sigma ,-})\Big)_{x\in \widetilde{\LG}},
\end{equation}
where $\widetilde{\LL}^{1}_{\sigma ,-}$ is a Poisson point process with intensity measure $\tilde{\mu}^{\rm loop}_{\sigma,-}$,
that is to say with intensity parameter
$\alpha = 1 = 2\times \dfrac{1}{2}$,
and $\widetilde{\LL}^{1}_{\sigma ,-}$ being independent from
$\tilde{\phi}_{\sigma}$.

The identity \eqref{Eq KL metric alpha 1} provides another proof
for Theorem \ref{Thm main 1}.
Indeed, $\tilde{\phi}\in\LT_{\sigma}$ if and only if
$\tilde{\phi}^{2}\in\LT_{\sigma}$, which, by \eqref{Eq KL metric alpha 1},
is equivalent to 
$\big(\frac{1}{2}\tilde{\phi}_{\sigma}(x)^{2}
+\ell^{x}(\widetilde{\LL}^{1}_{\sigma ,-})\big)_{x\in \widetilde{\LG}}$
belonging to $\LT_{\sigma}$.
A necessary condition for the latter is 
$\widetilde{\LL}^{1}_{\sigma ,-}=\emptyset$, since $\widetilde{\LL}^{1}_{\sigma ,-}$
consists precisely of loops with holonomy $-1$.
This condition is also sufficient,
since $\tilde{\phi}_{\sigma}^{2}\in\LT_{\sigma}$ a.s.
(Lemma \ref{Lem T sigma}).
Thus, we get that
\begin{displaymath}
\mathbb{P}(\tilde{\phi}\in\LT_{\sigma}) = 
\mathbb{P}(\widetilde{\LL}^{1}_{\sigma ,-}=\emptyset),
\end{displaymath}
which is precisely the probability appearing in Theorem \ref{Thm main 1}.
We also get that conditionally on the event $\{\tilde{\phi}\in\LT_{\sigma}\}$,
$\tilde{\phi}^{2}$ is distributed as $\tilde{\phi}_{\sigma}^{2}$.

Let us now return to \eqref{Eq Kassel Levy metric}.
Consider $\mathcal{C}$ a connected component of
\begin{displaymath}
\Big\{ x\in \widetilde{\LG}
\Big\vert 
\dfrac{1}{2}\tilde{\phi}_{\sigma}(x)^{2}
+\ell^{x}(\widetilde{\LL}^{1/2}_{\sigma ,-})
\neq 0 \Big\} .
\end{displaymath}
Then $\bpi^{-1}_{\sigma}(\mathcal{C})$ is either connected or not.
If $\bpi^{-1}_{\sigma}(\mathcal{C})$ is not connected, 
than $\mathcal{C}$ cannot contain any loop with holonomy $-1$,
and in particular, $\mathcal{C}$ cannot contain any loop in
$\widetilde{\LL}^{1/2}_{\sigma ,-}$.
So in this case, $\mathcal{C}$ is actually a connected component of
$\{\vert\tilde{\phi}_{\sigma}\vert\neq 0\}$.
We summarize this remark in the corollary below.

\begin{cor}
\label{Cor iso sigma metric}
One can couple on the same probability space the metric graph loop soups
$\widetilde{\LL}^{1/2}_{\sigma ,+}$ and $\widetilde{\LL}^{1/2}_{\sigma ,-}$,
and the field $\tilde{\phi}_{\sigma}$,
such that all of the following conditions hold.
\begin{enumerate}
\item The field $\tilde{\phi}_{\sigma}$
and the loop soup $\widetilde{\LL}^{1/2}_{\sigma ,-}$
are independent.
\item For every $x\in \widetilde{\LG}$,
\begin{displaymath}
\ell^{x}(\widetilde{\LL}^{1/2}_{\sigma ,+})
=
\dfrac{1}{2}\tilde{\phi}_{\sigma}(x)^{2}
+\ell^{x}(\widetilde{\LL}^{1/2}_{\sigma ,-}).
\end{displaymath}
\item For every cluster $\mathcal{C}$ of
$\widetilde{\LL}^{1/2}_{\sigma ,+}$ such that
$\bpi^{-1}_{\sigma}(\mathcal{C})$ is not connected, 
$\mathcal{C}$ is also a connected component of
$\{\vert\tilde{\phi}_{\sigma}\vert\neq 0\}$
and $\ell^{x}(\widetilde{\LL}^{1/2}_{\sigma ,+})$
coincides with $\frac{1}{2}\tilde{\phi}_{\sigma}(x)^{2}$
on $\mathcal{C}$.
\end{enumerate}

In particular, one can couple on the same probability space
the metric graph loop soups $\widetilde{\LL}^{1/2}_{\sigma ,+}$
and the field $\tilde{\phi}_{\sigma}$,
such that the following conditions hold.
\begin{enumerate}
\item For every $x\in \widetilde{\LG}$,
\begin{displaymath}
\dfrac{1}{2}\tilde{\phi}_{\sigma}(x)^{2}
\leq
\ell^{x}(\widetilde{\LL}^{1/2}_{\sigma ,+}).
\end{displaymath}
\item For every cluster $\mathcal{C}$ of
$\widetilde{\LL}^{1/2}_{\sigma ,+}$ such that
$\bpi^{-1}_{\sigma}(\mathcal{C})$ is not connected, 
$\mathcal{C}$ is also a connected component of
$\{\vert\tilde{\phi}_{\sigma}\vert\neq 0\}$
and $\ell^{x}(\widetilde{\LL}^{1/2}_{\sigma ,+})$
coincides with $\frac{1}{2}\tilde{\phi}_{\sigma}(x)^{2}$
on $\mathcal{C}$.
In other words, $\ell^{x}(\widetilde{\LL}^{1/2}_{\sigma ,+})$
and $\frac{1}{2}\tilde{\phi}_{\sigma}(x)^{2}$
can differ only on clusters $\mathcal{C}$ of
$\widetilde{\LL}^{1/2}_{\sigma ,+}$ such that
$\bpi^{-1}_{\sigma}(\mathcal{C})$ is connected.
\end{enumerate}
\end{cor}

In the example of Figure \ref{Fig non triv ann},
$\widetilde{\LL}^{1/2}_{\sigma ,+}$ consists of loops that turn around the inner hole an even number of times, including those that do not surround it,
and $\widetilde{\LL}^{1/2}_{\sigma ,-}$ consists of loops that turn around the inner hole an odd number of times.
Further, $\ell^{x}(\widetilde{\LL}^{1/2}_{\sigma ,+})$ and 
$\frac{1}{2}\tilde{\phi}_{\sigma}(x)^{2}$
coincide on clusters of $\widetilde{\LL}^{1/2}_{\sigma ,+}$ that do not surround the inner hole.
At the risk of being redundant, let us emphasize that the dichotomy for loops
in $\widetilde{\LL}^{1/2}$ and the dichotomy for clusters of 
$\widetilde{\LL}^{1/2}_{\sigma ,+}$ are different.
For loops in $\widetilde{\LL}^{1/2}$ one distinguishes between the loops that turn an even number of times around the hole, and the loops that turn an odd number of time around the hole. In this way, the loops that turn twice around the hole are in the same class as the loops that do not surround the hole at all.
For the clusters of loops however, the dichotomy is just surrounding or not
surrounding the inner hole.

In view of the above corollary, 
perhaps it is worth pointing out the difference between the clusters of
$\widetilde{\LL}^{1/2}_{\sigma ,+}$ on the metric graph $\widetilde{\LG}$ 
and the clusters of $\LL^{1/2}_{\sigma ,+}$ on the discrete graph
$\LG$.

\begin{prop}
\label{Prop sprinklig sigma plus}
The discrete loop soup $\LL^{1/2}_{\sigma ,+}$
is obtained, up to a rerooting of the loops,
from the metric graph loop soup $\widetilde{\LL}^{1/2}_{\sigma ,+}$
by taking the trace of loops on $V$
with the time change $A^{-1}$ \eqref{Eq CAF}.
By doing this, one only takes into account the loops in 
$\widetilde{\LL}^{1/2}_{\sigma ,+}$ that visits at least one vertex.
In particular, for every $x\in V$,
$\ell^{x}(\widetilde{\LL}^{1/2}_{\sigma ,+})=
\ell^{x}(\LL^{1/2}_{\sigma ,+})$.
Further, the crossings of edges-lines $I_{e}$ by 
$\widetilde{\LL}^{1/2}_{\sigma ,+}$ correspond to the jumps through discrete edges
$e$ by $\LL^{1/2}_{\sigma ,+}$.

If an edge  $e\in E$ is visited by $\LL^{1/2}_{\sigma ,+}$,
then $\forall x\in I_{e},\ell^{x}(\widetilde{\LL}^{1/2}_{\sigma ,+})>0$ a.s.
Moreover, for every $e\in E$,
\begin{multline}
\label{Eq sprinkling sigma plus}
\mathbb{P}
\Big(\forall x\in I_{e},\ell^{x}(\widetilde{\LL}^{1/2}_{\sigma ,+})>0
\Big\vert \LL^{1/2}_{\sigma ,+},
e \text{ not visited by } \LL^{1/2}_{\sigma ,+}\Big)
\\=
1-
\exp\Big(-2C(e)(
\ell^{e_{-}}(\widetilde{\LL}^{1/2}_{\sigma ,+})
\ell^{e_{+}}(\widetilde{\LL}^{1/2}_{\sigma ,+})
)^{1/2}\Big),
\end{multline}
with conditional independence (given $\LL^{1/2}_{\sigma ,+}$)
across the edges $e\in E$.
\end{prop}

\begin{proof}
These are the properties already satisfied by the loop soups
$\widetilde{\LL}^{1/2}$ and $\LL^{1/2}$,
proven in \cite{Lupu2016Iso}.
The fact that $\LL^{1/2}$ is the trace of $\widetilde{\LL}^{1/2}$ 
on the vertices $V$ immediately implies that
$\LL^{1/2}_{\sigma ,+}$ is the trace of $\widetilde{\LL}^{1/2}_{\sigma ,+}$ on $V$,
as well as that 
$\LL^{1/2}_{\sigma,-}$ is the trace of $\widetilde{\LL}^{1/2}_{\sigma,-}$ on $V$.

As for the formula \eqref{Eq sprinkling sigma plus},
given $e\in E$ not visited by $\LL^{1/2}_{\sigma ,+}$,
the connection between the two endpoints of $e$ can be created by a superposition of three objects: the Brownian excursions from $e_{+}$ to $e_{+}$
inside $I_{e}$, the Brownian excursions from $e_{-}$ to $e_{-}$
inside $I_{e}$, and the Brownian loops of $\widetilde{\LL}^{1/2}_{\sigma ,+}$
that stay inside $I_{e}$.
The formula \eqref{Eq sprinkling sigma plus}
is then the same as for $\widetilde{\LL}^{1/2}$ and $\LL^{1/2}$
appearing in \cite{Lupu2016Iso} since the same three types of Brownian paths appear in both settings. This is in particular due to the fact that all the Brownian
loops in $\widetilde{\LL}^{1/2}$ that stay inside $I_{e}$ have a holonomy $1$,
and thus also appear in $\widetilde{\LL}^{1/2}_{\sigma ,+}$.
Note that for $\widetilde{\LL}^{1/2}_{\sigma ,-}$ and $\LL^{1/2}_{\sigma ,-}$,
a similar formula is no longer true, precisely because the Brownian
loops staying inside $I_{e}$ no longer participate to connecting the two ends of 
$e$, as they all have a wrong holonomy.
\end{proof}

\section{Relation to disordered Ising model}
\label{Sec rel Ising}

The goal of this Section is to explain how Theorem \ref{Thm main 1}
can be alternatively derived from a similar result for the FK-Ising model.

\subsection{Spin Ising, FK-Ising and Edwards-Sokal coupling}
\label{Subsec intro Ising}

So as to avoid confusion with our notation $\sigma$ for the $\{-1,1\}$-gauge fields,
we will denote the Ising spins by $\zeta$.
Let $\LG=(V,E)$ be a finite connected graph as in
Section \ref{Subsec gauge}.
The edges $\{x,y\}\in E$ are endowed with weights
$\beta(x,y) = \beta(y,x)\geq 0$.
The \textit{spin Ising} field is a random collection of signs
$(\hat{\zeta}(x))_{x\in V}\in\{-1,1\}^{V}$,
with the distribution given by
\begin{displaymath}
\mathbb{P}((\hat{\zeta}(x))_{x\in V}=(\zeta(x))_{x\in V})
=\dfrac{1}{Z^{\rm Isg}_{\beta}}
\exp\Big(\sum_{\{x,y\}\in E}
\beta(x,y)\zeta(x)\zeta(y)\Big).
\end{displaymath}

Further, the \textit{FK-Ising random cluster model} \cite{Grim2006FK}
is a random configuration of edges
$(\hat{\omega}(e))_{e\in E}\in\{0,1\}^{E}$, 
where $\hat{\omega}(e)=1$ corresponds to an open edge
and $\hat{\omega}(e)=0$ corresponds to a closed edge.
The probability distribution of the FK-Ising is given by
\begin{equation}
\label{Eq FK Ising}
\mathbb{P}((\hat{\omega}(e))_{e\in E}=(\omega(e))_{e\in E})
=\dfrac{1}{Z^{\rm FK-Isg}_{\beta}}
2^{\# \text{ clusters of } \omega}
\prod_{e\in E}(1-e^{-2\beta(e)})^{\omega(e)}(e^{-2\beta(e)})^{1-\omega(e)} ,
\end{equation}
where $\# \text{ clusters of } \omega$ is the number of clusters induced by the open edges of $\omega$.

The spin Ising and the FK-Ising models are related through the Edwards-Sokal coupling
\cite{EdwardsSokal88Ising}.

\begin{thm}[Edwards-Sokal, \cite{EdwardsSokal88Ising}]
\label{Thm Edwards-Sokal}
The following holds.
\begin{enumerate}
\item The two partition function are related through
\begin{displaymath}
Z^{\rm Isg}_{\beta} = 
Z^{\rm FK-Isg}_{\beta}
\prod_{e\in E} e^{\beta(e)}.
\end{displaymath}
\item The spin Ising configuration $(\hat{\zeta}(x))_{x\in V}$
and the FK-Ising configuration $(\hat{\omega}(e))_{e\in E}$
can be coupled as follows.
On first samples the FK-Ising configuration $(\hat{\omega}(e))_{e\in E}$,
and then one samples an independent uniform sign for each cluster induced by
$\hat{\omega}$.
\item The coupling above can be alternatively described as follows.
One first samples the spin Ising configuration $(\hat{\zeta}(x))_{x\in V}$.
Then, for each edge $\{x,y\}\in E$,
one sets $\hat{\omega}(x,y)=0$ if $\hat{\zeta}(x)\hat{\zeta}(y)=-1$,
and sets $\hat{\omega}(x,y)=1$ with conditional probability
$1-e^{-2\beta(x,y)}$ if $\hat{\zeta}(x)\hat{\zeta}(y) = 1$.
\end{enumerate}
\end{thm}

\subsection{Disordered Ising and topological probabilities for FK-Ising}
\label{Subsec disord Ising}

Let $\sigma\in\{-1,1\}^{E}$ be a gauge field as in Section \ref{Subsec gauge}.
The \textit{disordered spin Ising} field is a random
configuration of signs 
$(\hat{\zeta}_{\sigma}(x))_{x\in V}\in\{-1,1\}^{V}$ with the following probability distribution:
\begin{displaymath}
\mathbb{P}((\hat{\zeta}_{\sigma}(x))_{x\in V}=(\zeta(x))_{x\in V})
=\dfrac{1}{Z^{\rm Isg}_{\beta, \sigma}}
\exp\Big(\sum_{\{x,y\}\in E}
\beta(x,y)\zeta(x)\sigma(x,y)\zeta(y)\Big).
\end{displaymath}
The terminology originates from \cite{KadanoffCeva71DisorderIsing}.
The field $\hat{\zeta}_{\sigma}$ is covariant under gauge tranformations on 
$\sigma$ just as in the GFF case \eqref{Eq cov GFF gauge transfo}.

Theorem \ref{Thm main 1} has an analogue in the Ising setting.
Recall the notations of Section \ref{Subsec double cover}:
$\bpi_{\sigma}: \LG^{\rm db}_{\sigma}\rightarrow \LG$
is the double cover induced by $\sigma$.
Let $\Triv_{\sigma}\subset \{0,1\}^{E}$ denote the subset of edge configurations
$\omega$ such that every $\mathcal{C}$ cluster induced by the open edges of $\omega$,
$\bpi_{\sigma}^{-1}(\mathcal{C})$ is \textbf{not} connected.

\begin{thm}
\label{Thm disord Ising}
Let $(\hat{\omega}(e))_{e\in E}\in\{0,1\}^{E}$ be an FK-Ising configuration distributed according to \eqref{Eq FK Ising}.
The probability $\mathbb{P}(\hat{\omega}\in\Triv_{\sigma})$ can be expressed as a ratio of Ising partition functions:
\begin{displaymath}
\mathbb{P}(\hat{\omega}\in\Triv_{\sigma}) = 
\dfrac{Z^{\rm Isg}_{\beta, \sigma}}{Z^{\rm Isg}_{\beta}}.
\end{displaymath}
Moreover, the configuration $\hat{\omega}$ conditioned on 
$\{\hat{\omega}\in\Triv_{\sigma}\}$
can be sampled as follows.
\begin{enumerate}
\item First sample a disordered spin Ising configuration $\hat{\zeta}_{\sigma}$.
\item For each edge $e=\{x,y\}\in E$ such that
$\zeta(x)\sigma(x,y)\zeta(y) = -1$, set the edge $e$ to closed (i.e. $0$).
\item For each edge $e=\{x,y\}\in E$ such that
$\zeta(x)\sigma(x,y)\zeta(y) = 1$, set the edge $e$ to open (i.e. $1$)
with conditional probability $1-e^{-2\beta(x,y)}$.
\end{enumerate}
\end{thm}

Theorem \ref{Thm disord Ising} is a gauge-twisted version of the Edwards-Sokal coupling (Theorem \ref{Thm Edwards-Sokal}),
and just as the latter, can be derived through an elementary computation.
We would like to emphasize that just like Theorem \ref{Thm main 1},
Theorem \ref{Thm disord Ising} does not rely on planarity at all.
This result has been communicated to us by Marcin Lis (TU Wien, Vienna)
after the prepublication of the first version of this paper.
It also seems to be common knowledge among the Ising community,
but we did not find a good reference for it in the literature.

Alternatively, Theorem \ref{Thm disord Ising} can be derived from the GFF case as in
Remark \ref{Rem interact bos},
by considering double-well $\varphi^{4}$ interactions
$g\sum_{x\in V}(\tilde{\phi}(x)^{2}-1)^{2}$
and letting the coupling constant $g\to +\infty$.

Note that in the GFF setting (Theorem \ref{Thm main 1}) the partition functions and thus the topological probabilities are way more explicit and tractable than in the Ising setting (Theorem \ref{Thm disord Ising}).

\subsection{Relation between Ising and the GFF}
\label{Subsec rel Ising GFF}

Consider an electrical network $\LG=(V,E)$ as in
Section \ref{Subsec gauge}
endowed with conductances  $C(x,y)=C(y,x)>0$ for $\{x,y\}\in E$.
Given a non-negative function $h:V\rightarrow \R_{+}$,
we will denote by $\beta_{h}$ the weights
\begin{equation}
\label{Eq beta h}
\beta_{h}(x,y) = C(x,y) h(x) h(y), \qquad \{x,y\}\in E.
\end{equation}

In the sequel we will consider $\phi$ a discrete GFF on $\LG$ with $0$ boundary conditions, and its metric graph extension $\tilde{\phi}$
($\tilde{\phi}_{\vert V} = \phi$).
From the density \eqref{Eq phi} it is clear that conditionally on
$(\vert\phi(x)\vert)_{x\in V}$,
the signs $(\sgn(\phi(x)))_{x\in V_{\rm int}}$
are distributed as an Ising spin field with weights
\begin{displaymath}
\beta_{\vert\phi\vert}(x,y) = C(x,y) \vert\phi(x)\vert \vert\phi(y)\vert.
\end{displaymath}
Note that $\beta_{\vert\phi\vert}(x,y)=0$ if
$x$ or $y$ is in $V_{\partial}$.
As observed by Lupu and Werner in \cite{LupuWernerIsing},
the metric graph extension $\tilde{\phi}$ also naturally enters this picture and can be interpreted in terms of the FK-Ising.
Denote by $\omega_{\tilde{\phi}}\in\{0,1\}^{E}$ the following edge configuration.
We set $\omega_{\tilde{\phi}}(e) = 1$ if $\tilde{\phi}$ has no zeroes on the edge-line
$I_{e}$, and $\omega_{\tilde{\phi}}(e) = 0$ otherwise.
Note that if $e$ is adjacent to the boundary $V_{\partial}$,
then $\omega_{\tilde{\phi}}(e) = 0$ a.s.

\begin{prop}[Lupu-Werner, \cite{LupuWernerIsing}]
\label{Prop metric GFF Ising}
Let $h:V\rightarrow \R_{+}$ be a random non-negative field distributed as the absolute value of the discrete GFF $(\vert\phi(x)\vert)_{x\in V}$.
Let $(\hat{\zeta}(x))_{x\in V}\in\{-1,1\}^{V}$
be a random spin field distributed, conditionally on $h$,
as spin Ising with weights $\beta_{h}$.
Let $(\hat{\omega}(e))_{e\in E}\in\{0,1\}^{E}$
be a random edge configuration distributed, conditionally on $h$,
as FK-Ising with weights $\beta_{h}$.
We further assume that conditionally on $h$,
$\hat{\zeta}$ and $\hat{\omega}$ are coupled as in Edwards-Sokal coupling
(Theorem \ref{Thm Edwards-Sokal}).
Then $((\hat{\zeta}(x)h(x))_{x\in V},(\hat{\omega}(e))_{e\in E})$
are jointly distributed as
$((\phi(x))_{x\in V},(\omega_{\tilde{\phi}}(e))_{e\in E})$.
\end{prop}

\subsection{From FK-Ising topological probabilities to GFF topological probabilities}
\label{Subsec 3rd proof}

Here we will sketch an alternative proof of Theorem \ref{Thm main 1}
that relies on Proposition \ref{Prop metric GFF Ising} and
Theorem \ref{Thm disord Ising}.

The probability distribution of the absolute value of the discrete GFF 
$(\vert\phi(x)\vert)_{x\in V}$ can be written as
\begin{displaymath}
\dfrac{1}{Z}
Z^{\rm Isg}_{\beta_{h}}
\prod_{x\in V_{\rm int}}e^{-\frac{1}{2}W(x) h(x)^{2}}
\prod_{x\in V_{\rm int}} d h(x),
\end{displaymath}
where the weights $\beta_{h}$ are given by \eqref{Eq beta h},
with the convention $h_{\vert V_{\partial}}\equiv 0$,
$Z^{\rm Isg}_{\beta_{h}}$ is the spin Ising partition function for weights
$\beta_{h}$,
$Z$ is the GFF partition function,
and $W(x)$ is given by \eqref{Eq W sum}.
Now, let be a gauge field $\sigma\in\{ -1,1\}$.
Then $\tilde{\phi}\in\LT_{\sigma}$ if and only if
$\omega_{\tilde{\phi}}\in\Triv_{\sigma}$.
Therefore,
\begin{displaymath}
\mathbb{P}(\tilde{\phi}\in\LT_{\sigma}\,\vert\,
(\vert\phi(x)\vert)_{x\in V}
) =
\mathbb{P}(\omega_{\tilde{\phi}}\in\Triv_{\sigma}\,\vert\,
(\vert\phi(x)\vert)_{x\in V}
)
=
\dfrac{Z^{\rm Isg}_{\beta_{\vert\phi\vert},\sigma}}
{Z^{\rm Isg}_{\beta_{\vert\phi\vert}}}.
\end{displaymath}
Further,
\begin{eqnarray*}
\mathbb{P}(\tilde{\phi}\in\LT_{\sigma})
&=&
\E\big[\mathbb{P}(\tilde{\phi}\in\LT_{\sigma}\,\vert\,
(\vert\phi(x)\vert)_{x\in V})\big]
=
\E\left[\dfrac{Z^{\rm Isg}_{\beta_{\vert\phi\vert},\sigma}}
{Z^{\rm Isg}_{\beta_{\vert\phi\vert}}}\right]
\\
&=&\dfrac{1}{Z}
\int
\dfrac{Z^{\rm Isg}_{\beta_{h},\sigma}}{Z^{\rm Isg}_{\beta_{h}}}
Z^{\rm Isg}_{\beta_{h}}
\prod_{x\in V_{\rm int}}e^{-\frac{1}{2}W(x) h(x)^{2}}
\prod_{x\in V_{\rm int}} d h(x)
\\&=&
\dfrac{1}{Z}
\int
Z^{\rm Isg}_{\beta_{h},\sigma}
\prod_{x\in V_{\rm int}}e^{-\frac{1}{2}W(x) h(x)^{2}}
\prod_{x\in V_{\rm int}} d h(x)
=\dfrac{Z_{\sigma}}{Z}.
\end{eqnarray*}
The conditional distribution of $\tilde{\phi}$ on the event
$\{\tilde{\phi}\in\LT_{\sigma}\}$ can be obtained by similar arguments.

\section{Interpretations and implications of the result}
\label{Sec interpret}

\subsection{GFF on annular domains and exploration from inside}
\label{Subsec annular domain}

In the introduction (Figure \ref{Fig ex ann}) we considered the example of planar annular domains and the event that $\tilde{\phi}$ has a sign cluster that surrounds the inner hole.
Here we will recall how annular domains naturally arise in a
``simply connected" context.

For simplicity, let us consider a two-dimensional discrete box
$\LG = (V,E)$ with
\begin{displaymath}
V = \{-n, -n +1,\dots, -1,0,1,\dots, n-1,n\}^{2},
\end{displaymath}
$V_{\rm int} = \{-n +1,\dots, -1,0,1,\dots, n-1\}^{2}$,
$V_{\partial} = V\setminus V_{\rm int}$ and
the edges formed by 
$z,w\in V$ such that $\vert z-w\vert = 1$ (square lattice).
Let $\widetilde{\LG}$ be the metric graph associated to $\LG$.
Let $\tilde{\phi}$ the metric graph GFF (non-twisted) on $\widetilde{\LG}$
with $0$ boundary conditions on $V_{\partial}$.

Let $\widetilde{K}_{0}$ be a deterministic non-empty compact subset of 
$\widetilde{\LG}$
such that $\widetilde{K}_{0}$ is connected and $d(\widetilde{K}_{0}, V_{\partial})>1$.
Given a sample of $\tilde{\phi}$, we define the following random
subset  $\widetilde{K}_{1}$ of $\widetilde{\LG}$,
depending on $\widetilde{K}_{0}$ and $\tilde{\phi}$:
\begin{displaymath}
\widetilde{K}_{1} = 
\overline{\widetilde{K}_{0}\cup \{z\in \widetilde{\LG}\setminus \widetilde{K}_{0}
\vert \exists \text{ continuous path } \wp \text{ in } \LG,
z\stackrel{\wp}{\longleftrightarrow}\widetilde{K}_{0}
\text{ and } \vert\tilde{\phi}\vert_{\vert\wp}>0\}}.
\end{displaymath}
In other words, $\widetilde{K}_{1}$ is made of $\widetilde{K}_{0}$ and all the points that can be connected to $\widetilde{K}_{0}$ by a path $\wp$ on which $\tilde{\phi}$
does not go through $0$, except possibly at the extremity.
The random compact subset $\widetilde{K}_{1}$ is a so-called
\textit{stopping set} for the GFF $\tilde{\phi}$ : 
given any deterministic open subset $U$ of $\widetilde{\LG}$,
the event $\{\widetilde{K}_{1}\subset U\}$
is measurable with respect to the restriction $\tilde{\phi}_{\vert U}$.
The subset $\widetilde{K}_{1}$ can be obtained by first discovering
$\tilde{\phi}$ on $\widetilde{K}_{0}$ and then exploring from there in all the directions and stopping an exploration branch whenever the value of $\tilde{\phi}$
on this branch reaches $0$.
By construction, $\widetilde{K}_{1}$ is connected and $\tilde{\phi}$ is $0$
on $\partial \widetilde{K}_{1}$.
It is easy to see that with positive probability,
$d(\widetilde{K}_{1}, V_{\partial})>1$.
We further introduce another random set $\widetilde{K}_{2}$,
obtained by filling the inner holes of $\widetilde{K}_{1}$:
\begin{displaymath}
\widetilde{K}_{2} =
\widetilde{K}_{1}
\cup\bigcup_{\substack{O \text{ connected}\\ \text{component} \\ 
\text{of } \widetilde{\LG}\setminus \widetilde{K}_{1},~
O\cap V_{\partial} = \emptyset}} O.
\end{displaymath}
Then $\widetilde{K}_{2}$ is again a stopping set for $\tilde{\phi}$.
By construction,
$d(\widetilde{K}_{2}, V_{\partial}) = d(\widetilde{K}_{1}, V_{\partial})$,
and again, $\tilde{\phi}$ is $0$ on $\partial \widetilde{K}_{2}$.

On the event $\{ d(\widetilde{K}_{2}, V_{\partial})>1\}$,
which has a positive probability,
the sub-metric-graph $\overline{\widetilde{\LG}\setminus \widetilde{K}_{2}}$
is annular, that is to say it contains one hole,
which is $\widetilde{K}_{2}\setminus \partial \widetilde{K}_{2}$.
The values of $\tilde{\phi}$ on 
$\partial (\widetilde{\LG}\setminus \widetilde{K}_{2})$ are $0$ by construction.
Since $\widetilde{K}_{2}$ is a stopping set,
by the strong Markov property of $\tilde{\phi}$ (see \cite{Lupu2016Iso}),
conditionally on $(\widetilde{K}_{2}, \tilde{\phi}_{\vert \widetilde{K}_{2}})$,
the field $\tilde{\phi}_{\vert \overline{\widetilde{\LG}\setminus \widetilde{K}_{2}}}$
is distributed as the GFF on the metric graph 
$\overline{\widetilde{\LG}\setminus \widetilde{K}_{2}}$.
Therefore, Theorem \ref{Thm main 1} gives the conditional probability for a
sign cluster of 
$\tilde{\phi}_{\vert \overline{\widetilde{\LG}\setminus \widetilde{K}_{2}}}$
to surround $\widetilde{K}_{2}$
given $(\widetilde{K}_{2}, \tilde{\phi}_{\vert \widetilde{K}_{2}})$.
The conditional probability given just $\widetilde{K}_{2}$ is the same,
since $\tilde{\phi}_{\vert \overline{\widetilde{\LG}\setminus \widetilde{K}_{2}}}$
and $\tilde{\phi}_{\vert \widetilde{K}_{2}}$ are independent conditionally on
$\widetilde{K}_{2}$.

\subsection{GFF on annular domains: continuum limits}
\label{Subsec annular continuum}

Here we will consider what happens in the scaling limit on doubly connected (annular) domains in the scaling limit.

Here we will call an \textit{annular domain} an open bounded subset
$A\subset\C$ such that $\C\setminus A$ has two connected components,
one being necessarily unbounded and the other one bounded (the inner hole),
with the additional condition that the hole is not reduced to one point.
We will denote by $\partial_{o} A$ and $\partial_{i} A$
the outer and the inner boundaries of $A$.
The conformal equivalence classes of annular domains are parametrized
by the \textit{extremal distance}, or \textit{extremal length}, 
between the outer and the inner boundary
$\ED(\partial_{o} A, \partial_{i} A)$.
The quantity
$\ED(\partial_{o} A, \partial_{i} A)$
is really nothing else than the electrical resistance between
$\partial_{o} A$ and $\partial_{i} A$.
Given a circular annulus
\begin{displaymath}
A_{r_{1},r_{2}} = \{ z\in\C\vert r_{1}<\vert z\vert< r_{2}\},
\end{displaymath}
the corresponding extremal distance is
\begin{displaymath}
\ED(\partial_{o} A_{r_{1},r_{2}}, \partial_{i} A_{r_{1},r_{2}})
=  \dfrac{1}{2\pi} \log (r_{2}/r_{1}).
\end{displaymath}
So every annular domain $A$ is conformally equivalent to a circular annulus
$A_{r,1}$, where
\begin{displaymath}
r = \exp(-2\pi \ED(\partial_{o} A, \partial_{i} A)).
\end{displaymath}
For details, we refer to \cite[Chapter 4]{Ahlfors2010ConfInv}.

\subsubsection{Probability of non-contractible sign clusters in the scaling limit}
\label{Subsubsec Proba ann}

Let $A$ be an annular domain as above, and $\Phi$ a continuum GFF on $A$ with $0$
boundary conditions, both on $\partial_{o} A$ and $\partial_{i} A$.
In \cite{ALS3},
Aru-Lupu-Sep{\'u}lveda considered the non-contractible interior \textit{level lines} of $\Phi$ with step $2\lambda$, 
where $2\lambda$ is the Schramm-Sheffield \textit{height gap} 
\cite{SchSh,SchSh2} of the 2D continuum GFF.
More precisely, one has a random sequence of simple (Jordan) loops
$\lvl_{1},\dots ,\lvl_{N}$, with $N$ also random,
and a sequence of labels
$v_{1},\dots, v_{N}\in 2\lambda\Z$, where:
\begin{itemize}
\item by convention,
$\lvl_{0} = \partial_{o}A$, 
$\lvl_{N+1} = \partial_{i}A$,
$v_{0} = v_{N+1} = 0$;
\item $N\geq 0$ and $N$ finite and even;
\item for all $i\in \{ 1,\dots, N\}$,
$\lvl_{i}$ is a simple loop in $A$ that separates 
$\partial_{o} A$ and $\partial_{i} A$;
\item for all $i\in \{0, 1,\dots, N\}$,
$\lvl_{i}$ surrounds $\lvl_{i+1}$;
\item for all $i\in \{0,1,\dots, N\}$,
$v_{i+1} - v_{i}\in \{-2\lambda, 2\lambda\}$;
\item the family of random variables
$(N, \lvl_{1},\dots ,\lvl_{N}, v_{1},\dots, v_{N})$
is measurable with respect to $\Phi$;
\item for all $i\in \{ 1,\dots, N\}$, $v_{i}$ is the value of $\Phi$ on the outer side of $\lvl_{i}$ and $v_{i+1}$ is the value of 
$\Phi$ on the inner side of $\lvl_{i}$.
\end{itemize}
Each of the loops $\lvl_{i}$ locally look like an SLE$_{4}$ curve.
See \cite[Section 4.1]{ALS3} for details.

Now, with positive probability, $N=0$.
On the event $\{N=0\}$, all the non-contractible level lines 
$\lvl_{1},\dots ,\lvl_{N}$ and the labels
$v_{1},\dots, v_{N}$ do not exist.
The event $\{N=0\}$ is the continuum analogue of the event of no sign cluster surrounding the inner hole, that is to say of the left side of 
Figure \ref{Fig ex ann}.
The probability $\mathbb{P}(N=0)$ can be expressed explicitly by applying 
SLE and local set tools, as detailed for instance in
\cite[Proposition 2.18]{ALS3}.
This is also the same probability for a 2D Brownian loop soup
$\LL^{1/2}_{A}$ in $A$ (intensity parameter $\alpha = 1/2$)
not having a non-contractible cluster.
The probability $\mathbb{P}(N=0)$ can be expressed through one-dimensional Brownian bridges.
Let $(\widehat{W}_{t})_{0\leq t\leq L}$ be a standard Brownian bridge on $\R$ from $0$ to $0$, of time-length $L=\ED(\partial_{o} A, \partial_{i} A)$.
Then the following holds.

\begin{prop}[Aru-Lupu-Sep{\'u}lveda, \cite{ALS3}]
\label{Prop ALS N 0}
One has the equality
\begin{displaymath}
\mathbb{P}(N=0) = 
\mathbb{P}((\widehat{W}_{t})_{0\leq t\leq L} \text{ stays in } 
(-\sqrt{\pi/2},\sqrt{\pi/2})).
\end{displaymath}
\end{prop}

Note that $\sqrt{\pi/2}$ is the value of the height gap $2\lambda$ in an appropriate normalization of the GFF $\Phi$.
The probability $\mathbb{P}(N=0)$ depends on $A$ only through the extremal distance
$L=\ED(\partial_{o} A, \partial_{i} A)$,
that it to say only through the conformal equivalence class of $A$.
This is consistent with the conformal invariance of $\Phi$.

Now consider $\widetilde{A}^{(n)}$ metric graph approximations of the annular domain $A$ in the square lattice $\frac{1}{n}\Z^{2}$, 
and let $\tilde{\phi}_{n}$ be the GFF on 
$\widetilde{A}^{(n)}$ with $0$ boundary conditions.
We are interested to verify that
\begin{displaymath}
\lim_{n\to 0}\mathbb{P}(\tilde{\phi}_{n} \text{ has no sign clusters surrounding the inner hole of } \widetilde{A}^{(n)})
= \mathbb{P}(N=0).
\end{displaymath}
Of course, this can be deduced from the abstract arguments on the convergence of sign clusters.
But this is not our goal here.
Our goal is to compare the exact formulas given on one hand by
Theorem \ref{Thm main 1},
and on the other hand by Proposition \ref{Prop ALS N 0}
and check that they indeed match.
Consider this as a sanity check.

For this we will need to introduce the Brownian loop measure on $A$;
see \cite{LawlerWerner2004ConformalLoopSoup} and
\cite[Section 5.6]{LawlerConformallyInvariantProcesses}.
It is an infinite measure given by
\begin{equation}
\label{Eq mu loop ann}
\mu^{\rm loop}_{A}(d\wp)
=
\int_{A}
\int_{0}^{+\infty}\mathbb{P}^{z,z}_{A,t}(d\wp) p_{A}(t,z,z) \dfrac{dt}{t}
\,d^{2} z,
\end{equation}
where $p_{A}(t,z,w)$ denotes the heat kernel on $A$ with $0$ boundary conditions on
$\partial A$,
and $\mathbb{P}^{z,z}_{A,t}$ are the 2D Brownian bridge probability measures where the bridge is conditioned on staying in $A$.
Theorem \ref{Thm main 1} implies that
\begin{multline*}
\lim_{n\to 0}\mathbb{P}(\tilde{\phi}_{n} \text{ has no sign clusters surrounding the inner hole of } \widetilde{A}^{(n)})\\ = 
\exp\big(-\mu^{\rm loop}_{A}(\{\text{Loops that wind an odd number of times arround the hole of } A \})\big).
\end{multline*}
For the convergence of the loop measure from discrete to continuum, we refer to
\cite{LawlerFerreras07RWLoopSoup}.
So our goal is to verify the following.

\begin{prop}
\label{Prop bridges 1D 2D}
The following identity holds:
\begin{multline*}
\exp\big(-\mu^{\rm loop}_{A}(\{\text{Loops that wind an odd number of times arround the hole of } A \})\big)
\\ =
\mathbb{P}((\widehat{W}_{t})_{0\leq t\leq L} \text{ stays in } 
(-\sqrt{\pi/2},\sqrt{\pi/2})).
\end{multline*}
\end{prop}

Note that in the identity above, the left-hand side involves 2D Brownian bridges,
and the right-hand side 1D Brownian bridges.
We do not have a direct probabilistic interpretation for this identity,
other than relying on the content of this paper and all the knowledge on the level lines of the 2D continuum GFF. 
So our verification will proceed through explicit computations.
These computations are however rather sophisticated and ultimately rely on the Jacobi triple product identity
\eqref{Eq Jacobi triple product}.

Denote by $p_{\R}$ the heat kernel on $\R$, and by
$p_{(-\sqrt{\pi/2},\sqrt{\pi/2})}$ the heat kernel on $(-\sqrt{\pi/2},\sqrt{\pi/2})$
with $0$ boundary conditions at $\pm \sqrt{\pi/2}$.
Then
\begin{displaymath}
\mathbb{P}((\widehat{W}_{t})_{0\leq t\leq L} \text{ stays in } 
(-\sqrt{\pi/2},\sqrt{\pi/2})) = 
\dfrac{p_{(-\sqrt{\pi/2},\sqrt{\pi/2})}(L,0,0)}{p_{\R}(L,0,0)}
= \sqrt{2\pi L} p_{(-\sqrt{\pi/2},\sqrt{\pi/2})}(L,0,0).
\end{displaymath}
The quantity $p_{(-\sqrt{\pi/2},\sqrt{\pi/2})}(L,0,0)$ can be decomposed into series in two different ways.
By using the reflection principle, one gets
\begin{equation}
\label{Eq heat ker reflect}
p_{(-\sqrt{\pi/2},\sqrt{\pi/2})}(L,0,0) = 
\dfrac{1}{\sqrt{2\pi L}}
\sum_{k\in\Z}
\Big(e^{-\frac{4 k^{2}\pi }{L}}
-e^{-\frac{(2k+1)^{2}\pi }{L}}\Big).
\end{equation}
This is for instance Formula 3.0.2 in 
\cite[Section 1.3]{BorodinSalminen2015}.
By rather using the Fourier decomposition of the heat kernel, we get
\begin{equation}
\label{Eq heat ker Fourier}
p_{(-\sqrt{\pi/2},\sqrt{\pi/2})}(L,0,0) = 
\sqrt{\dfrac{2}{\pi}}\sum_{j\geq 0}
e^{-\frac{(2j+1)^{2}\pi L}{4}}.
\end{equation}
The two series \eqref{Eq heat ker reflect} and \eqref{Eq heat ker Fourier}
are related by the Poisson summation formula.
The quantity $p_{(-\sqrt{\pi/2},\sqrt{\pi/2})}(L,0,0)$
can also be written in terms of
Jacobi Theta functions, or rather Theta Nullwert functions;
see \cite[Section 16.27]{AbramowitzStegun84}.
With the standard notations, let be
\begin{displaymath}
\theta_{2}(q) = \sum_{j\in \Z} q^{(j+1/2)^{2}},
\qquad
\theta_{4}(q) = \sum_{k\in\Z}(-1)^{k} q^{k^{2}}.
\end{displaymath}
We have that
\begin{eqnarray}
\nonumber
\mathbb{P}((\widehat{W}_{t})_{0\leq t\leq L} \text{ stays in } 
(-\sqrt{\pi/2},\sqrt{\pi/2}))
&=&
\sqrt{2\pi L} p_{(-\sqrt{\pi/2},\sqrt{\pi/2})}(L,0,0)
\\&=&
\label{Eq bridge theta}
\theta_{4}(q= e^{-\frac{\pi}{L}})
=\sqrt{L}\theta_{2}(q=e^{-\pi L}).
\end{eqnarray}

Now let us perform the computations on the 2D loop side.
Let be
\begin{displaymath}
r = \exp(-2\pi \ED(\partial_{o} A, \partial_{i} A)) = e^{-2\pi L}.
\end{displaymath}
Let $\widehat{A}$ be the circular annulus
$\widehat{A} = A_{r,1}$,
so that $A$ is conformally equivalent to $\widehat{A}$.
Let $\mu^{\rm loop}_{\widehat{A}}$ be the Brownian loop measure on $\widehat{A}$.
The conformal invariance of the Brownian loop measure
(see \cite[Proposition 5.27]{LawlerConformallyInvariantProcesses}) ensures that
the following identity holds:
\begin{multline*}
\mu^{\rm loop}_{A}(\{\text{Loops that wind an odd number of times arround the hole of } A \})
\\ =
\mu^{\rm loop}_{\widehat{A}}(\{\text{Loops that wind an odd number of times arround the hole of } \widehat{A} \}).
\end{multline*}
Rather than performing computations on the annulus $\widehat{A}$,
we will lift everything up to its universal cover via the $\log$ map.
However, for doing this, it is much more convenient to endow $\widehat{A}$
with the cylindrical metric 
$\vert z\vert^{-2} d^{2}z$ rather than the Euclidean metric $d^{2} z$.
Indeed, consider the strip
\begin{equation}
\label{Eq strip}
S = \{u+iv\vert -2\pi L < u <0, v\in\R\}.
\end{equation}
The $\exp$ map induces a covering of $\widehat{A}$ by $S$,
and it sends the Euclidean metric on a fundamental domain in $S$ to
the cylindrical metric $\vert z\vert^{-2} d^{2}z$ on $\widehat{A}$.
So consider the time-changed Brownian motion with infinitesimal generator
$\frac{1}{2}\vert z\vert^{2} \Delta$,
killed upon hitting $\partial \widehat{A}$.
This process is symmetric with respect to the cylindrical metric 
$\vert z\vert^{-2} d^{2}z$.
Let $\stackrel{\circ}{p}_{\widehat{A}}(t,z,w)$
be its transition densities (with condition $0$ on $\partial \widehat{A}$)
with respect to $\vert z\vert^{-2} d^{2}z$,
and $\stackrel{\circ}{P}^{z,z}_{\widehat{A},t}$
the corresponding bridge probability measures conditioned on staying in 
$\widehat{A}$.
The Brownian loop measure on $\widehat{A}$ for the cylindrical metric
$\vert z\vert^{-2} d^{2}z$ is given by
\begin{displaymath}
\stackrel{\circ}{\mu}^{\rm loop}_{\widehat{A}}(d\wp)
=
\int_{\widehat{A}}
\int_{0}^{+\infty}\stackrel{\circ}{P}^{z,z}_{\widehat{A},t}(d\wp) 
\stackrel{\circ}{p}_{\widehat{A}}(t,z,z) \dfrac{dt}{t}
\,\dfrac{d^{2} z}{\vert z\vert^{2}}.
\end{displaymath}
Since the cylindrical metric $\vert z\vert^{-2} d^{2}z$
is conformally equivalent to the Euclidean metric $d^{2}z$,
the conformal invariance of the Brownian loop measure ensures that
\begin{multline*}
\mu^{\rm loop}_{\widehat{A}}(\{\text{Loops that wind an odd number of times arround the hole of } \widehat{A} \})
\\=
\stackrel{\circ}{\mu}^{\rm loop}_{\widehat{A}}(\{\text{Loops that wind an odd number of times arround the hole of } \widehat{A} \})
.
\end{multline*}
Now, lifting from $\widehat{A}$ up to $S$ through the $\log$ map,
one gets that for every $z\in \widehat{A}$ and $t>0$,
\begin{multline*}
\stackrel{\circ}{P}^{z,z}_{\widehat{A},t}(\text{Bridge winds an odd number of times arround the hole of } \widehat{A})
\stackrel{\circ}{p}_{\widehat{A}}(t,z,z)
\\=
\sum_{k\in\mathbb{Z}}
p_{S}(t,u+iv,u+i(v + (2k+1)2\pi)),
\end{multline*}
where $p_{S}$ is the heat kernel on the strip $S$ with $0$ boundary conditions on
$\partial S$,
and $e^{u+iv} = z$, with $v\in [0,2\pi)$.
So, at the end of the day,
\begin{multline}
\label{Eq Ann univ cover}
\mu^{\rm loop}_{A}(\{\text{Loops that wind an odd number of times arround the hole of } A \})
\\ =
\int_{-2\pi L}^{0}
\int_{0}^{2\pi}
\int_{0}^{+\infty}
\sum_{k\in\mathbb{Z}}
p_{S}(t,u+iv,u+i(v + (2k+1)2\pi))\,
\dfrac{dt}{t}\,dv\,du.
\end{multline}

Further, one can factorize the heat kernel on $S$ by separating real and imaginary parts:
\begin{displaymath}
p_{S}(t,u+iv,u'+iv') = p_{(-2\pi L, 0)}(t,u,u') p_{\R}(t,v,v')
=p_{(-2\pi L, 0)}(t,u,u') \dfrac{1}{\sqrt{2\pi t}}e^{-\frac{(v'-v)^{2}}{2t}},
\end{displaymath}
where $p_{(-2\pi L, 0)}(t,u,u')$ denotes the heat kernel on
$(-2\pi L, 0)$ with boundary condition $0$ in $-2\pi L$ and $0$.
Thus, \eqref{Eq Ann univ cover} equals
\begin{displaymath}
2\pi
\int_{0}^{+\infty}
\int_{-2\pi L}^{0}
p_{(-2\pi L, 0)}(t,u,u)\,du
\sum_{k\in\mathbb{Z}}
e^{-\frac{2(2k+1)^{2}\pi^{2}}{t}}
\,\dfrac{dt}{\sqrt{2\pi}t^{3/2}}.
\end{displaymath}
Further, similarly to \eqref{Eq heat ker Fourier}, one can write
\begin{eqnarray*}
p_{(-2\pi L, 0)}(t,u,u) &=& 
\dfrac{1}{\pi L}\sum_{j\geq 0}
\cos\Big(\dfrac{2j+1}{2L}(u+\pi L)\Big)^{2} e^{-\frac{(2j+1)^{2} t}{8 L^{2}}}
\\ &&+
\dfrac{1}{\pi L}\sum_{j\geq 1}
\sin\Big(\dfrac{j}{L}(u+\pi L)\Big)^{2}
e^{-\frac{j^{2} t}{2L^{2}}}.
\end{eqnarray*}
Thus,
\begin{displaymath}
\int_{-2\pi L}^{0}
p_{(-2\pi L, 0)}(t,u,u)\,du =
\sum_{j\geq 0}e^{-\frac{(2j+1)^{2} t}{8 L^{2}}}
+ \sum_{j\geq 1}e^{-\frac{j^{2} t}{2L^{2}}}
=\sum_{j\geq 1}e^{-\frac{j^{2} t}{8 L^{2}}}.
\end{displaymath}
Thus, \eqref{Eq Ann univ cover} equals
\begin{displaymath}
\sqrt{2\pi}
\sum_{j\geq 1}\sum_{k\in\Z}
\int_{0}^{+\infty}e^{-\frac{j^{2} t}{8L^{2}}-\frac{2(2k+1)^{2}\pi^{2}}{t}}
\,\dfrac{dt}{t^{3/2}}.
\end{displaymath}

\begin{lemma}
\label{Lem IG}
For every $a,b>0$,
\begin{displaymath}
\int_{0}^{+\infty}e^{-at - b t^{-1}}\,\dfrac{dt}{t^{3/2}}
=\sqrt{\dfrac{\pi}{b}} e^{-2\sqrt{ab}}.
\end{displaymath}
\end{lemma}

\begin{proof}
This identity seems to be classic.
For instance,
\begin{displaymath}
\ind_{t\geq 0}
\sqrt{\dfrac{\lambda}{2\pi t^{3}}}
\exp\Big(-\dfrac{\lambda(t-\mu)^{2}}{2\mu^{2}t}\Big)
\end{displaymath}
is the density of the Inverse Gaussian distribution
$\operatorname{IG}(\mu,\lambda)$ \cite{InverseGaussian},
and in particular, its integral on $(0,+\infty)$ equals $1$.
\end{proof}

By applying Lemma \ref{Lem IG}, we get that \eqref{Eq Ann univ cover} equals
\begin{eqnarray*}
\sum_{j\geq 1}\sum_{k\in\Z}
\dfrac{1}{\vert 2k+1\vert}e^{-\frac{j\vert 2k+1\vert\pi}{L}}
&=&
2\sum_{j\geq 1}\sum_{k\geq 0}
\dfrac{1}{2k+1}e^{-\frac{j (2k+1)\pi}{L}}
\\
&=&
-\sum_{j\geq 1}\big(2\log\big(1-e^{-\frac{j\pi}{L}}) 
- \log(1-e^{-\frac{2j\pi}{L}}\big)\big)
.
\end{eqnarray*}
By recombining the terms in the sum, this is turn equals
\begin{displaymath}
-\sum_{l\geq 1}\big(\log\big(1-e^{-\frac{2l\pi}{L}}\big)
+2 \log\big(1-e^{-\frac{(2l-1)\pi}{L}}\big)\big).
\end{displaymath}
So, we get
\begin{multline*}
\exp\big(-\mu^{\rm loop}_{A}(\{\text{Loops that wind an odd number of times arround the hole of } A \})\big)
\\ =
\prod_{l \geq 1}\big(1-e^{-\frac{2l\pi}{L}}\big)
\prod_{l \geq 1}\big(1-e^{-\frac{(2l-1)\pi}{L}}\big)^{2}.
\end{multline*}
By comparing with \eqref{Eq bridge theta},
to conclude to Proposition \ref{Prop bridges 1D 2D},
we need the following identity.
For every $q\in [0,1)$,
\begin{equation}
\label{Eq Jacobi triple product}
\prod_{l \geq 1}(1-q^{2l})
\prod_{l \geq 1}(1-q^{2l-1})^{2}
=\sum_{k\in\Z}(-1)^{k} q^{k^{2}}.
\end{equation}
The identity \eqref{Eq Jacobi triple product} is a special case of the Jacobi triple product identity; see \cite{Andrews65JacobiTripleProd}.

\subsubsection{Decomposition through contractible CLE$_{4}$ on annular domains.}
\label{Subsubsec MS annulus}

Here we will see that the Miller-Sheffield coupling has an analogue on annular domains. First, we present the construction of the gauge twisted GFF on the annulus.

We start by considering a circular annulus
$\widehat{A} = A_{r,1}$,
with as previously $L=\ED(\partial_{o}\widehat{A},\partial_{i}\widehat{A})
=\frac{1}{2\pi}\log r^{-1}$.
Given an angle $\theta\in [0,2\pi)$, we will denote by
$\gamma_{\theta}$ the ray $(ue^{i\theta})_{r\leq u\leq 1}$.
Let $\sigma_{\theta}$ be the $\{-1,1\}$-valued gauge field on $\widehat{A}$
corresponding to the \textit{defect line} $\gamma_{\theta}$:
Given a continuous path $(\wp(t))_{0\leq t\leq T}$ in $\widehat{A}$
with $\wp(0), \wp(T)\not\in \gamma_{\theta}$,
the holonomy $\hol^{\sigma_{\theta}}(\wp)$ is defined as follows.
Through the $\log$ map, one can lift $\wp$ to a continuous path 
$\hat{\wp}$ on the universal cover $S$ \eqref{Eq strip} of $\widehat{A}$, 
with endpoints $u(0)e^{i v(0)}$ and $u(T)e^{i v(T)}$.
If there are an even number of points of $2\pi \Z + \theta$ between
$v(0)$ and $v(T)$, 
then $\hol^{\sigma_{\theta}}(\wp)=1$.
If this number is odd, then $\hol^{\sigma_{\theta}}(\wp)= -1$.
Roughly speaking, $\hol^{\sigma_{\theta}}(\wp)$ is given by the parity of the number of crossings of the defect line $\gamma_{\theta}$ by the path $\wp$,
with the caveat that this number may well be infinite, as for instance in the case of
$\wp$ being a Brownian path; still the parity is well defined even in the case of infinite crossings.
Now, if the path $\wp$ is a closed loop ($\wp(0) = \wp(T)$),
then one can remove the restrictive condition $\wp(0), \wp(T)\not\in \gamma_{\theta}$.
In this case, $\hol^{\sigma_{\theta}}(\wp)= (-1)^{(v(T)-v(0))/2\pi}$.

One immediately remarks that for a closed loop $\wp$, the holonomy
$\hol^{\sigma_{\theta}}(\wp)$ is the same whatever the value of $\theta$.
And indeed, all the gauge fields $\sigma_{\theta}$ belong to the same gauge equivalence class.
Given $0\leq \theta_{1}<\theta_{2}<2\pi$, define
$\hat{\sigma}_{\theta_{1},\theta_{2}}$ on 
$\widehat{A}\setminus(\gamma_{\theta_{1}}\cup \gamma_{\theta_{2}})$ as follows:
$\hat{\sigma}_{\theta_{1},\theta_{2}}$ equals $1$ on
$\{ue^{i v}\vert r<u<1, 0\leq v<\theta_{1} \text{ or }
\theta_{2}<v<2\pi\}$,
and equals $-1$ on 
$\{ue^{i v}\vert r<u<1, \theta_{1}<v<\theta_{2}\}$.
Then for every continuous path $(\wp(t))_{0\leq t\leq T}$ in $\widehat{A}$
(not necessarily a closed loop),
with $\wp(0), \wp(T)\not\in \gamma_{\theta_{1}}\cup \gamma_{\theta_{2}}$,
\begin{displaymath}
\hol^{\sigma_{\theta_{2}}}(\wp)=
\hat{\sigma}_{\theta_{1},\theta_{2}}(\wp(0))
\hol^{\sigma_{\theta_{1}}}(\wp)
\hat{\sigma}_{\theta_{1},\theta_{2}}(\wp(T)).
\end{displaymath}

Further, the defect line does not have at all to be a straight ray
$\gamma_{\theta}$.
One can take any continuous simple curve 
$(\gamma(s))_{0\leq s\leq 1}$ such that
$\vert\gamma(0)\vert = r$, $\vert\gamma(1)\vert = 1$,
and $\vert\gamma(s)\vert \in (r,1)$
for every $s\in (0,1)$.
On the universal cover $S$, the connected components of
$(\exp)^{-1}(\widehat{A}\setminus \gamma)$ are naturally ordered:
one can define a function 
$\operatorname{ord}: (\exp)^{-1}(\widehat{A}\setminus \gamma)\mapsto\Z$
which is constant on each connected component of
$(\exp)^{-1}(\widehat{A}\setminus \gamma)$ and such that for every
$w\in (\exp)^{-1}(\widehat{A}\setminus \gamma)$,
$\operatorname{ord}(w+2\pi i) = \operatorname{ord} (w)+1$.
Then the gauge field $\sigma_{\gamma}$ associated to the defect line $\wp$ is defined as follows.
Let be a continuous path $(\wp(t))_{0\leq t\leq T}$ in $\widehat{A}$
with $\wp(0), \wp(T)\not\in \gamma$.
Through the $\log$ map, we lift $\wp$ to a continuous path 
$\hat{\wp}$ on the universal cover $S$ of $\widehat{A}$.
Then
\begin{displaymath}
\hol^{\sigma_{\gamma}}(\wp) = 
(-1)^{\operatorname{ord}(\hat{\wp}(t))-
\operatorname{ord}(\hat{\wp}(0))}.
\end{displaymath}
The gauge field $\sigma_{\gamma}$ is in the same gauge equivalence class as all the
$\sigma_{\theta}$.

Let be $\widehat{A}^{\rm db}$ be the circular annulus
$\widehat{A}^{\rm db} = A_{\sqrt{r},1}$.
Then the square map $\bpi(w) = w^{2}$ induces a double cover of
$\widehat{A}$ by $\widehat{A}^{\rm db}$.
For this covering map $\bpi$, the covering automorphism that
interchanges the two sheets of the covering is simply given by
$\psi(w) = -w$.
The maps $\bpi$ and $\psi$ are both holomorphic,
a fact that we will need later.
Let $\widehat{\Phi}^{\rm db}$ be the continuum GFF on $\widehat{A}^{\rm db}$
with $0$ boundary conditions,
normalized so that its covariance function is the Green's function of
$-\Delta$ with $0$ Dirichlet boundary conditions,
denoted $G_{\widehat{A}^{\rm db}}(z,w)$.
As in Sections \ref{Subsec double cover} and \ref{Subsec cov metric},
let be
\begin{displaymath}
\widehat{\Phi}^{\rm db}_{+} = 
\dfrac{1}{\sqrt{2}}(\widehat{\Phi}^{\rm db}+\widehat{\Phi}^{\rm db}\circ\psi),
\qquad
\widehat{\Phi}^{\rm db}_{-} = 
\dfrac{1}{\sqrt{2}}(\widehat{\Phi}^{\rm db}-\widehat{\Phi}^{\rm db}\circ\psi).
\end{displaymath}
Note that the fields $\widehat{\Phi}^{\rm db}_{+}$ and 
$\widehat{\Phi}^{\rm db}_{-}$ are independent.

From now on we will assume that the defect line $\gamma$ has zero Lebesgue (area) measure. 
We want to construct $\widehat{\Phi}_{\sigma_{\gamma}}$, the continuum GFF on 
$\widehat{A}$ twisted by the gauge field $\sigma_{\gamma}$,
but we do not want to specify the values of $\widehat{\Phi}_{\sigma_{\gamma}}$
on the defect line $\gamma$ itself.
So the zero Lebesgue measure ensures that the specification on $\gamma$ does not matter.
Note that this condition is not automatically satisfied by continuous simple curves.
As a counterexample, there are the so called Osgood curves; 
see \cite[Chapter 8]{Sagan94Curves}.
Let $\mathsf{s}_{\gamma}: \widehat{A}\setminus\gamma\rightarrow \widehat{A}^{\rm db}$
be a section of the covering map $\bpi$ (for all $z\in \widehat{A}\setminus\gamma$,
$\bpi(\mathsf{s}_{\gamma}(z))=z$)
that is continuous on $\widehat{A}\setminus\gamma$.
The section $\mathsf{s}_{\gamma}$ is a determination of the square root on
$\widehat{A}\setminus\gamma$, 
and in particular it is holomorphic.

Since $\mathsf{s}_{\gamma}$ and $\psi$ are holomorphic,
the field $\widehat{\Phi}^{\rm db}_{+}\circ \mathsf{s}_{\gamma}$
is distributed like the usual continuum GFF on $\widehat{A}$
with $0$ boundary conditions,
$\widehat{\Phi}$.
Now define $\widehat{\Phi}_{\sigma_{\gamma}} = \widehat{\Phi}^{\rm db}_{-}\circ \mathsf{s}_{\gamma}$.
Then $\widehat{\Phi}_{\sigma_{\gamma}}$ is a Gaussian field on
$\widehat{A}$ with covariance function given by
\begin{displaymath}
G_{\widehat{A},\sigma_{\gamma}}(z,w)= 
G_{\widehat{A}^{\rm db}}(\mathsf{s}_{\gamma}(z),\mathsf{s}_{\gamma}(w))
-
G_{\widehat{A}^{\rm db}}(\mathsf{s}_{\gamma}(z),\psi(\mathsf{s}_{\gamma}(w))).
\end{displaymath}
The field $\widehat{\Phi}_{\sigma_{\gamma}}$ is the
$\sigma_{\gamma}$-gauge-twisted GFF on $\widehat{A}$ with $0$ boundary conditions.
Note the discontinuity of $G_{\widehat{A},\sigma_{\gamma}}(z,w)$
when either of the variables crosses the defect line $\gamma$.

Regarding the regularity of the field $\widehat{\Phi}_{\sigma_{\gamma}}$,
one can work in the Sobolev space $\mathcal{H}^{-1}(\widehat{A})$;
see \cite[Section 2.3]{Sheffield07GFF} and \cite[Section 4.2]{Dubedat09PartitionFunc}.
The $\mathcal{H}^{-1}(\widehat{A})$ norm is given by
\begin{displaymath}
\Vert f\Vert^{2}_{\mathcal{H}^{-1}(\widehat{A})} = 
\iint_{\widehat{A}\times \widehat{A}}
f(z) G_{\widehat{A}}(z,w) f(w)\,d^{2}z\,d^{2}w,
\end{displaymath}
where $G_{\widehat{A}}$ is the Green's function of $-\Delta$ on
$\widehat{A}$ with $0$ boundary conditions.
Then $\widehat{\Phi}_{\sigma_{\gamma}}$ can be seen as a random element of
$\mathcal{H}^{-1}(\widehat{A})$, with
\begin{displaymath}
\E[\Vert \widehat{\Phi}_{\sigma_{\gamma}}\Vert^{2}_{\mathcal{H}^{-1}(\widehat{A})}]
\leq 
\E[\Vert \widehat{\Phi}\Vert^{2}_{\mathcal{H}^{-1}(\widehat{A})}]
<+\infty.
\end{displaymath}
This is due to the fact that
\begin{displaymath}
\iint_{\widehat{A}\times \widehat{A}}
G_{\widehat{A}}(z,w)\vert G_{\widehat{A},\sigma_{\gamma}}(z,w)\vert \,d^{2}z\,d^{2}w
\leq 
\iint_{\widehat{A}\times \widehat{A}}
G_{\widehat{A}}(z,w)^{2} \,d^{2}z\,d^{2}w
<+\infty.
\end{displaymath}

Now consider the case of a general annular domain $A$ that is conformally equivalent to $\widehat{A}$.
Let be $\varrho : A\rightarrow \widehat{A}$ be a conformal mapping.
Let $(\gamma(s))_{0\leq s\leq 1}$ be a continuous simple curve such that
$\gamma(0)\in\partial_{i}A$, $\gamma(1)\in\partial_{o}A$,
and $\gamma(s)\in A$ for every $s\in (0,1)$.
We also assume that $\gamma$ has $0$ Lebesgue measure.
The gauge-twisted GFF $\Phi_{\sigma_{\gamma}}$ on $A$ associated to the defect line 
$\gamma$,
with $0$ boundary conditions,
can be constructed as
$\Phi_{\sigma_{\gamma}} = \widehat{\Phi}_{\sigma_{\varrho(\gamma)}}\circ\varrho$.
For the continuous extension of $\varrho$ along $\gamma$ up to the extremities
$\gamma(0)$ and $\gamma(1)$, 
we refer to \cite[Proposition 2.14]{Pommerenke}. 

\medskip 

Next we would like to point out that Corollary \ref{Cor bicolor tilde phi}
has a continuum version that involves the contractible CLE$_{4}$ conformal loop ensemble on annular domains.
It can be interpreted as a gauge-twisted version of the Miller-Sheffield coupling.
For the original Miller-Sheffield on simply connected domains we refer to 
\cite{WaWu17LLGFF,ASW}. 
The contractible version of CLE$_{4}$ on annular domains
has been first introduced in \cite{HSW}.

So let $A$ be an annular domain as previously.
The contractible CLE$_{4}$ on $A$ is a random countable collection 
$\mathfrak{C}$ of Jordan loops in $A$ where each loop is simple,
contractible (does not surround the inner hole of $A$),
and any two different loops do not intersect and do not surround one another.
Each loop looks locally like and SLE$_{4}$.
Given a loop $\lvl\in \mathfrak{C}$, we will denote
by $\Int(\lvl)$ the open domain
surrounded by $\lvl$.
The domain $\Int(\lvl)$ is simply connected by construction.

Now consider the following additional random objects.
\begin{itemize}
\item Let $(\Phi_{\lvl})_{\lvl\in \mathfrak{C}}$ be a family of fields such that
each $\Phi_{\lvl}$ is distributed, conditionally on $\mathfrak{C}$,
as a GFF on $\Int(\lvl)$ with $0$ boundary conditions on $\lvl$ 
(and $0$ outside $\Int(\lvl)$);
and with the fields being conditionally independent given $\mathfrak{C}$.
\item Let $(\zeta_{\lvl})_{\lvl\in \mathfrak{C}}$
be random family of sign ($\zeta_{\lvl}\in\{-1,1\}$)
that are, conditionally on $\mathfrak{C}$,
i.i.d. and uniform ($\mathbb{P}(\zeta_{\lvl}=1) = \mathbb{P}(\zeta_{\lvl}= -1) =1/2$),
and independent of the fields $(\Phi_{\lvl})_{\lvl\in \mathfrak{C}}$.
\end{itemize}
Recall that $2\lambda$ denotes the height gap of the GFF.
With our normalization, $2\lambda = \sqrt{\pi/2}$.
Let $\Phi^{\ast}$ be the field on $A$ obtained as
\begin{displaymath}
\Phi^{\ast} = \sum_{\lvl\in \mathfrak{C}}
\ind_{\Int(\lvl)}(\Phi_{\lvl} + \zeta_{\lvl} 2\lambda).
\end{displaymath}
Unlike in the Miller-Sheffield coupling on simply connected domains,
on the annular domain $A$, 
$\Phi^{\ast}$ is not distributed as a GFF, but as a conditioned GFF.
Given the $0$ boundary GFF $\Phi$ on $A$,
let $E_{\rm contractible}$ be the event that $\Phi$ has no non-contractible
$\pm 2\lambda$-level lines.
The event $E_{\rm contractible}$ is exactly the event $\{N=0\}$
from Section \ref{Subsubsec Proba ann}.
Then the field $\Phi^{\ast}$ is distributed as the GFF $\Phi$ conditioned on the event
$E_{\rm contractible}$; see \cite[Section 4.4]{ALS3}.

The conditioned GFF $\Phi^{\ast}$ is \textit{a priori} non-Gaussian.
However, we shall see that there is a transformation that makes $\Phi^{\ast}$
into a Gaussian field.
As previously, let $\gamma$ be a (deterministic) defect line on $A$ with $0$ Lebesgue measure.
Denote
\begin{displaymath}
\mathfrak{C}_{\gamma} = \{\lvl\in\mathfrak{C}\vert \Int(\lvl)\cap\gamma\neq \emptyset\}.
\end{displaymath}
Next we argue that the defect line $\gamma$ induces a bipartite structure on
$\Int(\lvl)$ for each loop
$\lvl\in \mathfrak{C}_{\gamma}$.
Let $\upsilon : S\rightarrow A$ be a $2\pi i$-periodic covering map of the annular domain $A$ by the strip $S$ \eqref{Eq strip} (universal cover).
For instance, one can take
$\upsilon = \varrho^{-1}\circ\exp$,
where as previously, 
$\varrho : A\rightarrow \widehat{A}$ is a conformal mapping
from the annular domain $A$ to the circular annulus $\widehat{A}$.
Let $\operatorname{ord} : \upsilon^{-1}(A\setminus\gamma)\rightarrow\Z$
be an order on the connected components of $\upsilon^{-1}(A\setminus\gamma)$
such that $\operatorname{ord}(w+2\pi i) = \operatorname{ord}(w) + 1$.
Given a CLE loop $\lvl\in \mathfrak{C}$,
since $\lvl$ is contractible,
any lift of $\lvl$ to $S$ is a Jordan loop in $S$, and the different lifts differ by
$2\pi k i$ translations for $k\in\Z$.
For each $\lvl\in \mathfrak{C}$,
we choose a particular lift, denoted $\Lift(\lvl)$, given by the condition
\begin{displaymath}
\inf\{\operatorname{ord}(w)\vert w\in\Int(\Lift(\lvl))\} = 0,
\end{displaymath}
where $\Int(\Lift(\lvl))$ is the interior surrounded by $\Lift(\lvl)$ in $S$.
This convention is somewhat arbitrary, but we need to choose a particular lift one way or another.
For $\lvl\in \mathfrak{C}\setminus \mathfrak{C}_{\gamma}$,
$\operatorname{ord}$ is constant equal to $0$ on $\Int(\Lift(\lvl))$.
However, for $\lvl\in \mathfrak{C}_{\gamma}$, $\operatorname{ord}$
is not constant on $\Int(\Lift(\lvl))\setminus \upsilon^{-1}(\gamma)$.
For $z\in \Int(\lvl)$, we will denote by
$\upsilon_{0}^{-1}(z)$ the unique point in 
$\Int(\Lift(\lvl))\cap\upsilon^{-1}(\{z\})$.
Let us define the function $\zeta_{\mathfrak{C},\gamma}:
A\setminus\gamma\rightarrow\{-1,1\}$
as follows.
The function $\zeta_{\mathfrak{C},\gamma}$ equals $1$ on
\begin{displaymath}
(A\setminus\gamma)\setminus\bigcup_{\lvl\in \mathfrak{C}_{\gamma}}\Int(\lvl).
\end{displaymath}
For $\lvl\in \mathfrak{C}_{\gamma}$ and $z\in \Int(\lvl)\setminus\gamma$, we set
\begin{displaymath}
\zeta_{\mathfrak{C},\gamma}(z) = (-1)^{\operatorname{ord}(\upsilon_{0}^{-1}(z))}.
\end{displaymath}
Now, this definition of $\zeta_{\mathfrak{C},\gamma}$ may look abstract.
What we are simply saying is that the sign $\zeta_{\mathfrak{C},\gamma}(z)$
changes to the opposite as $z$ crosses the defect line
$\gamma$ by moving inside $\Int(\lvl)$.
See Figure \ref{Fig CLE_4 bicolor} for an illustration.
The sign function $\zeta_{\mathfrak{C},\gamma}$ only depends on the defect line $\gamma$ and on the realization of $\mathfrak{C}$.
By taking Corollary \ref{Cor bicolor tilde phi} to the scaling limit, one gets the following.

\begin{figure}
\includegraphics[scale=0.4]{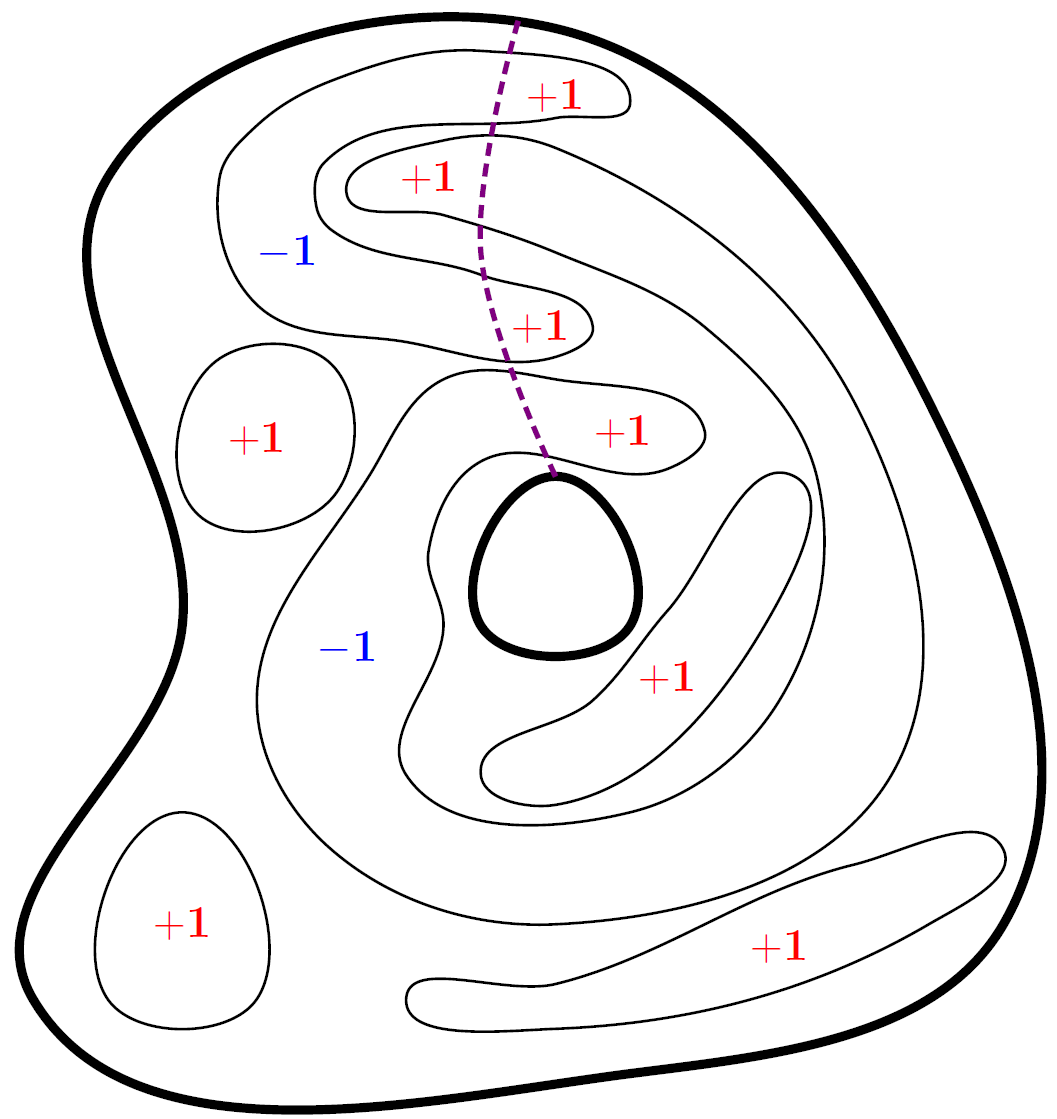}
\caption{
Conceptual depiction of the signs $\zeta_{\mathfrak{C},\gamma}$.
The defect line $\gamma$ is represented in violet dashed.}
\label{Fig CLE_4 bicolor}
\end{figure}

\begin{cor}
\label{Cor twisted MS}
The field $\zeta_{\mathfrak{C},\gamma}\Phi^{\ast}$ is distributed as the
gauge-twisted GFF $\Phi_{\sigma_{\gamma}}$, and
in particular, it is Gaussian.
\end{cor}

The convergence from metric graph to the continuum limit in dimension 2 is at this point standard;
see \cite{lupu2018convergence,ALS2,ALS4}.

Corollary \ref{Cor twisted MS} in particular implies that conditionally on the event
$E_{\rm contractible}$, the even local functions of the GFF $\Phi$ are related to even local functions of the gauge-twisted GFF $\Phi_{\sigma_{\gamma}}$.
Let us first detail the case of the renormalized hyperbolic cosine.

For $\beta\in(-2\sqrt{2\pi},2\sqrt{2\pi})$,
the renormalized exponential of $\Phi$
(Gaussian multiplicative chaos, GMC)
is given by the limit
\begin{equation}
\label{Eq GMC}
:e^{\beta\Phi}: = 
\lim_{\varepsilon\to 0}
e^{\beta\Phi_{\varepsilon} - 
\frac{\beta^{2}}{2}\Var(\Phi_{\varepsilon})},
\end{equation}
where $\Phi_{\varepsilon}$ are mollifications of $\Phi$.
The renormalized exponential $:e^{\beta\Phi}:$ is actually a random finite
Borel measure on $A$.
We refer to \cite{RhodesVargas14GMCReview,BerestyckiGMC,DuplantierSheffield} 
for details.
The unusual range $(-2\sqrt{2\pi},2\sqrt{2\pi})$ for $\beta$
is due to our choice of normalization for $\Phi$,
which is different from the one usually used in the GMC literature.
Similarly, for $\beta\in(-2\sqrt{2\pi},2\sqrt{2\pi})$, one can define the
renormalized exponential of $\Phi_{\sigma_{\gamma}}$:
\begin{equation}
\label{Eq GMC twisted}
:e^{\beta\Phi_{\sigma_{\gamma}}}: = 
\lim_{\varepsilon\to 0}
e^{\beta\Phi_{\sigma_{\gamma},\varepsilon} - 
\frac{\beta^{2}}{2}\Var(\Phi_{\sigma_{\gamma},\varepsilon})},
\end{equation}
where $\Phi_{\sigma_{\gamma},\varepsilon}$ are mollifications of 
$\Phi_{\sigma_{\gamma}}$.
Indeed, $\Phi_{\sigma_{\gamma}}$ is just another logarithmically correlated Gaussian field, with
\begin{displaymath}
G_{A,\sigma_{\gamma}}(z,w)\stackrel{w\to z}{\sim}
G_{A}(z,w),
\end{displaymath}
and the finite discontinuity of $G_{A,\sigma_{\gamma}}$ on the defect line
$\gamma$ is not an issue.
Note however that the renormalization factors in \eqref{Eq GMC} and 
\eqref{Eq GMC twisted} differ by a bounded factor.
The relevant quantity is given by
\begin{displaymath}
\Diff_{A}(z) = \lim_{w\to z} G_{A}(z,w) - G_{A,\sigma_{\gamma}}(z,w),
\qquad z\in A.
\end{displaymath}
In case of a circular annulus $\widehat{A}$,
\begin{displaymath}
\Diff_{\widehat{A}}(\bpi(w)) = 
2 G_{\widehat{A}^{\rm db}}(w,\psi(w)) = 2 G_{\widehat{A}^{\rm db}}(w, -w)
=2 G_{\widehat{A}^{\rm db}}(\vert w\vert, -\vert w\vert).
\end{displaymath}
Further, given a uniformization $\varrho : A\rightarrow \widehat{A}$,
\begin{displaymath}
\Diff_{A}(z) = \Diff_{\widehat{A}}(\varrho(z)).
\end{displaymath}
In particular, $\Diff_{A}(z)$ depends only on the conformal type of the domain with one marked interior point $(A,z)$,
and $\Diff_{A}(z)>0$.
Also note that $\Diff_{A}(z)$ does not depend on the particular choice of the defect line $\gamma$.
Further,
\begin{displaymath}
\Diff_{A}(z) = \lim_{\varepsilon\to 0}
\Var(\Phi_{\varepsilon}(z))-\Var(\Phi_{\sigma_{\gamma},\varepsilon}(z)).
\end{displaymath}

Now consider the renormalized hyperbolic cosines
$\frac{1}{2}(:e^{\beta\Phi}:+:e^{-\beta\Phi}:)$
and
$\frac{1}{2}(:e^{\beta\Phi_{\sigma_{\gamma}}}:+{:e^{-\beta\Phi_{\sigma_{\gamma}}}:})$.
While the fields $\Phi_{\sigma_{\gamma}}$,
$:e^{\beta\Phi_{\sigma_{\gamma}}}:$ and 
$:e^{-\beta\Phi_{\sigma_{\gamma}}}:$ each depend on the particular choice of the defect line $\gamma$, the hyperbolic cosine
$\frac{1}{2}(:e^{\beta\Phi_{\sigma_{\gamma}}}:+{:e^{-\beta\Phi_{\sigma_{\gamma}}}:})$
does not.
From Corollary \ref{Cor twisted MS},
and also directly from Theorem \ref{Thm main 1} by taking the scaling limit,
we obtain the following

\begin{cor}
\label{Cor twisted cosh}
Conditionally on the event $E_{\rm contractible}$,
the family of random measures
\\
${(\frac{1}{2}(:e^{\beta\Phi}:+{:e^{-\beta\Phi}:}))_{0\leq\beta <2\sqrt{2\pi}}}$ 
is distributed as
\begin{equation}
\label{Eq twisted cosh D A}
\Big(\dfrac{1}{2}
e^{-\frac{\beta^{2}}{2}\Diff_{A}}
(:e^{\beta\Phi_{\sigma_{\gamma}}}:+{:e^{-\beta\Phi_{\sigma_{\gamma}}}:})
\Big)_{0\leq\beta <2\sqrt{2\pi}}.
\end{equation}
\end{cor}

We further consider the renormalized (Wick) powers.
Let $(H_{n})_{n\geq 0}$ be the family of probabilistic Hermite polynomials,
that is to say unitary polynomials orthogonal for the Gaussian measure
$e^{-\frac{1}{2}x^{2}}\,dx$.
One can explicitly write
\begin{displaymath}
H_{n}(x) = \sum_{k=0}^{\lfloor n/2\rfloor}
(-1)^{k}\dfrac{n!}{2^{k} k! (n-2k)!}\,x^{n-2k}.
\end{displaymath}
One obtains the Wick powers as the limits
\begin{eqnarray*}
:\Phi^{n}: &=& \lim_{\varepsilon\to 0} 
\Var(\Phi_{\varepsilon})^{n/2}
H_{n}\big(
\Var(\Phi_{\varepsilon})^{-1/2}\Phi_{\varepsilon}\big),
\\
:\Phi_{\sigma_{\gamma}}^{n}: &=& \lim_{\varepsilon\to 0} 
\Var(\Phi_{\sigma_{\gamma},\varepsilon})^{n/2}
H_{n}\big(
\Var(\Phi_{\sigma_{\gamma},\varepsilon})^{-1/2}\Phi_{\sigma_{\gamma},\varepsilon}\big).
\end{eqnarray*}
We refer to \cite{Simon74EQFT,Janson1997GaussHilbSpaces,LeJan2011Loops,
LacoinRhodesVargas2017semiclassical}.
The Wick powers are random generalized functions.
The degree $0$ power is the constant $1$, and the degree $1$ power is the Gaussian field itself.
Note that the even Wick powers are not positive (except for $n=0$) because of the counterterms in the renormalization.
For the gauge-twisted GFF $\Phi_{\sigma_{\gamma}}$,
the odd Wick powers $:\Phi_{\sigma_{\gamma}}^{2k+1}:$ depend on the particular choice of the defect line $\gamma$,
but the even Wick powers $:\Phi_{\sigma_{\gamma}}^{2k}:$ do not.

For $\vert\beta\vert<\sqrt{2\pi}$ (the $L^{2}$ regime), 
one can expand the renormalized exponential into the Wick powers:
\begin{displaymath}
:e^{\beta\Phi}: = 
\sum_{n\geq 0}\dfrac{\beta^{n}}{n!}:\Phi^{n}:,
\qquad
:e^{\beta\Phi_{\sigma_{\gamma}}}: = 
\sum_{n\geq 0}\dfrac{\beta^{n}}{n!}:\Phi_{\sigma_{\gamma}}^{n}:.
\end{displaymath}
See \cite[Section 2.3.1]{LacoinRhodesVargas2017semiclassical}. 
In particular,
\begin{displaymath}
\frac{1}{2}(:e^{\beta\Phi}:+{:e^{-\beta\Phi}:})
= \sum_{k\geq 0}\dfrac{\beta^{2k}}{(2k)!}:\Phi^{2k}:,
\end{displaymath}
and similarly for $\Phi_{\sigma_{\gamma}}$.
Note however, that in the expansion of \eqref{Eq twisted cosh D A},
one needs to take into account the contribution of
$e^{-\frac{\beta^{2}}{2}\Diff_{A}}$.
We have that
\begin{displaymath}
\dfrac{1}{2}
e^{-\frac{\beta^{2}}{2}\Diff_{A}}
(:e^{\beta\Phi_{\sigma_{\gamma}}}:+{:e^{-\beta\Phi_{\sigma_{\gamma}}}:})
=
\sum_{k\geq 0}\dfrac{\beta^{2k}}{(2k)!}
\sum_{j=0}^{k}
(-1)^{k-j}
\dfrac{(2k)!}{2^{k-j}(k-j)!(2j)!}
\Diff_{A}^{k-j}:\Phi_{\sigma_{\gamma}}^{2j}:
.
\end{displaymath}
By identifying the terms in the expansions, 
we get the following.

\begin{cor}
\label{Cor twisted Wick}
Conditionally on the event $E_{\rm contractible}$,
the family of random fields
$(:\Phi^{2k}:)_{k\geq 1}$ 
is jointly distributed as
\begin{displaymath}
\Big(\sum_{j=0}^{k}
(-1)^{k-j}
\dfrac{(2k)!}{2^{k-j}(k-j)!(2j)!}
\Diff_{A}^{k-j}:\Phi_{\sigma_{\gamma}}^{2j}:
\Big)_{k\geq 1}
.
\end{displaymath}
\end{cor}

\begin{rem}
It would be interesting to have a proof of Corollaries \ref{Cor twisted MS},
\ref{Cor twisted cosh} and \ref{Cor twisted Wick}
directly in continuum, 
that does not rely on approximation by metric graphs.
\end{rem}

\begin{rem}
Corollaries \ref{Cor twisted cosh} and \ref{Cor twisted Wick} are also true on more general two-dimensional domains, not necessarily annular, such as the multiply connected domains considered in Section \ref{Subsec many holes}.
\end{rem}

\subsection{The planar case with multiple holes}
\label{Subsec many holes}

First, consider $D\subset \C$ an open bounded domain with two holes,
i.e. $\C\setminus D$ having three connected components, two of them being bounded.
We also assume that the holes are not reduced to one point (not punctures).
Let $\pi_{1}(D)$ denote the fundamental group of $D$.
It is isomorphic to $\mathbb{F}_{2}$, the free group with two generators.
In particular, there are $4=2^{2}$ group homomorphisms from 
$\pi_{1}(D)$ to $\{-1,1\}$,
one being the trivial homomorphism, and the three other being non-trivial.
Each of these homomorphisms corresponds to a gauge equivalence class for
$\{-1,1\}$-gauge fields on $D$ that are locally trivial, but not globally trivial in general.
One Figure \ref{Fig two holes} the three non-trivial equivalence classes are represented through defect lines: the holonomy is multiplied by $-1$ as one crosses a defect line.
Therefore, Theorem \ref{Thm main 1} provides in the scaling limit the probabilities of three different topological events on the continuum GFF on $D$ with $0$ boundary conditions.
These topological events deal with the sign excursion clusters of the GFF
(in the language of \cite{ALS4}), which are also the clusters of the Brownian loop soup on $D$ with parameter $\alpha=1/2$.
The defect line on top left on Figure \ref{Fig two holes} detects the sign excursion clusters that separate the first hole of $D$ from the outer boundary of $D$.
The defect line on top right detects the sign excursion clusters that separate the second hole from the outer boundary.
The defect line on the bottom detects the sign excursion clusters that separate
one hole from the other.
Note however that these three probabilities provide only a partial information on the distribution of the homotopical types of the clusters.
For instance, Figure \ref{Fig triple} depicts a cluster which is non-trivial simultaneously for all the three defect lines. 
Theorem \ref{Thm main 1} does not provide the probability of the existence of such a cluster.

\begin{figure}
\includegraphics[scale=0.3]{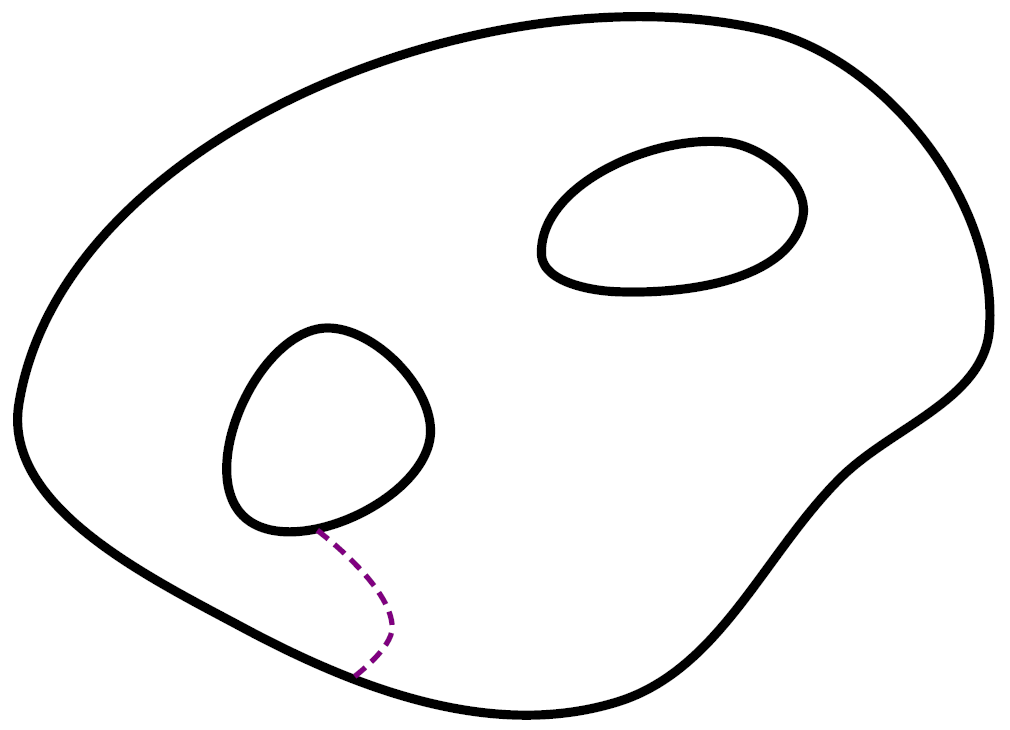}
\qquad
\includegraphics[scale=0.3]{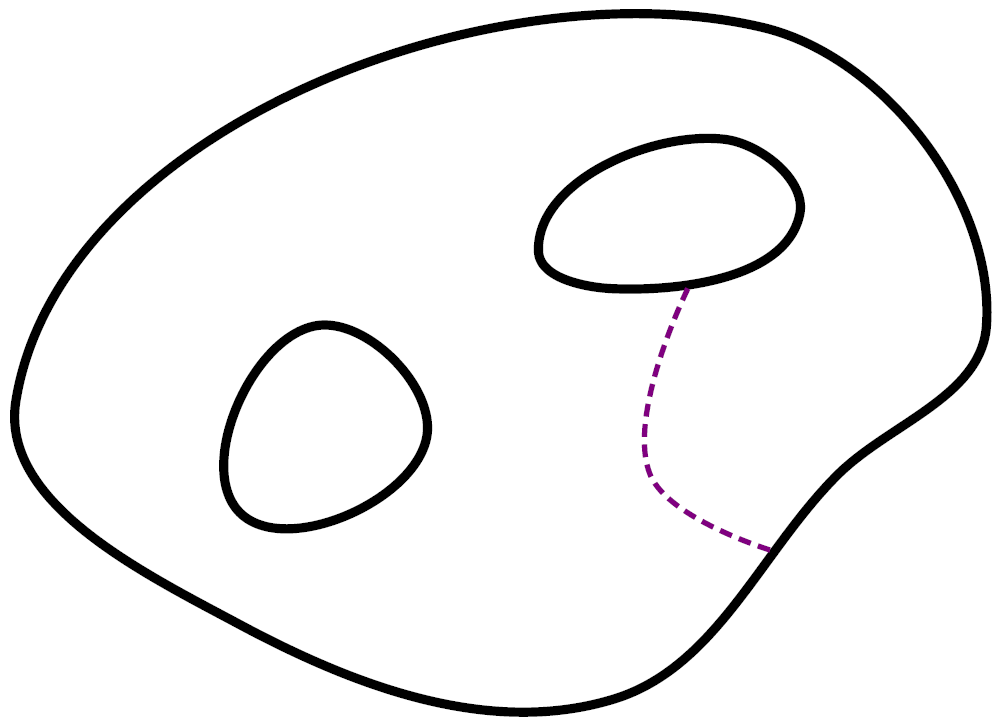}

\vspace{0.5cm}

\includegraphics[scale=0.3]{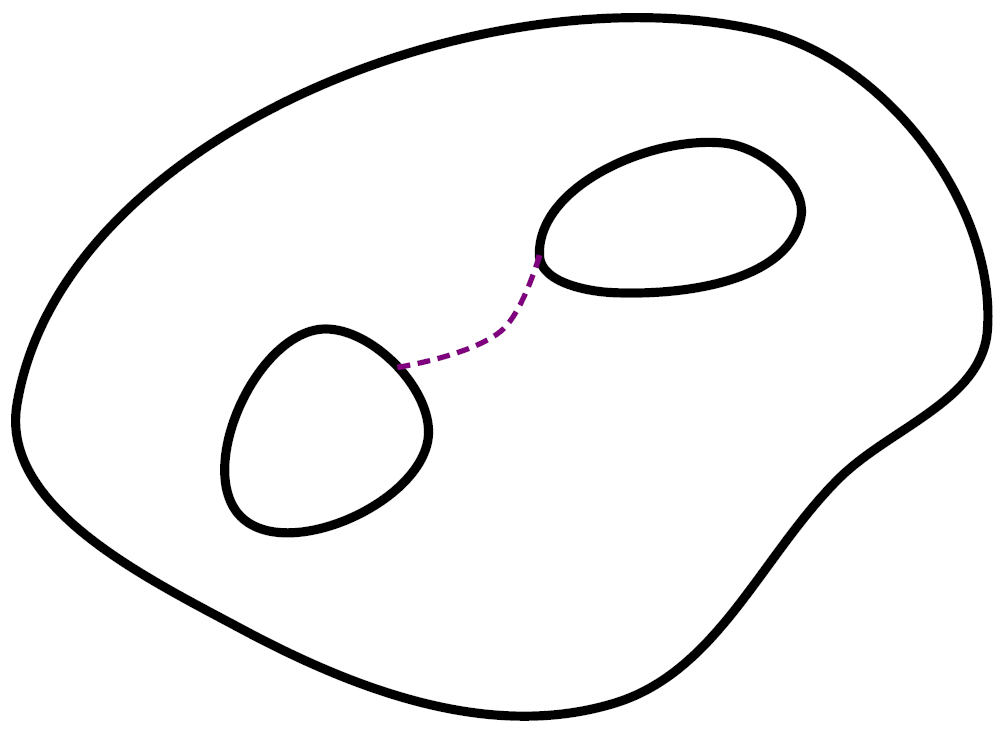}
\caption{
Depiction of 3 non-equivalent gauge fields on a planar domain with 2 holes.
The defect lines are in violet dashed.
}
\label{Fig two holes}
\end{figure}

\begin{figure}
\includegraphics[scale=0.3]{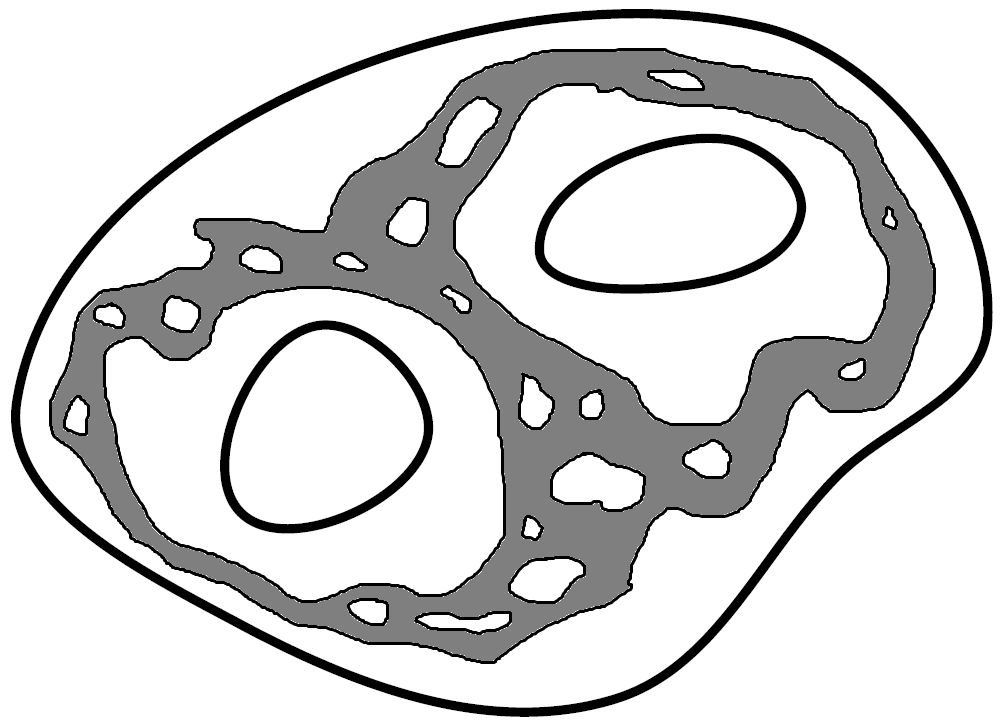}
\caption{
Depiction in grey of a cluster that is non-trivial for simultaneously all the 3 defect lines on Figure \ref{Fig two holes}.
}
\label{Fig triple}
\end{figure}

More generally, consider $D\subset \C$ an open bounded domain with $n$ holes,
the holes not being reduced to points (not punctures).
The fundamental group $\pi_{1}(D)$ is isomorphic to
$\mathbb{F}_{n}$, the free group with $n$ generators.
In particular, there are $2^{n}$ group homomorphisms from 
$\pi_{1}(D)$ to $\{-1,1\}$,
one of them being the trivial homomorphism.
Therefore, by applying Theorem \ref{Thm main 1} in the scaling limit,
one gets access to  probabilities of $2^{n}-1$ different topological events.
By contrast, the number of different homotopical types of a cluster in $D$
grows superexponentially with $n$.
Indeed, given $K$ a connected compact subset of  $D$,
the complementary $D\setminus K$ induces a partition on the connected components of the boundary $\partial D$.
Therefore, one looks at the number of partitions of
$\{0,1,\dots,n\}$ which is the Bell number $B_{n+1}$.
We have the asymptotic
$\log B_{n+1} \sim n\log n$;
see \cite[Section 6.3]{DeBruijn81Analysis}.

In the works of Dubédat \cite{Dubedat19DoubleDimers} and then Basok-Chelkak \cite{BasokChelkak18taufunc},
the authors obtained the probabilities of all possible homotopical
events in a nested CLE$_{4}$ on a simply connected domain with finitely many punctures,
through the use of isomonodromic deformations.
To our knowledge, the probabilities of homotopical events in case of macroscopic holes and not punctures were not known. 
Our method gives some of them.
Our method applies in general to compact bordered Riemann surfaces,
but we do not detail this here.

\subsection{The case of dimensions $d\geq 3$}
\label{Subsec dim geq 3}

As already mentioned several times, Theorem \ref{Thm main 1} does not
require planarity, and for instance holds on lattices in any dimension.
However, in a non-planar setting, 
the interpretation of the event $\{\tilde{\phi}\in\LT_{\sigma}\}$ is more
subtle, and this is something to keep in mind.
We would like to illustrate this on an example.

Fix $d\geq 3$ and $n\geq 1$.
Consider the hypercube
$Q = (-4n,4n)^{d}$
and the rectangular parallelotope
$Q' = (-n,n)^{2}\times(-4n,4n)^{d-2}$.
Let $D$ be the domain
$D=Q\setminus\overline{Q}'$.
Note that the fundamental group $\pi_{1}(D)$ is isomorphic to $\Z$,
as a loop in $D$ can surround several times 
the inner parallelotope $Q'$.
Let $\LG=(V,E)$ be the following graph.
We set $V=\overline{D}\cap\Z^{d}$, with
$V_{\rm int} = D\cap \Z^{d}$ and
$V_{\partial} = \partial D\cap \Z^{d}$.
The set of edges is given by
$E=\{\{x,y\}\vert x,y\in V, \Vert y-x\Vert =1\}$.
The edges are endowed with unit conductance:
$C(x,y) = 1$.
Now consider the following gauge field $\sigma$.
We set $\sigma(e) = -1$ for edges $e\in E$ of form
\begin{displaymath}
\{(k_{1},0,k_{3},\dots, k_{d}),(k_{1},1,k_{3},\dots, k_{d})\}.
\end{displaymath}
For all other edges $e$, we set $\sigma(e) = 1$.
Given a closed loop $\wp$ in $\LG$,
the holonomy $\hol^{\sigma}(\wp)$ equals $-1$ if
$\wp$ surrounds the inner parallelotope $Q'$ an odd number of times,
and it equals $1$ if $\wp$ surrounds $Q'$ and even number of times,
including if $\wp$ does not surrounds $Q'$ at all.
For instance, consider the loop
$\wp $ corresponding to a square with the following four corners:
\begin{displaymath}
(2n,-2n,0,\dots,0), 
(2n,2n,0,\dots,0),
(-2n,2n,0,\dots,0),
(-2n,-2n,0,\dots,0).
\end{displaymath}
Then $\hol^{\sigma}(\wp)=-1$.

Now here is what distinguished the non-planar setting from the planar setting.
For instance, on an annular-shaped metric graph as on
Figure \ref{Fig ex ann},
if there is a connected subset $\widetilde{K}$ that surrounds the inner hole,
then one can draw a loop $\wp$ in $\widetilde{K}$ that winds around the inner hole exactly once.
But this is a feature of the planarity, that is not true in general.
For instance, in our higher dimensional setting, we will give an example of a simple loop $\wp$ in $\LG$ that winds around $Q'$ twice.
Since $\wp$ is simple, it cannot have a sub-loop that would wind once around $Q'$.
In particular, we simultaneously have that $\wp$ is not contractible in the continuum domain $D$ and $\bpi^{-1}_{\sigma}(\wp)$ is not connected.
Again, such counter-example is impossible in a planar setting: if a loop winds twice around the hole of an annular domain, it cannot be simple.

So here is our example. 
We will take points $x_{1},x_{2},\dots ,x_{10}$ in $V$ and
straight line segments $\wp_{1},\wp_{2},\dots, \wp_{10}$,
where $\wp_{i}$ is a line from $x_{i}$ to $x_{i+1}$ for
$1\leq i\leq 9$, and $\wp_{10}$ is a line from $x_{10}$ to $x_{1}$.
Our loop $\wp$ will be the concatenation
\begin{displaymath}
\wp = \wp_{1}\wedge\wp_{2}\wedge\dots\wedge\wp_{10}.
\end{displaymath}
So here are the coordinates of the points $x_{i}$:
\begin{align*}
x_{1} &= (2n,-2n,0,\dots,0), &
&x_{2} = (2n,2n,0,\dots,0), &
&x_{3} = (-2n,2n,0,\dots,0),
\\
x_{4} &= (-2n,-3n,0,\dots,0), &
&x_{5} = (3n,-3n,0,\dots,0), &
\\
x_{6} &= (3n,-3n,n,0,\dots,0), &
&x_{7} = (3n,3n,n,0,\dots,0), &
&x_{8} = (-3n,3n,n,0,\dots,0),
\\
x_{9} &= (-3n,-2n,n,0,\dots,0), &
&x_{10} = (2n,-2n,n,0,\dots,0). &
\end{align*}
The projection of $\wp$ on the first two coordinates winds twice around
$(-n,n)^{2}$, but is not a simple loop.
We use the third coordinate to avoid self-intersections.

The example above shows that the dichotomy on clusters induced by the gauge field 
$\sigma$ is not contractible/non-contractible in $D$.
The contractible clusters are indeed trivial for $\sigma$,
but $\sigma$ detects only a certain type of non-contractible clusters,
those that allow to wind around $Q'$ an odd number of times.

\section{Intensity doubling conjecture in high dimension}
\label{Sec conj high dim}

Let be a dimension $d\geq 2$.
Let be the hypercube $Q = (-1,1)^{d}$.
Let $\LG^{(n)} = (V^{(n)},E^{(n)})$ be an approximation of $Q$ in the rescaled
hypercubic lattice $\frac{1}{n}\Z^{d}$.
More precisely,
we set $V^{(n)} = \overline{Q}\cap \frac{1}{n}\Z^{d}$,
$V^{(n)}_{\rm int} = Q\cap \frac{1}{n}\Z^{d}$,
$V^{(n)} = (\partial Q)\cap \frac{1}{n}\Z^{d}$,
and $E^{(n)} = \{\{x,y\}\vert x,y\in V^{(n)}, \Vert y-x\Vert = \frac{1}{n}\}$.
All the conductances are uniform equal to
$C(x,y) = \frac{1}{n^{d-2}}$.
In this regime, $\tilde{\phi}^{(n)}$, the metric graph GFF on
$\widetilde{\LG}^{(n)}$ with $0$ boundary conditions,
converges in law towards $\Phi$, the continuum GFF on $Q$
with $0$ boundary conditions on $\partial Q$.
Let $\widetilde{\LL}^{1/2}_{n}$ denote the metric graph loop soup on
$\widetilde{\LG}^{(n)}$ that is related to $\tilde{\phi}^{(n)}$
through isomorphism (Theorem \ref{Thm Lupu iso}).

We will further consider $(d-2)$-dimensional affine subspaces on $\R^{d}$,
that is to say like linear, but without the condition of going through $0$.
Given $\ax$ such an affine subspace, and $\varepsilon>0$,
we will denote by $\tub_{\varepsilon}(\ax)$ the hypertube
\begin{displaymath}
\tub_{\varepsilon}(\ax) = 
\{x\in\R^{d}\vert d(x,\ax)\leq\varepsilon\}.
\end{displaymath}
The fundamental group
$\pi_{1}(\R^{d}\setminus \tub_{\varepsilon}(\ax))$ is isomorphic to $\Z$,
since a loop can wind around the hypertube $\tub_{\varepsilon}(\ax)$.
The fundamental group
$\pi_{1}(Q\setminus \tub_{\varepsilon}(\ax))$ is also isomorphic to $\Z$,
provided that $Q\cap \ax \neq \emptyset$ and
$\varepsilon$ is small enough.
We will denote by $\widetilde{\LG}^{(n)}_{\ax,\varepsilon}$
a sub-metric-graph of $\widetilde{\LG}^{(n)}$
that reasonably approximates $Q\setminus \tub_{\varepsilon}(\ax)$
as $n\to +\infty$.
We will denote by $\tilde{\phi}^{(n)}_{\ax,\varepsilon}$
the metric graph GFF on $\widetilde{\LG}^{(n)}_{\ax,\varepsilon}$
with $0$ boundary conditions,
that is to say $0$ both on $\partial Q$ and
$\partial \tub_{\varepsilon}(\ax)$.
On $\widetilde{\LG}^{(n)}_{\ax,\varepsilon}$, one can choose a
$\{-1,1\}$ gauge field $\sigma_{n,\ax,\varepsilon}$
that detects the odd number of windings around $\tub_{\varepsilon}(\ax)$
(see also Section \ref{Subsec dim geq 3}),
provided that it is at all possible to wind around $\tub_{\varepsilon}(\ax)$ in
$Q$.
By applying Theorem \ref{Thm main 1} to $\sigma_{n,\ax,\varepsilon}$,
we get the value of
\begin{equation}
\label{Eq prob n odd}
\mathbb{P}(\exists \mathcal{C} \text{ sign cluster of } 
\tilde{\phi}^{(n)}_{\ax,\varepsilon} \text{ with an odd winding around } \tub_{\varepsilon}(\ax)).
\end{equation}
Further, one can express the limit of \eqref{Eq prob n odd} in terms of Brownian loop soups in $Q$.
For this we first recall the Brownian loop measure on $\R^{d}$.

The natural Brownian loop measure on $\R^{d}$ is given by
\begin{displaymath}
\mu^{\rm loop}_{\R^{d}}(d\wp) =
\int_{\R^{d}}\int_{0}^{+\infty}\mathbb{P}^{x,x}_{\R^{d},t}(d\wp)
\dfrac{dt}{(2\pi)^{d/2} t^{d/2 +1}}
dx,
\end{displaymath}
where $\mathbb{P}^{x,x}_{\R^{d},t}$ is the Brownian bridge probability measure in time $t$ from $x$ to $x$.
The measure is obviously invariant by translations and rotations,
but it is also scale invariant, provided one scales both space and time according to the Brownian scaling.
This scaling invariance imposes the power $t^{-(d/2+1)}$.
The total mass of $\mu^{\rm loop}_{\R^{d}}$ is infinite.
However, it gives a finite mass to families of loops of form
$\{\wp \text{ loop } \vert \wp\subset(-R,R)^{d}, \diam(\wp)\geq\delta\}$.
Given $D$ an open subset of $\R^{d}$, we will denote by
$\mu^{\rm loop}_{D}$ the restriction of the measure $\mu^{\rm loop}_{\R^{d}}$
to the loops that stay in $D$.
Note that this is consistent with the previous definition 
\eqref{Eq mu loop ann} in dimension 2.
Given $\alpha>0$,
we will denote by $\LL^{\alpha}_{D}$ the Poisson point process of Brownian loops in
$D$ with intensity $\alpha\mu^{\rm loop}_{D}$.
In particular, for every $\delta>0$, the Poisson collection
$\{\wp\in \widetilde{\LL}^{1/2}_{n}\vert \diam(\wp)>\delta\}$
converges in law as $n\to +\infty$ towards
$\{\wp\in \LL^{1/2}_{Q}\vert \diam(\wp)>\delta\}$.

Now, the limit of \eqref{Eq prob n odd} converges as $n\to+\infty$ towards
\begin{displaymath}
1-\exp\Big(
-\mu^{\rm loop}_{Q}\{\wp \text{ loop }\vert
\wp\cap  \tub_{\varepsilon}(\ax)=\emptyset,
\wp\text{ has an odd winding around }\tub_{\varepsilon}(\ax)\}
\Big).
\end{displaymath}
This in turn can be rewritten as
\begin{equation}
\label{Eq odd winding alpha 1}
\mathbb{P}(\exists \wp\in \LL^{1}_{Q}
\text{ with an odd winding around }\tub_{\varepsilon}(\ax)).
\end{equation}
Note that the intensity parameter $\alpha$ that appears in
\eqref{Eq odd winding alpha 1} is not
$\alpha = 1/2$ as in the isomorphism theorems
(Theorems \ref{Thm Le Jan iso} and \ref{Thm Lupu iso}),
but the double, $\alpha=1$.
This has been already noted in our Remark \ref{Rem not iso},
and further developed in the identity \eqref{Eq KL metric alpha 1}.
Now, in general, the sign clusters of a metric graph GFF are not Poisson, as they repel each other.
This is for instance understood precisely in dimension $2$ with the
CLE$_{4}$.
However, in high enough dimension $d$, the sign clusters should be in a sense approximately Poisson, as conjectured by Werner in
\cite{Werner21HigherD}.
Here we will argue that as soon as $d>8$
(and perhaps as soon as $d>6$),
asymptotically, as $n\to +\infty$,
the clusters formed by the microscopic loops in $\widetilde{\LL}^{1/2}_{n}$
give rise to an independent copy of the macroscopic loops of
$\widetilde{\LL}^{1/2}_{n}$.
Previously, the case of dimensions $d>6$
has been studied both at an heuristic and rigorous level by
Werner \cite{Werner21HigherD} and 
Cai and Ding \cite{CaiDing23HighD}.
Our conjecture emerged as an attempt to conciliate the above remark with the vision of the high-dimensional picture sketched by Werner and 
further confirmed by Cai and Ding.
It is also natural in view of the identity \eqref{Eq KL metric alpha 1},
where the intensity $\alpha=1$ appears explicitly.

To state our conjecture precisely, we will need a notion of size for the loops which is different from the diameter.
We will call it the \textit{encompass}.
For $\wp$ continuous loop in $\R^{d}$, set
\begin{multline*}
\enc(\wp) = 
\sup
\{
\varepsilon>0
\vert
\exists \ax \text{ affine subspace of } \R^{d},
\dim \ax  = d-2,
\\
\wp\cap \tub_{\varepsilon}(\ax) = \emptyset,
\wp \text{ non-contractible in } \R^{d}\setminus \tub_{\varepsilon}(\ax)
\}.
\end{multline*}
Note that trivially,
$\diam(\gamma)\geq 2 \enc(\wp)$.
But a converse control does not hold.
Even a simple loop with a fixed diameter can have an arbitrarily small
encompass.
Further, given a compact connected subset $K\subset\R^{d}$, we will denote
\begin{displaymath}
\widehat{\enc}(K) = 
\sup\{\enc(\wp)\vert \wp \text{ continuous loop}, \wp\subset K\}.
\end{displaymath}
If $K$ does not contain a loop with positive encompass,
then we set $\widehat{\enc}(K)=0$.
Again, $\diam(2)\geq 2 \widehat{\enc}(K)$.
But in the other direction, a tree, however large, has a zero $\widehat{\enc}$.
Note that if $\widehat{\enc}(K)>\varepsilon$,
this does not mean that there is an affine subspace $\ax$ of dimension $d-2$
such that $K\cap \tub_{\varepsilon}(\ax) = \emptyset$ and
$K$ non-contractible in $\R^{d}\setminus \tub_{\varepsilon}(\ax)$.
In particular, for a continuous but non-simple loop $\wp$,
the quantities $\enc(\wp)$ and
$\widehat{\enc}(\wp)$ are not necessarily equal,
and $\widehat{\enc}(\wp)$ may be strictly larger than $\enc(\wp)$ if
$\wp$ contains a sub-loop with a larger encompass.
As an other example, the $\widehat{\enc}$ of a ball equals its radius.

\begin{conj}[Intensity doubling]
\label{Conj double}
Assume that $d>8$.
Fix $\varepsilon>0$.
Consider the following two families of random objects.
First, the family of loops
\begin{equation}
\label{Eq loops enc}
\{\wp\in \widetilde{\LL}^{1/2}_{n}\vert \enc(\wp)>\varepsilon\}.
\end{equation}
Second, the family of clusters
\begin{equation}
\label{Eq clusters enc}
\{\mathcal{C} \text{ cluster of } \widetilde{\LL}^{1/2}_{n}\vert 
\widehat{\enc}(\mathcal{C})>\varepsilon\}.
\end{equation}
Note that by definition, the clusters corresponding to the loops in
\eqref{Eq loops enc} belong to the family \eqref{Eq clusters enc}.
We conjecture that as $n\to +\infty$,
\eqref{Eq loops enc} and \eqref{Eq clusters enc}
jointly converge in law for the Hausdorff distance towards
$(\LL_{\infty,\varepsilon},\mathcal{K}_{\infty,\varepsilon})$
with the following properties.
\begin{enumerate}
\item $\LL_{\infty,\varepsilon}$ is a Poisson point process of Brownian loop with intensity measure
\begin{displaymath}
\mathbf{\dfrac{1}{2}}
\times \ind_{\enc{\wp}>\varepsilon}
\mu^{\rm loop}_{Q}(d\wp).
\end{displaymath}
This point, concerning the convergence of
\eqref{Eq loops enc} alone, is known.
\item $\mathcal{K}_{\infty,\varepsilon}$ is an a.s. finite Poisson point process of compact connected subsets of $\overline{Q}$.
\item Each $K\in \mathcal{K}_{\infty,\varepsilon}$ is of dimension $4$
(Hausdorff and Minkowski),
and in particular is polar for the Brownian motion.
\item The compacts in $\mathcal{K}_{\infty,\varepsilon}$ are pairwise disjoint.
\item Each $K\in \mathcal{K}_{\infty,\varepsilon}$ contains a single simple loop 
which we denote $\Loop(K)$.
We will also denote $\Loop(\mathcal{K}_{\infty,\varepsilon})=
\{\Loop(K)\vert K\in \mathcal{K}_{\infty,\varepsilon}\}$.
\item We have that $\LL_{\infty,\varepsilon}\subset \Loop(\mathcal{K}_{\infty,\varepsilon})$.
\item The family of loops $\Loop(\mathcal{K}_{\infty,\varepsilon})$ is distributed as a Poisson point process with intensity
\begin{displaymath}
\mathbf{1}
\times \ind_{\enc{\wp}>\varepsilon}
\mu^{\rm loop}_{Q}(d\wp).
\end{displaymath}
In particular, with positive probability,
$\LL_{\infty,\varepsilon}$ is a strict subset of $\Loop(\mathcal{K}_{\infty,\varepsilon})$.
\item The two subfamilies
$\{K\in \mathcal{K}_{\infty,\varepsilon}\vert \Loop(K)\in \LL_{\infty,\varepsilon}\}$
and 
$\{K\in \mathcal{K}_{\infty,\varepsilon}\vert \Loop(K)\not\in \LL_{\infty,\varepsilon}\}$
have the same distribution and are independent.
\end{enumerate}
\end{conj}

The dimension $4$ is that of the Integrated Super-Brownian Excursions (ISE) 
\cite{DawsonPerkins98ISE,LeGall99ISE},
which is conjectured to appear in the scaling limit of different critical lattice models in high dimension \cite{Slade99ISE}.
Here we took $d>8 = 4+4$ to ensure that the compacts in 
$\mathcal{K}_{\infty,\varepsilon}$ stay pairwise disjoint.
We believe that the Poisson structure for the clusters in the scaling limit
could already emerge in dimensions $d\in\{7,8\}$.
Since the clusters are expected to be asymptotically polar also in these dimensions,
and thus invisible for the GFF, there should be no macroscopic repulsion between clusters.
However, for $d\in\{7,8\}$, two disjoint clusters might intersect in the scaling limit, and a cluster might also get additional cycles in the scaling limit that are not there at the metric graph level.
Yet, if one focuses solely on the cycles that are already there at the metric graph level,
the intensity doubling should also occur in dimensions $d\in\{7,8\}$.

Next we bring arguments for Conjecture \ref{Conj double}.
The first question is whether it provides the correct limit
\eqref{Eq odd winding alpha 1}.
One should pay attention that in \eqref{Eq prob n odd}
appears not the GFF on $\widetilde{\LG}^{(n)}$ but the GFF on
$\widetilde{\LG}^{(n)}_{\ax,\varepsilon}$.
However, if $K$ is a compact in $\mathcal{K}_{\infty,\varepsilon_{0}}$
such that $\Loop(K)\cap \tub_{\varepsilon}(\ax) = \emptyset$,
then the connected component of $\Loop(K)$ in
$\overline{K\setminus\tub_{\varepsilon}(\ax)}$
is the scaling limit of a sign cluster of the GFF on 
$\widetilde{\LG}^{(n)}_{\ax,\varepsilon}$.
Indeed, in case $\Loop(K)\in \LL_{\infty,\varepsilon_{0}}$,
the metric graph cluster approximating $K$ contains just one macroscopic loop
of $\widetilde{\LL}^{1/2}_{n}$, 
the one approximating $\Loop(K)$, 
and no macroscopic loop of $\widetilde{\LL}^{1/2}_{n}$
in case $\Loop(K)\not\in \LL_{\infty,\varepsilon_{0}}$.
The rest of the cluster is made of microscopic loops.
By removing the microscopic loops that intersect $\tub_{\varepsilon}(\ax)$,
one changes the cluster only on the hypertube $\tub_{\varepsilon}(\ax)$
and its infinitesimal neighborhood.
The fact that each cluster in \eqref{Eq clusters enc} asymptotically contains at most one macroscopic loop of $\widetilde{\LL}^{1/2}_{n}$ is also consistent with the
asymptotic polarity of the clusters.

Let be $d>6$. For fixed $\delta>0$,
Werner showed that the number of clusters of $\widetilde{\LL}^{1/2}_{n}$
of diameter larger than $\delta$ is of order at least $n^{d-6 + o(1)}$
\cite[Proposition 4]{Werner21HigherD}.
So, in particular, the $\delta$-large clusters proliferate as $n\to +\infty$.
A typical large cluster is believed to look like a tree at a macroscopic scale,
with a scaling limit describable in terms of superbrownian excursions.
In contrast, according to our Conjecture \ref{Conj double},
the number of clusters with large $\widehat{\enc}$ is tight,
and therefore such clusters are untypical among clusters of large diameter.

Now let us fix a hypertube $\tub_{\varepsilon}(\ax)$ 
such that $\pi_{1}(Q\setminus\tub_{\varepsilon}(\ax))\cong\Z$.
We look at the clusters in $\widetilde{\LG}^{(n)}_{\ax,\varepsilon}$ of diameter
larger then $2\varepsilon$.
The number of those diverges as $n\to +\infty$.
As for the number of loops in $\widetilde{\LL}^{1/2}_{n}$
with an odd winding around $\tub_{\varepsilon}(\ax)$, it is tight.
So the number of clusters in $\widetilde{\LG}^{(n)}_{\ax,\varepsilon}$ of diameter
larger then $2\varepsilon$ and that do not contain loops of 
$\widetilde{\LL}^{1/2}_{n}$ with an odd winding around $\tub_{\varepsilon}(\ax)$ 
also diverges as $n\to +\infty$.
What our Theorem \ref{Thm main 1} shows is that the clusters of the latter type
still can achieve an odd winding around $\tub_{\varepsilon}(\ax)$,
and this happens with a probability bounded away from $0$ and $1$.
Since this probability is bounded away from $1$, whereas the number of clusters proliferates, we infer that for an individual cluster it is increasingly costly as 
$n\to+\infty$ to achieve such an odd winding,
and this odd winding is achieved at all only because many clusters try to do so.
So heuristically, we interpret things in terms of a Bernoulli scheme as follows.
\begin{itemize}
\item The number of clusters that try to achieve an odd winding around
$\tub_{\varepsilon}(\ax)$ diverges as $n\to +\infty$.
\item In high enough dimension $d$, each of these clusters tries to achieve this odd winding almost independently from the other clusters.
\item As $n\to +\infty$, the probability that at least one cluster achieves an odd winding around $\tub_{\varepsilon}(\ax)$ converges to a value in
$(0,1)$.
\end{itemize}
So we are in the regime of the law of rare events, and the limit distribution of the number of clusters that achieve an odd winding around $\tub_{\varepsilon}(\ax)$
is Poisson.
The parameter of the corresponding Poisson distribution is the one given by the Brownian loop soup, in agreement with the limit probability of success.

Now, let us take a second hypertube $\tub_{\varepsilon'}(\ax')$.
A typical macroscopic cluster does not contain macroscopic cycles at all.
It is already exceptional for such a cluster to have a cycle around
$\tub_{\varepsilon}(\ax)$,
and it would be exceptional squared for a cluster to have a second cycle around 
$\tub_{\varepsilon'}(\ax')$, unless it is the same cycle that achieves an odd winding
both around $\tub_{\varepsilon}(\ax)$ and $\tub_{\varepsilon'}(\ax')$.
So asymptotically, a cluster can contain at most one non-trivial macroscopic cycle.

Now, since the Poisson distribution in the limit is true for an arbitrary choice
of the hypertube $\tub_{\varepsilon}(\ax)$,
this determines not only the limit law of the number of clusters with prescribed topology,
but also the very shape of the non-trivial macroscopic cycles that can be drawn inside the clusters of $\widetilde{\LL}^{1/2}_{n}$.
In the limit, these macroscopic cycles are distributed as a 
Brownian loop soup with intensity parameter
$\alpha = 1 = 2\times \dfrac{1}{2}$,
in accordance with \eqref{Eq odd winding alpha 1}.
So this is our heuristic reasoning.
Figure \ref{Fig concept high D} depicts our understand of clusters in high dimension.

\begin{figure}
\includegraphics[scale=0.5]{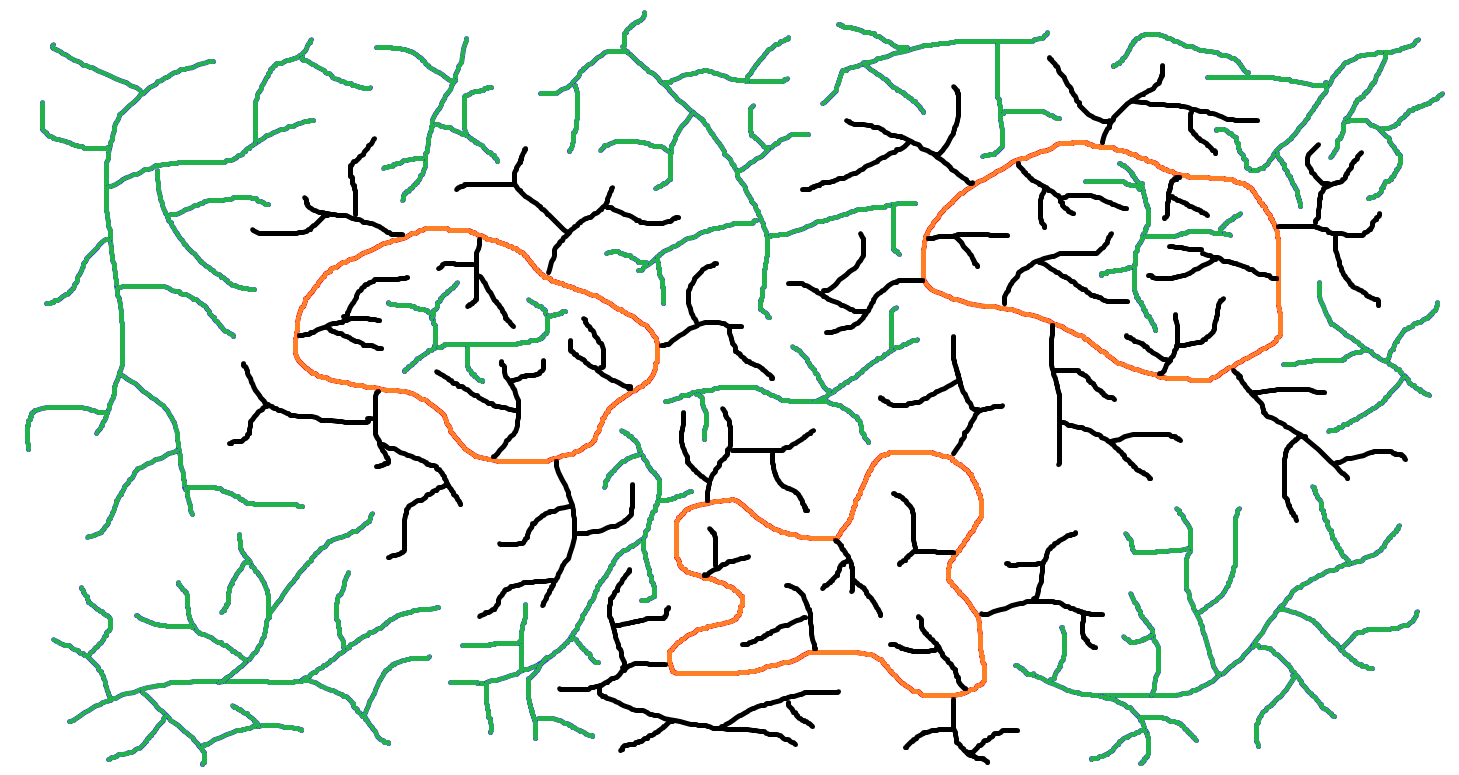}
\caption{
Conceptual depiction of the sign clusters of the metric GFF in high dimension.
The trees are not exactly so, but are actually made of microscopic loops at the metric graph level.
Most large clusters are macroscopically tree-like, as the trees in green.
Their number diverges in the scaling limit.
However, there is a tight number per scale of macroscopic cycles, that are represented in orange.
The limit distribution of the macroscopic cycles is that of a Brownian loop soup of intensity parameter $\alpha = 1$,
that is to say double of the intensity appearing in isomorphism theorems.
}
\label{Fig concept high D}
\end{figure}

\begin{rem}
The picture on Figure \ref{Fig concept high D} is somewhat reminiscent of Kenyon's cycle-rooted spanning forests (CRSF) \cite{Kenyon2011CRSF}.
We expect that in high enough dimension,
the cycles in a CRSF also scale to a Poisson point process of Brownian loops,
but with a different intensity measure.
If the unitary gauge field (values in $\mathbb{S}^{1}$) inducing the CRSF in discrete has
a continuum fine mesh limit $\mathcal{U}$, then the intensity measure for the limit Poisson point process should be
$(1-\operatorname{Re}\hol^{\mathcal{U}}(\wp))\mu^{\rm loop}_{\R^{d}}(d\wp)$.
\end{rem}

\section*{Acknowledgments}
\label{Sec ackno}

The author thanks Marcin Lis (TU Wien, Vienna) for pointing out the topological probabilities for the FK-Ising model (Section \ref{Subsec disord Ising}).
The author thanks Juhan Aru (EPFL), Adrien Kassel (ENS Lyon) and Thierry Lévy (Sorbonne Université) for inspiring discussions.

\bibliographystyle{alpha}
\bibliography{titusbib}

\end{document}